\documentclass[final,a4paper,reqno]{siamart1116}
\usepackage{amsmath,amssymb,epsfig,arydshln,float,amsfonts,latexsym,color,algorithm,fancyhdr}

\definecolor{lightblue}{rgb}{0.22,0.45,0.70}

\usepackage[T1]{fontenc}
\renewtheorem{lemma}{Lemma}[section]
\renewtheorem{theorem}{Theorem}[section]
\renewtheorem{proposition}{Proposition}[section]
\renewtheorem{corollary}{Corollary}

\newcommand\curl{\mathop{\mathbf{curl}}\nolimits}

\newcommand{\vdiv}{\operatorname*{div}}

\newcommand{\bu}{\boldsymbol{u}}
\newcommand{\bv}{\boldsymbol{v}}

\newcommand{\vtheta}{\vec{\boldsymbol{\theta}}}
\newcommand{\vomega}{\vec{\boldsymbol{\omega}}}

\newcommand\bx{\boldsymbol{x}}

\newcommand\bL{\mathbf{L}}
\newcommand\bH{\mathbf{H}}
\newcommand\bW{\mathbf{W}}
\newcommand\bzeta{\boldsymbol{\zeta}}

\newcommand\bomega{\boldsymbol{\omega}}
\newcommand\btheta{\boldsymbol{\theta}}
\newcommand\bnabla{\boldsymbol{\nabla}}
\newcommand\ff{\boldsymbol{f}}
\renewcommand\gg{\boldsymbol{g}}
\newcommand\nn{\boldsymbol{n}}

\newcommand\cero{\boldsymbol{0}}

\newcommand{\bX}{\mathbf{X}}
\newcommand{\bV}{\mathbf{V}}

\newcommand{\rQ}{\mathrm{Q}}
\newcommand{\rZ}{\mathrm{Z}}
\newcommand{\rL}{\mathrm{L}}
\newcommand{\rH}{\mathrm{H}}
\newcommand{\RR}{\mathbb{R}}

\newcommand\cM{\mathcal{M}}

\newcommand\bbP{\mathbb{P}}

\newcommand\cT{\mathcal{T}}
\newcommand{\estE}{\Theta_{K}}
\newcommand{\estP}{\Psi_{K}}
\newcommand{\estint}{\Lambda_{e}}
\newcommand{\UpsilonE}{\widetilde{\Upsilon}_{\!K}}
\newcommand{\UpsilonP}{\widehat{\Upsilon}_{\!K}}
\newcommand{\bfitn}{\nn}
\numberwithin{equation}{section}
\numberwithin{table}{section}
\numberwithin{figure}{section}
\allowdisplaybreaks

\DeclareFontEncoding{FMS}{}{}
\DeclareFontSubstitution{FMS}{futm}{m}{n}
\DeclareFontEncoding{FMX}{}{}
\DeclareFontSubstitution{FMX}{futm}{m}{n}
\DeclareSymbolFont{fouriersymbols}{FMS}{futm}{m}{n}
\DeclareSymbolFont{fourierlargesymbols}{FMX}{futm}{m}{n}
\DeclareMathDelimiter{\VERT}{\mathord}{fouriersymbols}{152}{fourierlargesymbols}{147}

\headers{\textit{A posteriori} analysis for rotation  elasticity/poroelasticity}{V. Anaya, A. Khan, D. Mora, R. Ruiz-Baier}

\title{Robust  a posteriori error analysis for rotation-based \\ formulations of the elasticity/poroelasticity coupling\thanks{\textbf{Updated:} \today.\funding{
This work has been partially supported by DIUBB through projects 2020127 IF/R and 194608 GI/C; 
by ANID-Chile through projects FONDECYT 1211265 and
Centro de Modelamiento Matem\'atico (AFB170001) of the PIA Program:
Concurso Apoyo a Centros Cient\'ificos y Tecnol\'ogicos de Excelencia
con Financiamiento Basal; by the Sponsored Research \& Industrial Consultancy (SRIC), Indian Institute of Technology Roorkee, India through the faculty initiation
grant MTD/FIG/100878; by SERB MATRICS grant
MTR/2020/000303.; by the Monash Mathematics Research Fund S05802-3951284; and by  the Ministry of Science and Higher Education of the Russian Federation within the framework of state support for the creation and development of World-Class Research Centers ``Digital biodesign and personalised healthcare" No. 075-15-2020-926.}}}

\author{Ver\'onica Anaya \thanks{GIMNAP, Departamento de Matem\'atica,
Universidad del B\'io-B\'io, Concepci\'on, Chile
and CI$^2$MA, Universidad de Concepci\'on, Concepci\'on, Chile (\email{vanaya@ubiobio.cl}).} \and 
Arbaz Khan\thanks{Department of Mathematics, Indian Institute of Technology Roorkee, Roorkee 247667, India (\email{arbaz@ma.iitr.ac.in}).} \and 
David Mora\thanks{GIMNAP, Departamento de Matem\'atica,
Universidad del B\'io-B\'io, Concepci\'on, Chile
and CI$^2$MA, Universidad de Concepci\'on, Concepci\'on, Chile  (\email{dmora@ubiobio.cl}).} \and 
Ricardo Ruiz-Baier\thanks{School of Mathematical Sciences, Monash University, 9 Rainforest Walk, Melbourne 3800 VIC, Australia; and 
  Institute of Computer Science and Mathematical Modelling, Sechenov University, Moscow, Russian Federation; and Universidad Adventista de Chile, Casilla 7-D, Chill\'an, Chile (\email{ricardo.ruizbaier@monash.edu}).}}
 
\begin{document}\maketitle
\date{\today}
\begin{abstract}
We develop the \textit{a posteriori} error analysis of three mixed finite element formulations for rotation-based equations in elasticity, poroelasticity, and interfacial elasticity-poroelasticity. The discretisations use $H^1$-conforming finite elements of degree $k+1$ for displacement and fluid pressure, and discontinuous piecewise polynomials of degree $k$ for rotation vector, total pressure, and elastic pressure. Residual-based estimators are constructed, and upper and lower bounds (up to data oscillations) for all global estimators are rigorously derived. The methods are all robust with respect to the model parameters (in particular, the Lam\'e constants), they are valid in 2D and 3D, and also for arbitrary polynomial degree $k\geq 0$. The error behaviour predicted by the theoretical analysis is then demonstrated numerically on a set of computational examples including different geometries on which we perform adaptive mesh refinement guided by the \textit{a posteriori} error estimators.  
\end{abstract}


\section{Introduction}
The interaction between interstitial fluid flow and the deformation of the underlying porous structure gives rise to a variety of mechanisms of fluid-structure coupling. In the specific case of Biot poromechanics, this interaction occurs when the linearly elastic porous medium is saturated, and such problem is relevant to a large class of very diverse applications ranging from bone healing to, e.g., petroleum engineering, or sound isolation. We are also interested in the interface between elastic and poroelastic systems that are encountered in hydrocarbon production in deep subsurface reservoirs (a pay zone and the surrounding non-pay rock formation) \cite{girault_cmame20}, or in the study of tooth and periodontal ligament interactions \cite{anaya_sisc20}. 

Rotation-based formulations are found in applications to the modelling of non-polar media and helicoidal motion (see, e.g., \cite{alturi_ar95,ibra_rev95,merlini_ijss05} and the references therein). The resulting theory has a similarity with vorticity-based formulations for incompressible flow such as  \cite{anaya_crm19,bernardi_sinum06,dubois_cmame03,ern_cras96,glowinski_siam79}.

The schemes for elasticity and transmission elasticity-poroelasticity and their \emph{a priori} error analysis have  been studied in \cite{anaya_cmame19} and \cite{anaya_sisc20}, respectively. The solvability of the rotation-based poroelasticity has not been addressed yet, and for sake of completeness we outline its analysis in Appendix~\ref{sec:poroelast-wellp} and Appendix~\ref{sec:poroelast-FE}. The well-posedness of the continuous problem is studied by grouping the unknowns with compatible regularity and realising that the resulting problem is a mixed variational formulation that resembles the system introduced in \cite{lee_sisc17,oyarzua_sinum16} that describes the Biot equations in their displacement-pressure-total pressure formulation. Our analysis also discusses the limit case when the specific storage coefficient goes to zero, and we observe that the continuous dependence on data is robust with respect to the Lam\'e constants.

Our focus is on the design, analysis, and testing of \textit{a posteriori} error estimators for these three rotation-based models and discretisations.  
Robust \textit{a posteriori} error estimators for Biot poroelasticity include the weakly symmetric tensor reconstruction for total stress and Darcy flux from \cite{bertrand_camwa21,girault_cmame20}, two fully mixed methods from \cite{ahmed_cmame19} (requiring the solution of auxiliary local problems), the guaranteed equilibrated bounds for fixed-stress splitting scheme from \cite{kumar_camwa21} and for double-diffusive poroelasticity from \cite{nordbotten_cmam10}, the robust residual \textit{a posteriori} estimates for displacement-flux-pressure advanced in \cite{li_ima20}, and for displacement-elastic pressure-fluid pressure from \cite{khan_ima20}. We follow the latter approach and construct residual-type error estimators. All the terms that conform the \textit{a posteriori} error estimators are easily fully computable locally. The derivation of the upper bounds for each of the terms conforming the \textit{a posteriori} estimators for rotation-based elasticity and rotation-based poroelasticity, is based on exploiting scaling arguments and bubble function techniques. The results obtained for these two sub-problems are then combined with estimates for the additional terms that appear in the transmission problem. As mentioned above, in all cases 
a careful treatment of the model parameters is essential to maintain robustness with respect to the sensible Lam\'e constants of the elastic and poroelastic media (going to infinity when the Poisson ratio approaches 1/2).

The remainder of the manuscript has been structured in the following manner. Instead of grouping the continuous results together and the error bounds separately for all problems, we have divided the analysis by type of problem. Therefore, Section~\ref{sec:elast} defines the rotation-based elasticity problem, recalls the solvability and stability of the continuous problem and of the mixed finite element discretisation, and provides the construction and analysis of an \textit{a posteriori} error estimator. An analogous presentation is given in Section~\ref{sec:poroelast} for the rotation-based Biot equations. These results are then combined in Section~\ref{sec:interface} to treat the rotation-based transmission problem between a poroelastic and an elastic sub-domain.  A few examples are presented in Section~\ref{sec:results}, showing in particular that mesh adaptivity steered by the \textit{a posteriori} error estimators leads to an important reduction in the number of degrees of freedom that are needed to reach a certain accuracy level, and the tests also indicate the sharpness of the \textit{a posteriori} error analysis. We also illustrate the use of the adaptive method in the simulation of a 3D aquifer interface problem.
Finally, in an appendix, we present the a priori error analysis of the rotation-based poroelasticity problem.

\section{Rotation-based linear elasticity}\label{sec:elast}
This section is devoted to deriving reliability and efficiency of a residual \textit{a posteriori} error estimator for a formulation of linear elasticity in terms of rotation, displacement, and pressure. We start with preliminary results regarding the continuous and discrete formulations. 
\subsection{Continuous formulation}
Let $\Omega\subset\mathbb{R}^d$, $d \in \{2,3\}$, be a bound\-ed Lipschitz domain with boundary $\Gamma:=\partial\Omega$.
Our starting point is the rotation-based elasticity problem, as proposed in \cite{anaya_cmame19}:
Given an external force $\ff^{\mathrm{E}}$, we seek the displacement $\bu$, the rotation $\bomega$ and the pressure $p$ such that
\begin{subequations}\label{rbepeq}
\begin{align}
\sqrt{\mu^{\mathrm{E}}}\curl \bomega+\nabla p&=\ff^{\mathrm{E}} & \text{in $\Omega$,}  \\
\bomega-\sqrt{\mu^{\mathrm{E}}}\curl \bu &=0& \text{in $\Omega$,} \\
\mathrm{div}\bu+(2\mu^{\mathrm{E}}+\lambda^{\mathrm{E}})^{-1}p &=0 & \text{in $\Omega$,} \\
\bu&=\cero & \text{on $\Gamma$},
\end{align}
\end{subequations}
where $\mu^{\mathrm{E}}$ and $\lambda^{\mathrm{E}}$ are the Lam\'e coefficients (material properties of the solid, and here assumed constant).
The weak formulation of (\ref{rbepeq}) is as follows: find $(\bu,\bomega,p)\in \bH^1_0(\Omega)\times\bL^2(\Omega)\times \rL^2(\Omega)$ such that 
\begin{subequations}\label{wfrbe}
\begin{align}
-\sqrt{\mu^{\mathrm{E}}}\int_\Omega \curl\bv\cdot\bomega +\int_{\Omega} p \vdiv\bv&=-\int_{\Omega}\ff^{\mathrm{E}} \bv & \forall \bv\in\bH^1_0(\Omega),\\
\int_{\Omega}\bomega \cdot\btheta -\sqrt{\mu^{\mathrm{E}}}\int_{\Omega}\btheta\cdot\curl \bu &=0 &\forall\btheta\in\bL^2(\Omega),\\
\int_{\Omega}\vdiv\bu q+(2\mu^{\mathrm{E}}+\lambda^{\mathrm{E}})^{-1}\int_{\Omega}p q &=0 & \forall q\in \rL^2(\Omega),
\end{align}
\end{subequations}
or more conveniently written in the form 
\[
B_{\mathrm{E}}((\bu,\bomega,p),(\bv,\btheta,q))=-(\ff^{\mathrm{E}},\bv)_{0,\Omega},
\]
where the multilinear form (having a subscript E, for \emph{elasticity}) is 
\begin{align*}
B_{\mathrm{E}}((\bu,\bomega,p),(\bv,\btheta,q)) &:= -\sqrt{\mu^{\mathrm{E}}}\int_\Omega \curl\bv\cdot\bomega
+\int_{\Omega} p \vdiv\bv+\int_{\Omega}\bomega\cdot \btheta \\
&\quad -\sqrt{\mu^{\mathrm{E}}}\int_{\Omega}\btheta\cdot\curl \bu+\int_{\Omega}\vdiv\bu q+(2\mu^{\mathrm{E}}+\lambda^{\mathrm{E}})^{-1}\!\int_{\Omega}p q. 
\end{align*}

For the considered boundary conditions, the term $\vdiv \bH^1_0$ can control only the $L^2$ norm of the mean-value zero part of $p$, and an additional contribution is needed to control the mean-value part of $p$ (see, e.g., \cite{lee_sisc17}). Thus we can decompose  $p$ into $P_mp$ and $p_0=p-P_mp$, where $P_m p$ is the mean value part and $p_0$ is the mean value zero part. This is required only in the incompressibility limit, as Herrmann's problem approaches Stokes equations and pressure (for $\bu$ prescribed everywhere on the boundary) is no longer unique.  This will be relevant also in the case of rotation-based Biot equations in Section~\ref{sec:poroelast}, below. 



The well-posedness of the above variational problem is a direct consequence of the following result (see \cite{gatica_book14}).

\begin{theorem}\label{WPelasticity}
For every $(\bu,\bomega,p)\in\bH^1_0(\Omega)\times\bL^2(\Omega)\times \rL^2(\Omega)$,
there exists $(\bv,\btheta,q)\in\bH^1_0(\Omega)\times\bL^2(\Omega)\times \rL^2(\Omega)$
with $\VERT(\bv,\btheta,q)\VERT\le C_1\VERT(\bu,\bomega,p)\VERT$ such that
\[
B_{\mathrm{E}}((\bu,\bomega,p),(\bv,\btheta,q))\ge C_2 \VERT(\bu,\bomega,p)\VERT^2,
\]
where $\VERT(\bv,\btheta,q)\VERT^2:= \mu^{\mathrm{E}}\|\curl \bv\|_{0,\Omega}^2+\mu^{\mathrm{E}}\|\vdiv\bv\|_{0,\Omega}^2
+\|\btheta\|_{0,\Omega}^2+(2\mu^{\mathrm{E}}+\lambda^{\mathrm{E}})^{-1}\|q\|_{0,\Omega}^2+(\mu^{\mathrm{E}})^{-1}\|q_0\|_{0,\Omega}^2$.
\end{theorem}
\begin{proof}
Consider the decomposition $p=p_0+P_mp$. As 
 a consequence of the inf-sup condition, 
for every $p_0\in L^2_0(\Omega)$ there exists 
$\bv_0\in\bH^1_0(\Omega)$ such that $(p_0,\vdiv\bv_0)_{0,\Omega}\ge C_{\Omega}(\mu^{\mathrm{E}})^{-1}\|p_0\|^2_{0,\Omega}$
and $(\mu^{\mathrm{E}})^{1/2}\Vert\bv_0\Vert_{1,\Omega}\le (\mu^{\mathrm{E}})^{-1/2}\Vert p_0\Vert_{0,\Omega}$.
Thus, we have
\begin{align*}
B_{\mathrm{E}}((\bu,\bomega,p),(\bv_0,\cero,0))&\ge \frac{C_{\Omega}}{\mu^{\mathrm{E}}}\|p_0\|_{0,\Omega}^2-\sqrt{\mu}(\bomega,\curl\bv_0)_{0,\Omega}
\ge\left(C_{\Omega}-\frac{1}{2\epsilon}\right)\frac{1}{\mu^{\mathrm{E}}}\|p_0\|_{0,\Omega}^2-\frac{\epsilon}{2}\|\bomega\|_{0,\Omega}^2.
\end{align*}
Choosing $\bv=-\bu$, $\btheta=\bomega$ and $q=p$, we arrive at
\[
B_{\mathrm{E}}((\bu,\bomega,p),(-\bu,\bomega,p))= \|\bomega\|_{0,\Omega}^2+(2\mu^{\mathrm{E}}+\lambda^{\mathrm{E}})^{-1}\|p\|_{0,\Omega}^2.
\] 
Next, we can select $\bv=\cero$, $\btheta=-\sqrt{\mu^{\mathrm{E}}} \curl\bu $ and $q=\mu^{\mathrm{E}}\vdiv\bu$, which leads to  
\begin{align*}
& B_{\mathrm{E}}((\bu,\bomega,p),(\cero,-\sqrt{\mu^{\mathrm{E}}}\curl\bu,\mu^{\mathrm{E}}\vdiv\bu))
\ge \frac{\mu^{\mathrm{E}}}{2}\|\curl\bu\|_{0,\Omega}^2+\frac{\mu^{\mathrm{E}}}{2}\|\vdiv\bu\|_{0,\Omega}^2
- \frac{1}{2}\|\bomega\|_{0,\Omega}^2- \frac{\mu^{\mathrm{E}}}{2(2\mu^{\mathrm{E}}+\lambda^{\mathrm{E}})^2}\|p\|_{0,\Omega}^2,
\end{align*}
We can also take $\bv=-\bu+\delta_1\bv_0$, $\btheta=\bomega-\delta_2\sqrt{\mu^{\mathrm{E}}}\curl\bu$, together with $q=p+\delta_2\mu^{\mathrm{E}}\vdiv\bu$, giving 
\begin{align*}
B_{\mathrm{E}}((\bu,\bomega,p),(\bv,\btheta,q))&=B_{\mathrm{E}}((\bu,\bomega,p),(-\bu,\cero,0))+\delta_1B_{\mathrm{E}}((\bu,\bomega,p),(\bv_0,\cero,0))
-\delta_2B_{\mathrm{E}}((\bu,\bomega,p),(\cero,\sqrt{\mu^{\mathrm{E}}}\curl\bu,\mu^{\mathrm{E}}\vdiv\bu)) \\
&\ge\left(1-\frac{\delta_1\epsilon}{2}-\frac{\delta_2}{2}\right)\|\bomega\|_{0,\Omega}^2
+\delta_2\frac{\mu^{\mathrm{E}}}{2}\|\curl\bu\|_{0,\Omega}^2+\delta_2\frac{\mu^{\mathrm{E}}}{2}\|\vdiv\bu\|_{0,\Omega}^2 \\
&\quad+\delta_1\left(C_{\Omega}-\frac{1}{2\epsilon}\right)\frac{1}{\mu^{\mathrm{E}}}\|p_0\|_{0,\Omega}^2
+\frac{1}{2\mu^{\mathrm{E}}+\lambda^{\mathrm{E}}}\left(1-\frac{\delta_2\mu^{\mathrm{E}}}{2\mu^{\mathrm{E}}+\lambda^{\mathrm{E}}}\right)\|p\|_{0,\Omega}^2.
\end{align*}
Choosing $\epsilon=1/C_{\Omega}$, $\delta_1=1/2\epsilon$ and $\delta_2=1/2$, we have
\[
B_{\mathrm{E}}((\bu,\bomega,p),(\bv,\btheta,q))\ge\min\left\{\frac{C_{\Omega}^2}{2},\frac{1}{4}\right\} \VERT(\bu,\bomega,p)\VERT^2,
\]
and the assertion of the theorem can be established by obtaining   
\[
\VERT(\bv,\btheta,q)\VERT^2= \VERT(-\bu+\delta_1\bv_0,\bomega-\delta_2\sqrt{\mu^{\mathrm{E}}}\curl\bu,p+\delta_2\mu^{\mathrm{E}}\vdiv\bu)\VERT^2\le 2 \VERT(\bu,\bomega,p)\VERT^2.
\]

\end{proof}

\subsection{Discrete spaces and Galerkin formulation}\label{sectdiscrete1}
Let $\{\cT_{h}\}_{h>0}$ be a shape-regular
family of partitions of the closed domain 
$\bar\Omega$, conformed by tetrahedra (or triangles 
in 2D) $K$ of diameter $h_K$, with mesh size
$h:=\max\{h_K:\; K\in\cT_{h}\}$.
We specify for any $k\ge0$ the finite-dimensional subspaces of the functional spaces for 
displacement, pressure and rotation; as follows 
\begin{align*}	
&\bV_h:=\{\bv_h \in \mathbf{C}(\overline{\Omega})\cap \bH^1_0(\Omega): \bv_h|_{K} \in \bbP_{k+1}(K)^d,\ \forall K\in \cT_h\},\nonumber\\
&\bW_h:=\{\btheta_h\in\mathbf{L}^2(\Omega): \btheta_h|_T\in\bbP_k(K)^d,\ \forall K\in\cT_{h}\},
\quad \rZ_h:=\{q_h\in\mathrm{L}^2(\Omega): q_h|_T\in\bbP_k(K),\ \forall K\in\cT_{h}\}.
\end{align*}
The discrete weak formulation reads: find $(\bu_h,\bomega_h,p_h)\in \bV_h\times\bW_h\times\rZ_h$ such that
 \begin{equation}\label{weakdisE}
 B_{\mathrm{E}}((\bu_h,\bomega_h,p_h),(\bv,\btheta,q))=-(\ff^{\mathrm{E}},\bv)_{0,\Omega} \qquad \forall (\bv,\btheta,q)\in \bV_h\times\bW_h\times\rZ_h.
 \end{equation}
 
In view of the comment above regarding pressure uniqueness in the nearly incompressible limit,
 we can either add a real Lagrange multiplier to fix the mean value of pressure, or
 (for the specific case of discontinuous pressures), simply add a jump stabilisation
 (from, e.g., \cite{hughes1987new}). Then, for  $k\ge0$, the modified  discrete weak
 formulation of the rotation based elasticity reads: find $(\bu_h,\bomega_h,q_h)\in \bV_h\times\bW_h\times\rZ_h$ such that
 \begin{align}\label{weakdisEstab}
 B_{\mathrm{E}}((\bu_h,\bomega_h,p_h),(\bv,\btheta,q))+\mu^{-1}\sum_{e\in\mathcal{E}(\cT_{h})}h_e\int_{e}[\![p_h]\!][\![q]\!]
 =-(\ff^{\mathrm{E}},\bv)_{0,\Omega}\qquad \forall (\bv,\btheta,q)\in \bV_h\times\bW_h\times\rZ_h,
\end{align}
where $h_e$ stands for the diameter of a given edge, $[\![\cdot]\!]$ the edge jump,
$\mu>0$ is an stabilisation parameter and $\mathcal{E}(\cT_{h})$
denotes the set of all edges in $\cT_h$.

By repeating the arguments in Theorem~\ref{WPelasticity}, we have
that \eqref{weakdisE} and \eqref{weakdisEstab} are well-posed.
In addition, by using standard arguments, it is possible to establish
the corresponding C\'ea's estimate and the \emph{a priori} estimates.

\subsection{\textit{A posteriori} error analysis}
%
First we define the local elastic error estimator $\estE$
and the elastic data oscillation $\UpsilonE$ for each $K\in\mathcal{T}_h$ as
\begin{align*}
\estE^2:= \frac{h_K^2}{\mu^{\mathrm{E}}}\|\mathbf{R}_1
\|_{0,K}^2+\sum_{e\in\partial K}\frac{h_e}{\mu^{\mathrm{E}}} \|\mathbf{R}_e
\|_{0,e}^2+\|\mathbf{R}_2\|_{0,K}^2+{\frac{1}{\frac{1}{\mu^{\mathrm{E}}}+\frac{1}{2\mu^{\mathrm{E}}+\lambda^{\mathrm{E}}}}}\|R_3\|_{0,K}^2, \quad 
\UpsilonE^2= \frac{h_K^2}{\mu^{\mathrm{E}}}\|\ff^{\mathrm{E}}-\ff_h^{\mathrm{E}}\|_{0,K}^2,
\end{align*}
where $\ff^{\mathrm{E}}_h\in\mathbf{L}^2(\Omega)$ is a piecewise polynomial approximation of $\ff^{\mathrm{E}}$. Moreover, 
the element-wise residuals are  
\[
\mathbf{R}_1:= \{\ff_h^{\mathrm{E}}-\sqrt{\mu^{\mathrm{E}}}\curl\bomega_h -\nabla p_h\}_K,\quad 
\mathbf{R}_2:= \{\bomega_h-\sqrt{\mu^{\mathrm{E}}}\curl\bu_h\}_K, \quad 
R_3:= \{\vdiv{\bu_h}+(2\mu^{\mathrm{E}}+\lambda^{\mathrm{E}})^{-1}p_h\}_K,
\]
and the edge residual is defined as
\[
\mathbf{R}_e:=\begin{cases}
\frac{1}{2}[\![\sqrt{\mu^{\mathrm{E}}}\bomega_h\times\nn+p_h\nn]\!]_e & e\in \mathcal{E}(\mathcal{T}_h)\setminus\Gamma\\
0 & e\in \Gamma.
\end{cases}
\]
Finally,  the global elastic residual error estimator $\Theta$ and the global elastic data oscillation term as 
\begin{equation}\label{estosc}
\Theta^2 :=\sum_{K\in\mathcal{T}_h} \estE^2,\quad\widetilde{\Upsilon}^2 :=\sum_{K\in\mathcal{T}_h}\UpsilonE^2.
\end{equation}

\subsubsection{Reliability estimate}
%
Using the Cl\'{e}ment interpolation estimate, the following results hold:
\begin{equation}\label{clement}
h_K^{-1}\sqrt{\mu^{\mathrm{E}}}\|\bv-I_h(\bv)\|_{0,K}\lesssim \sqrt{\mu^{\mathrm{E}}}|\bv|_{1,\bomega_K},\quad h_e^{-1/2} \sqrt{\mu^{\mathrm{E}}}\|\bv-I_h(\bv)\|_{0,e}\lesssim \sqrt{\mu^{\mathrm{E}}}|\bv|_{1,\bomega_K}.
\end{equation}
In next theorem, we discuss the reliability bound of the estimator $\Theta$. The Cl\'{e}ment interpolation estimate and the stability estimate are the main ingredients in the proof.
\begin{theorem}[Reliability estimate for the elasticity problem]
Let $(\bu,\bomega,p)$ be the solution to  (\ref{wfrbe}) and $(\bu_h,\bomega_h,p_h)$
the solution to (\ref{weakdisE}) (or (\ref{weakdisEstab})). Let $\Theta,\widetilde{\Upsilon}$ be as  in (\ref{estosc}). Then 
\begin{equation}\label{aux-rel0}
\VERT(\bu-\bu_h,\bomega-\bomega_h,p-p_h)\VERT\le C_{\mathrm{rel}} (\Theta+\widetilde{\Upsilon}).
\end{equation}
\end{theorem}
\begin{proof}
Since $(\bu-\bu_h,\bomega-\bomega_h,p-p_h)\in {\bH}^1_0(\Omega)\times {\bL}^2(\Omega)\times \rL^2(\Omega)$,
the stability Theorem~\ref{WPelasticity} implies 
\[
C_1\VERT(\bu-\bu_h,\bomega-\bomega_h,p-p_h)\VERT^2\le B_{\mathrm{E}}((\bu-\bu_h,\bomega-\bomega_h,p-p_h),(\bv,\btheta,q)),
\]  
with $\VERT(\bv,\btheta,q)\VERT\le C_2 \VERT(\bu-\bu_h,\bomega-\bomega_h,p-p_h)\VERT$. 
Using the definition of the weak forms, it follows:
\begin{align*}
B_{\mathrm{E}}((\bu-\bu_h,&\bomega-\bomega_h,p-p_h),(\bv,\btheta,q))=B_{\mathrm{E}}((\bu-\bu_h,\bomega-\bomega_h,p-p_h),(\bv-\bv_h,\btheta,q))\\
&=-(\ff^{\mathrm{E}}-\ff^{\mathrm{E}}_h,\bv-\bv_h)_{0,\Omega}-(\ff^{\mathrm{E}}_h,\bv-\bv_h)_{0,\Omega}-B_{\mathrm{E}}((\bu_h,\bomega_h,p_h),(\bv-\bv_h,\btheta,q)).
\end{align*}
Integration by parts, Cauchy-Schwarz inequality and the approximation results (\emph{cf.} \eqref{clement}), 
imply the bound 
\[
B_{\mathrm{E}}((\bu-\bu_h,\bomega-\bomega_h,p-p_h),(\bv,\btheta,q))\le C (\Theta+\widetilde{\Upsilon}) \VERT(\bv,\btheta,q)\VERT,
\]
which, in turn, implies \eqref{aux-rel0}. 
\end{proof}

\subsubsection{Efficiency bounds}
%
%
Let $K\in\mathcal{T}_h$ and consider  the interior polynomial bubble function $b_K$ (positive in the interior of $K$ and vanishing on $\partial K$). 
From  \cite{verfurth2013posteriori}, the following estimates hold:
\begin{equation}\label{ele1}
\|v\|_{0,K}\lesssim \|b_K^{1/2}v\|_{0,K}, \qquad 
\|b_Kv\|_{0,K}\lesssim \|v\|_{0,K}, \qquad 
\|\nabla(b_Kv)\|_{0,K}\lesssim h_K^{-1}\|v\|_{0,K},
\end{equation}
where $v$ is a scalar-valued polynomial function defined on $K$.

Each term defining $\Theta_K$ in terms of local errors are bounded using the 
following collection of results.  
%
\begin{lemma}\label{lemE1}
There holds:
\[
h_K^2(\mu^{\mathrm{E}})^{-1}\|\mathbf{R}_1\|_{0,K}^2\lesssim (\mu^{\mathrm{E}})^{-1/2}h_K\|\ff^{\mathrm{E}}-\ff_h^{\mathrm{E}}\|_{0,K}+(\mu^{\mathrm{E}})^{-1}\|p-p_h\|_{0,K}+\|\bomega-\bomega_h\|_{0,K}.
\]
\end{lemma}
\begin{proof}
For each $K\in \mathcal{T}_h$, we can  define 
%
$\bzeta|_K=(\mu^{\mathrm{E}})^{-1}h_K^2\mathbf{R}_1b_K$. We can then employ 
(\ref{ele1}) to arrive at 
\[
h_K^2(\mu^{\mathrm{E}})^{-1}\|\mathbf{R}_1\|_{0,K}^2\lesssim \int_K \mathbf{R}_1\cdot ((\mu^{\mathrm{E}})^{-1}h_K^2\mathbf{R}_1b_K)=\int_{K}\mathbf{R}_1\cdot \bzeta.
\]
Recall that $\ff^{\mathrm{E}}-\sqrt{\mu^{\mathrm{E}}}\curl\bomega -\nabla p=\cero$. We subtract this from the last term and then integrate using $\bzeta|_{\partial K}=\cero$
\[
h_K^2(\mu^{\mathrm{E}})^{-1}\|\mathbf{R}_1\|_{0,K}^2\lesssim \int_{K} (\ff_h^{\mathrm{E}}-\ff^{\mathrm{E}})\cdot \bzeta +\sqrt{\mu^{\mathrm{E}}}\int_K (\bomega-\bomega_h) \cdot \curl \bzeta +\int_K (p-p_h)\nabla\cdot \bzeta.
\]
Then, Cauchy-Schwarz inequality gives
{\small\begin{align*}
&h_K^2(\mu^{\mathrm{E}})^{-1}\|\mathbf{R}_1\|_{0,K}^2\lesssim (\mu^{\mathrm{E}})^{-1/2}h_K\|\ff^{\mathrm{E}}-\ff_h^{\mathrm{E}}\|_{0,K}+\|p-p_h\|_{0,K}+\|\bomega-\bomega_h\|_{0,K}(\mu^{\mathrm{E}})^{1/2}\|\bnabla \bzeta\|_{0,K}+(\mu^{\mathrm{E}})^{1/2}h_K^{-1}\|\bzeta\|_{0,K}.
\end{align*}}
And the proof can be completed thanks to the following estimate 
\begin{align*}
(\mu^{\mathrm{E}})^{1/2}\|\bnabla \bzeta\|_{0,K}+(\mu^{\mathrm{E}})^{1/2}h_K^{-1}\|\bzeta\|_{0,K}&\lesssim
(\mu^{\mathrm{E}})^{1/2}(\|\bnabla \bzeta\|_{0,K}+h_K^{-1}\|\bzeta\|_{0,K})\\
&\lesssim (\mu^{\mathrm{E}})^{1/2} h_K^{-1}\|\bzeta\|_{0,K} 
= h_K(\mu^{\mathrm{E}})^{-1/2}\|\mathbf{R}_1\|_{0,K}.
\end{align*}
\end{proof}
\begin{lemma}\label{lemE2}
There holds:
\begin{align*}
\|\mathbf{R}_2\|_{0,K}\lesssim \|\bomega-\bomega_h\|_{0,K}+\sqrt{\mu^{\mathrm{E}}}\|\curl(\bu-\bu_h)\|_{0,K}.
\end{align*}
\end{lemma}
\begin{proof}
The constitutive relation $\bomega-\sqrt{\mu^{\mathrm{E}}}\curl\bu=\cero$ implies that 
\begin{align*}
\|\mathbf{R}_2\|_{0,K}&=\|\bomega_h-\sqrt{\mu^{\mathrm{E}}}\curl\bu_h\|_{0,K} 
=\|(\bomega_h-\bomega)-\sqrt{\mu^{\mathrm{E}}}(\curl\bu_h-\curl\bu)\|_{0,K}\\
&\lesssim \|\bomega-\bomega_h\|_{0,K}+\sqrt{\mu^{\mathrm{E}}}\|\curl(\bu-\bu_h)\|_{0,K}.
\end{align*}
\end{proof}
\begin{lemma}\label{lemE3}
There holds:
\[
((\mu^{\mathrm{E}})^{-1}+(2\mu^{\mathrm{E}}+\lambda^{\mathrm{E}})^{-1})^{{-1/2}}\|R_3\|_{0,K}\lesssim  \sqrt{\mu^{\mathrm{E}}}\|\vdiv(\bu-\bu_h)\|_{0,K}+(2\mu^{\mathrm{E}}+\lambda^{\mathrm{E}})^{-1/2}\|{p}-{p}_h\|_{0,K}.
\]
\end{lemma}
\begin{proof}
Using the expression $\vdiv \bu+(2\mu^{\mathrm{E}}+\lambda^{\mathrm{E}})^{-1}p=0$, we have
\begin{align*}
((\mu^{\mathrm{E}})^{-1}+(2\mu^{\mathrm{E}}+\lambda^{\mathrm{E}})^{-1})^{{-1/2}}\|R_3\|_{0,K}&=((\mu^{\mathrm{E}})^{-1}+(2\mu^{\mathrm{E}}+\lambda^{\mathrm{E}})^{-1})^{{-1/2}}\|\vdiv \bu_h+(2\mu^{\mathrm{E}}+\lambda^{\mathrm{E}})^{-1}p_h\|_{0,K}\\
&\lesssim \sqrt{\mu^{\mathrm{E}}}\|\vdiv(\bu-\bu_h)\|_{0,K}+(2\mu^{\mathrm{E}}+\lambda^{\mathrm{E}})^{-1/2}\|{p}-{p}_h\|_{0,K}.
\end{align*}
\end{proof}
Let $e$ be an interior edge (or interior facet in 3D) shared by two elements $K$ and $K'$. We assume that $b_e$,
the edge polynomial bubble function on $e$, is positive in the interior of the
patch $P_e$ formed by $K\cup K'$, and $b_e$ is zero on the boundary of the patch.
Then, also from  \cite{verfurth2013posteriori}, the following estimates hold:
\begin{gather}\label{edgee1}
\|q\|_{0,e}\lesssim \|b_e^{1/2}q\|_{0,e},\quad 
\|b_eq\|_{0,K}\lesssim h_e^{1/2}\|q\|_{0,e}, \quad 
\|\nabla(b_eq)\|_{0,K}\lesssim h_e^{-1/2}\|q\|_{0,e}\qquad \forall K\in P_e,
\end{gather}
where $q$ denotes the scalar-valued polynomial function defined on the edge $e$.
\begin{lemma}\label{lemE4}
There holds:
\begin{align*}
(\sum_{e\in\partial K}h_e(\mu^{\mathrm{E}})^{-1} \|\mathbf{R}_e\|_{0,e}^2)^{1/2}\lesssim
\sum_{K\in P_e}((\mu^{\mathrm{E}})^{-1/2}h_K\|\ff^{\mathrm{E}}-\ff_h^{\mathrm{E}}\|_{0,K}+(\mu^{\mathrm{E}})^{-1/2}\|p-p_h\|_{0,K}+\|\bomega-\bomega_h\|_{0,K}). 
\end{align*}
\end{lemma}
\begin{proof}
For  $e\in \mathcal{E}(\mathcal{T}_h)$ we  define locally 
%
$\bzeta_e=(\mu^{\mathrm{E}})^{-1}h_e\mathbf{R}_eb_e$. 
Therefore, relation (\ref{edgee1}) implies 
\[
h_e(\mu^{\mathrm{E}})^{-1}\|\mathbf{R}_e\|_{0,e}^2\lesssim \int_e \mathbf{R}_e\cdot ((\mu^{\mathrm{E}})^{-1}h_e\mathbf{R}_eb_e)=\int_{e}\mathbf{R}_e\cdot \bzeta_e.
\]
Since $[\![\bomega\times\nn]\!]_e=\cero$ and $[\![p\nn]\!]_e=\cero$, we have
 \begin{align*}
 \int_e [\![\sqrt{\mu^{\mathrm{E}}}(\bomega_h-\bomega)\times\nn+(p_h-p)\nn]\!]_e \cdot \bzeta_e&=\sum_{K\in P_e}\int_K(\sqrt{\mu^{\mathrm{E}}}\curl(\bomega_h-\bomega)+\nabla (p_h-p))\cdot \bzeta_e\\
 &\quad+\sum_{K\in P_e}\int_K(\sqrt{\mu^{\mathrm{E}}}(\bomega_h-\bomega)\cdot  \curl \bzeta_e+ (p_h-p) \nabla\cdot  \bzeta_e),
 \end{align*}
 where we have used integration by parts element-wise.
Recalling that $\ff^{\mathrm{E}}-\sqrt{\mu^{\mathrm{E}}}\curl\bomega -\nabla p=\cero|_K$, gives 
\begin{align*}
\frac{h_e}{\mu^{\mathrm{E}}}\|\mathbf{R}_e\|_{0,e}^2&\lesssim \sum_{K\in P_e}\int_{K} \left((\ff_h^{\mathrm{E}}-\ff^{\mathrm{E}})\cdot \bzeta_e
+\sqrt{\mu^{\mathrm{E}}}\int_K (\bomega_h-\bomega) \cdot \curl \bzeta_e +\int_K (p_h-p)\nabla\cdot \bzeta \right)+ \sum_{K\in P_e}\int_{K} \mathbf{R}_1 \cdot \bzeta_e.
\end{align*}
From Cauchy-Schwarz inequality we can then infer that 
\begin{align*}
h_e(\mu^{\mathrm{E}})^{-1}\|\mathbf{R}_e\|_{0,e}^2\lesssim 
&\sum_{K\in P_e}((\mu^{\mathrm{E}})^{-1/2}h_K\|\ff^{\mathrm{E}}-\ff_h^{\mathrm{E}}\|_{0,K}+(\mu^{\mathrm{E}})^{-1/2}\|p-p_h\|_{0,K}+\|\bomega-\bomega_h\|_{0,K})\times\\
&\quad((\mu^{\mathrm{E}})^{1/2}\|\bnabla \bzeta_e\|_{0,K}+(\mu^{\mathrm{E}})^{1/2}h_K^{-1}\|\bzeta_e\|_{0,K}).
\end{align*}
And the assertion of the lemma is proven after obtaining the bound 
\[
(\mu^{\mathrm{E}})^{1/2}\|\bnabla \bzeta_e\|_{0,K}+(\mu^{\mathrm{E}})^{1/2}h_K^{-1}\|\bzeta_e\|_{0,K}
\lesssim (\mu^{\mathrm{E}})^{1/2} h_K^{-1}\|\bzeta_e\|_{0,K} 
= h_e^{1/2}(\mu^{\mathrm{E}})^{-1/2}\|\mathbf{R}_e\|_{0,e}.
\]
\end{proof}
Now, we are in position to state the  efficiency   of the proposed estimator $\Theta$. 
\begin{theorem}[Efficiency estimate for the elasticity problem]\label{elasteff}
Let $(\bu,\bomega,p)$ be the solution to \eqref{wfrbe} and $(\bu_h,\bomega_h,p_h)$
the solution to \eqref{weakdisE} (or \eqref{weakdisEstab}). Also, let $\Theta,\widetilde{\Upsilon}$ be as in \eqref{estosc}. Then:
\[
\Theta\le C_{\mathrm{eff}}(\VERT(\bu-\bu_h,\bomega-\bomega_h,p-p_h)\VERT +\widetilde{\Upsilon}).
\]
\end{theorem}
\begin{proof}
It suffices to combine Lemmas~\ref{lemE1}--\ref{lemE4}.
\end{proof}

\section{Rotation-based poroelasticity with total pressure}\label{sec:poroelast}
In this section we propose a mixed finite element method for the approximation of linear poroelasticity equations, formulated in terms of displacement, rotation vector, fluid pressure, and total pressure. Then, we will present an \textit{a posteriori} error analysis.
%
\subsection{Continuous formulation}
We consider the steady poroelasticity equations written in terms of displacement $\bu$, fluid pressure $p$,
rescaled total pressure $\phi := \alpha p -(2\mu^{\mathrm{P}}+\lambda^{\mathrm{P}}) \vdiv\bu$,  and 
rescaled rotation vector $\bomega:= \sqrt{\mu^{\mathrm{P}}}\curl \bu$, where  $\alpha>0$ is
the Biot-Willis parameter, and $\lambda^{\mathrm{P}},\mu^{\mathrm{P}}$ are the Lam\'e constants. Moreover,  
$s^{\mathrm{P}}$ is a smooth fluid source term, $\kappa$ is the
permeability (isotropic and
satisfying $0<\kappa_1\leq \kappa(\bx)\leq \kappa_2<\infty$, for all
$\bx\in\Omega$), 
$c_0>0$ is the storativity coefficient, $\gg$ is gravity, $\ff^{\mathrm{P}}$ is the external load, 
and $\xi,\rho$ are the viscosity and density of the pore fluid, respectively.
The system reads  
\begin{subequations}
\begin{align}
    \sqrt{\mu^{\mathrm{P}}}\curl\bomega+\nabla \phi &= \ff^{\mathrm{P}} & \text{in $\Omega$,}  \label{eqP:momentum}
    \\
    \bomega-\sqrt{\mu^{\mathrm{P}}}\curl \bu &= \mathbf{0} & \text{in $\Omega$,}  \label{eqP:omega}
    \\
   (2\mu^{\mathrm{P}}+\lambda^{\mathrm{P}})^{-1} \phi+\vdiv\bu  -\alpha(2\mu^{\mathrm{P}}+\lambda^{\mathrm{P}})^{-1}p &= 0 & \text{in $\Omega$,}   \label{eqP:phi}
    \\
    \big[c_0+\alpha^2(\mu^{\mathrm{P}}+\lambda^{\mathrm{P}})^{-1}\big]p - \alpha(2\mu^{\mathrm{P}}+\lambda^{\mathrm{P}})^{-1}\phi- \xi^{-1}\vdiv \big{[}\kappa(\nabla p - \rho \boldsymbol{g})\big{]} &= s^{\mathrm{P}} & \text{in $\Omega$},     \label{eqP:mass}
\end{align}\end{subequations}
%
and we assume  that the domain is clamped and consider zero filtration flux on the boundary
\begin{equation*}
	\bu   =\cero \quad \text{on} \quad \partial\Omega,\qquad 
	\kappa\xi^{-1} (\nabla p - \rho\gg)\cdot\nn  = 0
	\quad \text{on}\quad \partial\Omega.
\end{equation*}


Testing each equation of \eqref{eqP:momentum}-\eqref{eqP:mass},
integrating by parts whenever adequate (see \cite[Theorem 2.11]{girault_book86}) and applying 
the boundary conditions we obtain:
\begin{align}
\nonumber	-\sqrt{\mu^{\mathrm{P}}}\int_{\Omega}\curl\bv\cdot \bomega+\int_{\Omega}\phi \vdiv \bv & = -\int_{\Omega}\ff^{\mathrm{P}}\cdot \bv,  
	\\
\int_{\Omega}\bomega\cdot \btheta-\sqrt{\mu^{\mathrm{P}}}\int_{\Omega}\btheta \cdot\curl \bu &= 0,  \label{weak-poro}
	\\
	(2\mu^{\mathrm{P}}+\lambda^{\mathrm{P}})^{-1} \int_{\Omega}\phi\psi+\int_{\Omega}\psi\vdiv\bu  -\alpha(2\mu^{\mathrm{P}}+\lambda^{\mathrm{P}})^{-1}\int_{\Omega}p\psi &= 0, 	\nonumber \\
\nonumber	-\big[c_0+\frac{\alpha^2}{(2\mu^{\mathrm{P}}+\lambda^{\mathrm{P}})}\big]\int_{\Omega}pq + \frac{\alpha}{2\mu^{\mathrm{P}}+\lambda^{\mathrm{P}}}\int_{\Omega}\phi q-\int_{\Omega} \frac{\kappa}{\xi}\nabla p\cdot \nabla q 
&= -\frac{\rho}{\xi}\int_{\Omega}\kappa\gg\cdot \nabla q -\int_{\Omega} s^{\mathrm{P}}q , 
\end{align}
for each $(\bv,\btheta ,\psi ,q )\in \mathbf{H}^1_0(\Omega)\times \mathbf{L}^2(\Omega)\times\mathrm{L}^2(\Omega)\times \mathrm{H}^1(\Omega)$.

We can regard the rotation and the rescaled total pressure $\phi$ as a single unknown  $\vomega:=(\bomega,\phi)$. This gives the unsymmetric variational form: 
find $(\vomega,\bu,p )\in\mathbf{H}\times\bV\times\rQ$ 
such that
\begin{subequations}
\label{eq:poro-weak-2}
\begin{align}
	a(\vomega,\vtheta)+b_1(\vtheta,\bu)-b_2(\vtheta,p)&= \; 0 &\forall\,\vtheta\in\mathbf{H}, \label{weak-uP}\\
	b_1(\vomega,\bv) &= \;F(\bv) &\forall\, \bv \in\bV, \label{weak-pP}\\
	b_2(\vomega,q)  -   c(p ,q)  &=\;G(q ) &\forall \,q\in\rQ, \label{weak-phi}
\end{align}\end{subequations}
where $\vtheta:=(\btheta,\psi)$, $\mathbf{H}:=\bL^2(\Omega)\times \rL^2(\Omega)$, $\bV:=\bH_0^1(\Omega)$, $\rQ:=\rH^1(\Omega)$, 
and the bilinear forms $a:\mathbf{H}\times\mathbf{H}\to \RR$, 
$b_1:\mathbf{H}\times\bV\to \RR$, $b_2:\mathbf{H}\times\rQ\to \RR$, $b_3:\mathbf{H}\times\rQ\to\RR$, $c:\rQ\times\rQ\to \RR$, 
and linear functionals $F:\bV\to\RR$, $G:\rQ\to\RR$ 
are specified in the following way
\begin{align*}
	& a(\vomega,\vtheta) := \int_{\Omega}\bomega\cdot \btheta + \frac{1}{2\mu^{\mathrm{P}}+\lambda^{\mathrm{P}}}\int_{\Omega}\phi\psi,\quad 
	 b_1(\vtheta,\bv):=-\sqrt{\mu^{\mathrm{P}}}\int_{\Omega}\btheta\cdot\curl\bv +\int_{\Omega}\psi\vdiv\bv, \quad b_2(\vtheta,p ):=\frac{\alpha}{2\mu^{\mathrm{P}}+\lambda^{\mathrm{P}}}\int_{\Omega}p\psi,\\
	& c(p ,q ):=\left[c_0+\frac{\alpha^2}{2\mu^{\mathrm{P}}+\lambda^{\mathrm{P}}}\right]\int_{\Omega}pq+\frac{1}{\xi}\int_{\Omega} \kappa\nabla p\cdot \nabla q,\quad 
	 F(\bv) := -\int_{\Omega} \ff^{\mathrm{P}}\cdot\bv,\quad G(q) := -\frac{\rho}{\xi} \int_{\Omega}\kappa \gg\cdot \nabla q 
	-  \int_{\Omega} s^{\mathrm{P}}q.
\end{align*}
Note that the displacement space $\rH_0(\curl,\Omega)\cap \rH_0(\vdiv,\Omega)$ is 
algebraically and topologically equivalent to $\bV$ if $\Omega$ is a polyhedral
bounded domain with Lipschitz  boundary 
 \cite[Lemma 2.5, Remark 2.7]{girault_book86}. 

The formulation in \eqref{eq:poro-weak-2}
can be also written, more concisely, as
\[
B_{\mathrm{P}}((\bu,\bomega,\phi,p),(\bv,\btheta,\psi,q))=F(\bv) + G(q),
\]
where the multilinear form (now having a subscript P, for \emph{poroelasticity}), is defined as 
\[
B_{\mathrm{P}}((\bu,\bomega,\phi,p),(\bv,\btheta,\psi,q)) := a(\vomega,\vtheta) +b_1(\vtheta,\bu)-b_2(\vtheta,p)+ b_1(\vomega,\bv) + b_2(\vomega,q)  -   c(p ,q).\]


The following result will be useful in the next section. 
\begin{theorem}\label{stab-poro}
For every $(\bu,\bomega,\phi,p)\in\mathbf{H}^1_0(\Omega)\times \mathbf{L}^2(\Omega)\times\mathrm{L}^2(\Omega)\times \mathrm{H}^1(\Omega)$, there exists $(\bv,\btheta,\psi,q)\in\mathbf{H}^1_0(\Omega)\times \mathbf{L}^2(\Omega)\times\mathrm{L}^2(\Omega)\times \mathrm{H}^1(\Omega)$ with $\VERT(\bv,\btheta,\psi,q)\VERT\le C_1\VERT(\bu,\bomega,\phi,p)\VERT$ such that
\[
B_{\mathrm{P}}((\bu,\bomega,\phi,p),(\bv,\btheta,\psi,q))\ge C_2 \VERT(\bu,\bomega,\phi,p)\VERT^2,
\]
where 
\begin{align*}
\VERT(\bv,\btheta,\psi,q)\VERT^2 & :=\mu^{\mathrm{P}}(\|\curl \bv\|_{0,\Omega}^2+\|\vdiv \bv\|_{0,\Omega}^2)+\|\btheta\|_{0,\Omega}^2+\frac{1}{\mu^{\mathrm{P}}}\|\psi_0\|_{0,\Omega}^2+\frac{1}{2\mu^{\mathrm{P}}+\lambda^{\mathrm{P}}}\|\psi\|_{0,\Omega}^2 \\
& \qquad \qquad \qquad \qquad +\left(c_0+\frac{\alpha^2}{2\mu^{\mathrm{P}}+\lambda^{\mathrm{P}}}\right)\|q\|_{0,\Omega}^2+\|\frac{\kappa}{\xi} \nabla q\|_{0,\Omega}^2.
\end{align*}
\end{theorem}
\begin{proof}
Analogously as in the proof of Theorem~\ref{WPelasticity},
we have that there exists $\bv_0\in\mathbf{H}^1_0(\Omega)$ such that 
\[
B_{\mathrm{P}}((\bu,\bomega,\phi,p),(\bv_0,\cero,0,0))\ge \frac{C_{\Omega}}{\mu^{\mathrm{P}}}\|\phi_0\|_{0,\Omega}^2-\sqrt{\mu^{\mathrm{P}}}(\bomega,\curl\bv_0)\ge\left(C_{\Omega}-\frac{1}{2\epsilon}\right)\frac{1}{\mu^{\mathrm{P}}}\|\phi_0\|_{0,\Omega}^2-\frac{\epsilon}{2}\|\bomega\|_{0,\Omega}^2.
\]
First we take $\bv=-\bu$, $\btheta=\bomega$, $\psi=\phi$ and $q=-p$. Consequently, 
\begin{align*}
B_{\mathrm{P}}((\bu,\bomega,\phi,p),(-\bu,\bomega,\phi,-p))&= \|\bomega\|_{0,\Omega}^2+(2\mu^{\mathrm{P}}+\lambda^{\mathrm{P}})^{-1}\|\phi\|_{0,\Omega}^2-2\alpha(2\mu^{\mathrm{P}}+\lambda^{\mathrm{P}})^{-1}(p,\phi)_{0,\Omega}\nonumber\\
&\quad +(c_0+\alpha^2(2\mu^{\mathrm{P}}+\lambda^{\mathrm{P}})^{-1})\|p\|_{0,\Omega}^2+\|\kappa/\xi\nabla p\|_{0,\Omega}^2.
\end{align*} 
Next, we choose $\bv=\cero$, $\btheta=-\sqrt{\mu^{\mathrm{P}}} \curl\bu$, $\psi=\mu^{\mathrm{P}}\vdiv\bu$ and $q=0$, and therefore 
\begin{align*}
& B_{\mathrm{P}}((\bu,\bomega,\phi,p),(\cero,-\sqrt{\mu^{\mathrm{P}}}\curl\bu,\mu^{\mathrm{P}}\vdiv\bu,0))\\
&\quad =\mu^{\mathrm{P}}\|\curl\bu\|_{0,\Omega}^2+\mu^{\mathrm{P}}\|\vdiv\bu\|_{0,\Omega}^2 -\sqrt{\mu^{\mathrm{P}}}(\bomega,\curl\bu)+\mu^{\mathrm{P}}(2\mu^{\mathrm{P}}+\lambda^{\mathrm{P}})^{-1}(\phi, \vdiv\bu)-\alpha\mu^{\mathrm{P}}(2\mu^{\mathrm{P}}+\lambda^{\mathrm{P}})^{-1} (p,\vdiv\bu)\nonumber\\
&\quad \ge \frac{\mu^{\mathrm{P}}}{2}\|\curl\bu\|_{0,\Omega}^2+\left(1-\frac{\mu^{\mathrm{P}}}{(2\mu^{\mathrm{P}}+\lambda^{\mathrm{P}})}\right)\mu^{\mathrm{P}}\|\vdiv\bu\|_{0,\Omega}^2- \frac{1}{2}\|\bomega\|_{0,\Omega}^2- \frac{\alpha^2}{2(2\mu^{\mathrm{P}}+\lambda^{\mathrm{P}})}\|p\|_{0,\Omega}^2-\frac{1}{2(2\mu^{\mathrm{P}}+\lambda^{\mathrm{P}})}\|\phi\|_{0,\Omega}^2.
\end{align*}
Finally, we can take $\bv=-\bu+\delta_1 \bv_0$, $\btheta=\bomega-\delta_2\sqrt{\mu^{\mathrm{P}}}\curl\bu$, $\psi=\phi+\delta_2\mu^{\mathrm{P}} \vdiv\bu$ and $q=-p$, to obtain
\begin{align}
B_{\mathrm{P}}((\bu,\bomega,\phi,p),(\bv,\btheta,\psi,q))&=B_{\mathrm{P}}((\bu,\bomega,\phi,p),(-\bu,\bomega,\phi,-p))+\delta_1 B_{\mathrm{P}}((\bu,\bomega,\phi,p),(\bv_0,\cero,0,0))\nonumber\\
&+\delta_2 B_{\mathrm{P}}((\bu,\bomega,\phi,p),(\cero,-\sqrt{\mu^{\mathrm{P}}}\curl\bu,\mu^{\mathrm{P}}\vdiv\bu,0))\nonumber\\
&\ge \left(1-\frac{\delta_1\epsilon}{2}-\frac{\delta_2}{2}\right)\|\bomega\|_{0,\Omega}^2+\frac{\mu^{\mathrm{P}}\delta_2}{2}\|\curl\bu\|_{0,\Omega}^2+\delta_1\left(1-\frac{\mu^{\mathrm{P}}}{(2\mu^{\mathrm{P}}+\lambda^{\mathrm{P}})}\right)\|\vdiv\bu\|_{0,\Omega}^2\nonumber\\
&\quad+\left(C_{\Omega}-\frac{1}{2\epsilon}\right)\frac{\delta_1}{\mu^{\mathrm{P}}}\|\phi_0\|_{0,\Omega}^2+\frac{1}{2\mu^{\mathrm{P}}+\lambda^{\mathrm{P}}}\left(\frac{1}{2}-\frac{\delta_2}{2}\right)\|\phi\|_{0,\Omega}^2+\|\kappa/\xi \nabla p\|_{0,\Omega}^2\nonumber\\
&\quad+\left(c_0+\frac{\alpha^2}{(2\mu^{\mathrm{P}}+\lambda^{\mathrm{P}})}\left(\frac{1}{2}-\frac{\delta_2}{2}\right)\right)\|p\|_{0,\Omega}^2.\nonumber
\end{align}
Choosing $\epsilon =1/C_{\Omega}$, $\delta_1 = 1/2\epsilon$ and $\delta_2=1/2$, we have
\begin{align*}
B_{\mathrm{P}}((\bu,\bomega,\phi,p),(\bv,\btheta,\psi,q))\ge \min\left\{\frac{C_{\Omega}^2}{4},\frac{1}{4}\right\} \VERT(\bu,\bomega,\phi,p)\VERT^2.
\end{align*}
And from that, the following estimate completes the proof 
\[
\VERT(\bv,\btheta,\psi,q)\VERT^2= \VERT(-\bu+\delta_1 \bv_0,\bomega-\delta_2\sqrt{\mu^{\mathrm{P}}}\curl\bu,\phi+\delta_2\mu^{\mathrm{P}} \vdiv\bu,-p)\VERT^2 \le 2 \VERT(\bu,\bomega,\phi,p)\VERT^2.
\]
\end{proof}

\subsection{Discrete spaces and Galerkin formulation}

With the same notation as in Section~\ref{sectdiscrete1},
%
we specify 
finite-dimensional 
for 
displacement, fluid pressure,  rotations, and total 
 pressure; as follows 
\begin{align}	
&\bV_h:=\{\bv_h \in \mathbf{C}(\overline{\Omega})\cap\bV: \bv_h|_{K} \in \bbP_{k+1}(K)^d,\ \forall K\in \cT_h\},\quad 
\nonumber	
\rQ_h:=\{q_h\in\mathrm{C}(\overline{\Omega})\cap\rQ: q_h|_K\in\bbP_{k+1}(K),\ \forall K\in\cT_{h}\},\\
\label{fe-spaces-porous}	
&\bW_h:=\{\btheta_h\in\mathbf{L}^2(\Omega): \btheta_h|_K\in\bbP_k(K)^d,\ \forall K\in\cT_{h}\},\quad
\rZ_h:=\{\psi_h\in\mathrm{L}^2(\Omega): \psi_h|_K\in\bbP_k(K),\ \forall K\in\cT_{h}\}.
\end{align}
Then the discrete formulation consists in finding $(\bu_h,\bomega_h,\phi_h,p_h)\in \bV_h\times\bW_h\times\rZ_h\times\rQ_h$ such that
 \begin{align}\label{weakdisP}
 B_{\mathrm{P}}((\bu_h,\bomega_h,\phi_h,p_h),(\bv,\btheta,\psi,q))=F(\bv) + G(q),
 \end{align}
 for all $(\bv,\btheta, \psi, q)\in \bV_h\times\bW_h\times\rZ_h\times\rQ_h$. 
Likewise, for each $k\ge0$, the modified (stabilised) discrete weak formulation of the rotation based poroelasticity
is: find $(\bu_h,\bomega_h,\phi_h,p_h)\in \bV_h\times\bW_h\times\rZ_h\times\rQ_h$ such that
 \begin{equation}\label{weakdisPstab}
 B_{\mathrm{P}}((\bu_h,\bomega_h,\phi_h,p_h),(\bv,\btheta,\psi,q))+\mu^{-1}\sum_{e\in\mathcal{E}(\mathcal{T}_h)}h_e\int_{e}[\![\phi_h]\!][\![\psi]\!]=-(\ff^{\mathrm{P}},\bv)_{0,\Omega}-\frac{\rho}{\xi}(\kappa\gg, \nabla q)_{0,\Omega} -(s^{\mathrm{P}},q)_{0,\Omega},
 \end{equation}
 for all $(\bv,\btheta,\psi,q)\in \bV_h\times\bW_h\times\rZ_h\times\rQ_h$.

The analysis of the continuous and discrete formulations is not found
in \cite{anaya_cmame19,anaya_sisc20}. For sake of completeness, we outline it in the Appendix. 

\subsection{\textit{A posteriori} error analysis}
%
First, we define the poroelastic local error estimator $\estP$  as 
\[
\estP^2:= \frac{h_K^2}{\mu^{\mathrm{P}}}\|\mathbf{R}_1
\|_{0,K}^2+\sum_{e\in\partial K}\frac{h_e}{\mu^{\mathrm{P}}} \|\mathbf{R}_e
\|_{0,e}^2+\|\mathbf{R}_2\|_{0,K}^2+\rho_d\|R_3\|_{0,K}^2
+\rho_1\|R_4\|_{0,K}^2 +\sum_{e\in\partial K} \rho_2\|{R}_e\|_{0,e}^2,
\]
where the elemental residuals are defined as:
\begin{align*}
\mathbf{R}_1 &:= \{\ff_h^{\mathrm{P}}-\sqrt{\mu^{\mathrm{P}}}\curl\bomega_h -\nabla \phi_h\}_K,\quad 
\mathbf{R}_2:= \{\bomega_h-\sqrt{\mu^{\mathrm{P}}}\curl\bu_h\}_K,\\
R_3&:= \{\vdiv{\bu_h}+(2\mu^{P}+\lambda^{\mathrm{P}})^{-1}\phi_h-\alpha(2\mu+\lambda^{\mathrm{P}})^{-1}p_h\}_K,\\
R_4&:=\{s^{\mathrm{P}}_h-(c_0+\alpha^2(2\mu+\lambda)^{-1}p_h+\alpha(2\mu^{\mathrm{P}}+\lambda^{\mathrm{P}})^{-1}\phi_h+\xi^{-1}\vdiv[\kappa(\nabla p_h-\rho \textbf{g})]\}_K,
\end{align*}
and the edge residuals are defined as
\[
\mathbf{R}_e:=\begin{cases}
\frac{1}{2}[\![\sqrt{\mu^{\mathrm{P}}}\bomega_h\times\nn+\phi_h\nn]\!]_e & e\in \mathcal{E}(\mathcal{T}_h)\setminus\Gamma\\
0 & e\in \Gamma
\end{cases}, \quad 
{R}_e:=\begin{cases}
\frac{1}{2}[\![\xi^{-1}\kappa(\nabla p_h-\rho \textbf{g})]\!]_e & e\in \mathcal{E}(\mathcal{T}_h)\setminus\Gamma\\
\xi^{-1}\kappa(\nabla p_h-\rho \textbf{g}) & e\in \Gamma
\end{cases},
\]
%
%
with the scaling constants taken as 
\[
\rho_1:=\min\{(c_0+\alpha^2(2\mu^{\mathrm{P}}+\lambda^{\mathrm{P}})^{-1})^{-1},h_K^2 \xi \kappa^{-1}\},\quad \rho_2:= \xi\kappa^{-1}h_e, \quad \rho_d:= ((\mu^{\mathrm{P}})^{-1}+(2\mu^{\mathrm{P}}+\lambda^{\mathrm{P}})^{-1})^{-1}.
\]
On the other hand, the definition of the poroelastic oscillation term $\UpsilonP$ is as follows:
\[
\UpsilonP^2= h_K^2 (\mu^{\mathrm{P}})^{-1}\|\ff^{\mathrm{P}}-\ff^{\mathrm{P}}_h\|_{0,K}^2+\rho_1 \|s^{\mathrm{P}}-s^{\mathrm{P}}_h\|_{0,K}^2.
\]
Finally,  the global residual error estimator   and   data oscillation terms are, respectively,  
\begin{equation}\label{estosc1}
\Psi^2 :=\sum_{K\in\mathcal{T}_h} \estP^2,\qquad\widehat{\Upsilon}^2 :=\sum_{K\in\mathcal{T}_h}\UpsilonP^2.
\end{equation}

\subsubsection{Reliability} 
%
In this section, we establish reliability of \eqref{estosc1}. The main ingredients are the stability theorem and the interpolation estimate to establish the upper bound. 
\begin{theorem}[Reliability for the Biot problem]
Let $(\bu,\bomega,\phi,p)$ and $(\bu_h,\bomega_h,\phi_h,p_h)$ be the solutions of the weak
formulations (\ref{weak-poro}) and (\ref{weakdisP}) (or (\ref{weakdisPstab})), respectively. Then the following reliability bound holds
\[
\VERT(\bu-\bu_h,\bomega-\bomega_h,\phi-\phi_h,p-p_h)\VERT\le C_{\mathrm{rel}} (\Psi+\widehat{\Upsilon}),
\]
where $C_{\mathrm{rel}}>0$ is a positive constant independent of mesh size and parameters.
\end{theorem}
\begin{proof}
Since $(\bu-\bu_h,\bomega-\bomega_h,\phi-\phi_h,p-p_h)\in \mathbf{H}^1_0(\Omega)\times
\mathbf{L}^2(\Omega)\times\rL^2(\Omega)\times \mathrm{H}^1(\Omega)$, then   Theorem \ref{stab-poro} implies  
\[
C_2\VERT(\bu-\bu_h,\bomega-\bomega_h,\phi-\phi_h,p-p_h)\VERT^2\le B_{\mathrm{P}}((\bu-\bu_h,\bomega-\bomega_h,\phi-\phi_h,p-p_h),(\bv,\btheta,\psi,q)),
\]  
with $\VERT(\bv,\btheta,\psi,q)\VERT\le C_1 \VERT(\bu-\bu_h,\bomega-\bomega_h,\phi-\phi_h,p-p_h)\VERT$. 
From the definition of $B$, it follows that 
\begin{align*}
& B_{\mathrm{P}}((\bu-\bu_h,\bomega-\bomega_h,\phi-\phi_h,p-p_h),(\bv,\btheta,\psi,q))=-(\ff^{\mathrm{P}}-\ff_h^{\mathrm{P}},\bv-\bv_h)_{0,\Omega}-\rho\xi^{-1}(\kappa\gg, \nabla (q-q_h))_{0,\Omega}\\
&\quad \quad +(s^{\mathrm{P}}-s^{\mathrm{P}}_h,q-q_h)_{0,\Omega}-(\ff^{\mathrm{P}}_h,\bv-\bv_h)_{0,\Omega}+(s^{\mathrm{P}}_h,q-q_h)_{0,\Omega}-B_{\mathrm{P}}((\bu_h,\bomega_h,\phi_h,p_h),(\bv-\bv_h,\btheta,\psi,q-q_h)).
\end{align*}
Finally, applying integration by parts, Cauchy-Schwarz inequality and approximation results, yields 
\[
B_{\mathrm{P}}((\bu-\bu_h,\bomega-\bomega_h,\phi-\phi_h,p-p_h),(\bv,\btheta,\psi,q))\le C (\Psi+\widehat{\Upsilon}) \VERT(\bv,\btheta,\psi,q)\VERT.
\]

\end{proof}

\subsubsection{Efficiency} 
The following lemmas provide   upper bounds for each term defining $\estP$.

\begin{lemma}\label{lemP1}
There holds:
\[
h_K(\mu^{\mathrm{P}})^{-1/2}\|\mathbf{R}_1\|_{0,K}^2\lesssim (\mu^{\mathrm{P}})^{-1/2}h_K\|\ff^{\mathrm{P}}-\ff_h^{\mathrm{P}}\|_{0,K}+(\mu^{\mathrm{P}})^{-1/2}\|\phi-\phi_h\|_{0,K}+\|\bomega-\bomega_h\|_{0,K}.
\]
\end{lemma}
\begin{proof}
It follows from Lemma \ref{lemE1}.
\end{proof}
\begin{lemma}\label{lemP2}
There holds:
\begin{align*}
\|\mathbf{R}_2\|_{0,K}\lesssim \|\bomega-\bomega_h\|_{0,K}+\sqrt{\mu^{\mathrm{P}}}\|\curl(\bu-\bu_h)\|_{0,K}.
\end{align*}
\end{lemma}
\begin{proof}
It follows from Lemma \ref{lemE2}.
\end{proof}
\begin{lemma}\label{lemP3}
There holds:
\[
\rho_d^{1/2}\|\mathbf{R}_3\|_{0,K}\lesssim \sqrt{\mu^{P}}\|\vdiv(\bu-\bu_h)\|_{0,K}+(2\mu^{\mathrm{P}}+\lambda^{\mathrm{P}})^{-1/2}\|\phi-\phi_h\|_{0,K}+\alpha(2\mu^{\mathrm{P}}+\lambda^{\mathrm{P}})^{-1/2}\|{p}-{p}_h\|_{0,K}.
\]
\end{lemma}
\begin{proof}
Using the expression $\vdiv \bu+(2\mu^{\mathrm{P}}+\lambda^{\mathrm{P}})^{-1}\phi-\alpha(2\mu^{\mathrm{P}}+\lambda^{\mathrm{P}})^{-1}p=0$, we have
\begin{align*}
\rho_d^{1/2}\|\mathbf{R}_3\|_{0,K}&=\rho_d^{1/2}\|\vdiv{\bu}_h+(2\mu^{\mathrm{P}}+\lambda^{\mathrm{P}})^{-1}\phi_h-\alpha(2\mu^{\mathrm{P}}+\lambda^{\mathrm{P}})^{-1}p_h\|_{0,K}\\
&\lesssim \sqrt{\mu^{P}}\|\vdiv(\bu-\bu_h)\|_{0,K}+(2\mu^{\mathrm{P}}+\lambda^{\mathrm{P}})^{-1/2}\|\phi-\phi_h\|_{0,K}+\alpha(2\mu^{\mathrm{P}}+\lambda^{\mathrm{P}})^{-1/2}\|{p}-{p}_h\|_{0,K}.
\end{align*}
\end{proof}
\begin{lemma}\label{lemP4}
There holds:
\begin{align*}
h_K(\mu^{\mathrm{P}})^{-1/2}\|{R}_4\|_{0,K}^2&\lesssim (\rho_1)^{1/2}\|s^{\mathrm{P}}-s^{\mathrm{P}}_h\|_{0,K}+[c_0+\alpha^2(2\mu+\lambda)^{-1}]^{1/2}\|p-p_h\|_{0,K}+(\kappa/\xi)^{1/2}\|\nabla(p-p_h)\|_{0,K}\\
&\quad +(2\mu^{\mathrm{P}}+\lambda^{\mathrm{P}})^{-1/2}\|\phi-\phi_h\|_{0,K}.
\end{align*}
\end{lemma}
\begin{proof}
For each $K\in \mathcal{T}_h$, we can take  
$\bzeta|_K=(\rho_1)^{-1}{R}_4b_K.$ Then, 
invoking (\ref{ele1}), we end up with  
\[
(\rho_1)^{-1}\|{R}_4\|_{0,K}^2\lesssim \int_K {R}_4 ((\rho_1)^{-1}{R}_4b_K) =\int_{K}{R}_4 \bzeta .
\]
Recall that $s-[c_0+\alpha^2(2\mu+\lambda)^{-1}]p+\alpha(2\mu^{\mathrm{P}}+\lambda^{\mathrm{P}})^{-1}\phi+\xi^{-1}\vdiv[\kappa(\nabla p-\rho \gg)]_K=0$. We subtract this from the last term, and then integrate using $\bzeta|_{\partial K}=\cero$, to obtain 
\begin{align*}
(\rho_1)^{-1}\|{R}_4\|_{0,K}^2&\lesssim \int_{K} (s^{\mathrm{P}}_h-s^{\mathrm{P}}) \bzeta +[c_0+\alpha^2(2\mu+\lambda)^{-1}]\int_K (p-p_h) \bzeta +\xi^{-1}\int_K \kappa\nabla(p-p_h)\cdot\bnabla \bzeta \\
&\quad+\alpha(2\mu+\lambda)^{-1}\int_K (\phi-\phi_h) \bzeta.
\end{align*}
Then, Cauchy-Schwarz inequality gives
\begin{align*}
(\rho_1)^{-1}\|{R}_4\|_{0,K}^2\lesssim &((\rho_1)^{1/2}\|s^{\mathrm{P}}-s^{\mathrm{P}}_h\|_{0,K}+[c_0+\alpha^2(2\mu+\lambda)^{-1}]^{1/2}\|p-p_h\|_{0,K}+\xi^{-1/2}\|\kappa^{1/2}\nabla(p-p_h)\|_{0,K}\\
&\quad +(2\mu^{\mathrm{P}}+\lambda^{\mathrm{P}})^{-1/2}\|\phi-\phi_h\|_{0,K})((\kappa/\xi)^{1/2}\|\bnabla \bzeta\|_{0,K}+(\rho_1)^{-1/2})\|\bzeta\|_{0,K}.
\end{align*}
And the proof follows after noting that 
\begin{align*}
(\frac{\kappa}{\xi})^{1/2}\|\bnabla \bzeta\|_{0,K}+(\rho_1)^{-1/2}\|\bzeta\|_{0,K}&\lesssim
(\frac{\kappa}{\xi})^{1/2}h_K^{-1}\| \bzeta\|_{0,K}+\rho_1^{-1/2}\|\bzeta\|_{0,K})\lesssim (\rho_1)^{-1/2} \|\bzeta\|_{0,K}
= (\rho_1)^{1/2}\|\mathbf{R}_4\|_{0,K}.
\end{align*}
\end{proof}
\begin{lemma}\label{lemP5}
There holds:
\begin{align*}
(\sum_{e\in\partial K}h_e(\mu^{\mathrm{P}})^{-1} \|\mathbf{R}_e\|_{0,e}^2)^{1/2}\lesssim \sum_{K\in P_e}((\mu^{\mathrm{P}})^{-1/2}h_K\|\ff^{\mathrm{P}}-\ff_h^{\mathrm{P}}\|_{0,K}+(\mu^{\mathrm{P}})^{-1/2}\|\phi-\phi_h\|_{0,K}+\|\bomega-\bomega_h\|_{0,K}).
\end{align*}
\end{lemma}
\begin{proof}
It readily follows from Lemma \ref{lemE4}. 
\end{proof}
\begin{lemma}\label{lemP6}
There holds:
\begin{align*}
(\sum_{e\in\partial K}\rho_2 \|{R}_e\|_{0,e}^2)^{1/2}&\lesssim \sum_{K\in P_e} ((\rho_1)^{1/2}\|s^{\mathrm{P}}-s^{\mathrm{P}}_h\|_{0,K}+[c_0+\alpha^2(2\mu+\lambda)^{-1}]^{1/2}\|p-p_h\|_{0,K}+(\kappa/\xi)^{1/2}\|\nabla(p-p_h)\|_{0,K}\\
&\quad +(2\mu^{\mathrm{P}}+\lambda^{\mathrm{P}})^{-1/2}\|\phi-\phi_h\|_{0,K}).
\end{align*}
\end{lemma}
\begin{proof}
For  $e\in \mathcal{E}(\mathcal{T}_h)$  we can locally choose 
$\bzeta_e =\rho_2{R}_eb_e.$ Then, 
from (\ref{edgee1}), we readily have that 
\begin{align}
\rho_2\|{R}_e\|_{0,e}^2\lesssim \int_e {R}_e\cdot (\rho_2{R}_eb_e)=\int_{e}{R}_e\cdot \bzeta_e.\label{edge11}
\end{align}
And the weak form (\ref{weak-poro}) leads to  
\begin{align*}
\int_{\Omega}(\kappa/\xi)\nabla (p-p_{h})\cdot \nabla q_h&= (-\big[c_0+\frac{\alpha^2}{(2\mu^{\mathrm{P}}+\lambda^{\mathrm{P}})}\big]\int_{\Omega}pq_h + \frac{\alpha}{2\mu^{\mathrm{P}}+\lambda^{\mathrm{P}}}\int_{\Omega}\phi q_h)\\
&\quad+ (\frac{\rho}{\xi}\int_{\Omega}\kappa\gg\cdot \nabla q_h +\int_{\Omega} sq )-\int_{\Omega}(\kappa/\xi)\nabla p_{h}\cdot \nabla q_h, 
\end{align*}
for all $q_h\in V_h$. We can next apply integration by parts 
and choose 
$q_h=\bzeta_e$. Then, from (\ref{edge11}) we arrive at 
\begin{align*}
\rho_2\|{R}_e\|_{0,e}^2\lesssim& \sum_{K\in P_e}\int_K R_4 \bzeta_e+\sum_{K\in P_e}\int_{K}(\kappa/\xi)\nabla (p-p_{h})\cdot \bnabla \bzeta_e
+ \sum_{K\in P_e}\int_K (s^{\mathrm{P}}-s^{\mathrm{P}}_h)\bzeta_e\\
&-\big(c_0+\frac{\alpha^2}{(2\mu^{\mathrm{P}}+\lambda^{\mathrm{P}})}\big)\sum_{K\in P_e}\int_{K}(p-p_h)\bzeta_e + \frac{\alpha}{2\mu^{\mathrm{P}}+\lambda^{\mathrm{P}}}\sum_{K\in P_e}\int_{K}(\phi-\phi_h) \bzeta_e. 
\end{align*}
And the proof is completed after applying Cauchy-Schwarz inequality. 
\end{proof}
\begin{theorem}[Efficiency estimate for the Biot problem]\label{poroeff}
Let $(\bu,\bomega,\phi,p)$ and $(\bu_h,\bomega_h,\phi_h,p_h)$ be the solutions to the formulations (\ref{weak-poro}) and (\ref{weakdisP}) (or (\ref{weakdisPstab})), respectively. Then the following efficiency bound holds.
\[
\Psi\le C_{\mathrm{eff}}(\VERT(\bu-\bu_h,\bomega-\bomega_h,\phi-\phi_h,p-p_h)\VERT+\widehat{\Upsilon}),
\]
where $C_{\mathrm{eff}}>0$ is a constant independent of mesh size and model parameters.
\end{theorem}
\begin{proof}
It is a direct consequence of combining Lemmas \ref{lemP1}--\ref{lemP6}.
\end{proof}


\section{Rotation-based elasticity-poroelasticity interface problem}\label{sec:interface}
\subsection{Continuous formulation}
Let $\Omega$ be now partitioned into non-overlapping and connected subdomains $\Omega^{\mathrm{E}}$, $\Omega^{\mathrm{P}}$ representing zones composed by the non-pay rock (linearly elastic domain) and a reservoir (poroelastic domain), respectively. We focus on the case where the reservoir is completely surrounded by the elastic subdomain,  such that the interface $\Sigma  = \partial \Omega^{\mathrm{P}} \cap \partial \Omega^{\mathrm{E}}$, coincides with the boundary of the pay zone. We consider that the normal unit vector $\bfitn$ on $\Sigma$ points from $\Omega^{\mathrm{P}}$ to $\Omega^{\mathrm{E}}$. The problem is stated as follows,  
which is as in \cite{anaya_sisc20}, except for the particular scaling used herein 
\begin{subequations}
\begin{align}
    \sqrt{\mu^{\mathrm{P}}}\curl\bomega^{\mathrm{P}}+\nabla \phi^{\mathrm{P}} &= \ff^{\mathrm{P}} & \text{in $\Omega^{\mathrm{P}}$,}  \label{eqP:momentumPE}
    \\
    \bomega^{\mathrm{P}}-\sqrt{\mu^{\mathrm{P}}}\curl \bu^{\mathrm{P}} &= \mathbf{0} & \text{in $\Omega^{\mathrm{P}}$,}  \label{eqP:omegaintPE}
    \\
   (2\mu^{\mathrm{P}}+\lambda^{\mathrm{P}})^{-1} \phi^{\mathrm{P}}+\vdiv\bu^{\mathrm{P}}  -\alpha(2\mu^{\mathrm{P}}+\lambda^{\mathrm{P}})^{-1}p^{\mathrm{P}} &= 0 & \text{in $\Omega^{\mathrm{P}}$,}   \label{eqP:phiPE}
    \\
    \big[c_0+\alpha^2(\mu^{\mathrm{P}}+\lambda^{\mathrm{P}})^{-1}\big]p^{\mathrm{P}} - \alpha(2\mu^{\mathrm{P}}+\lambda^{\mathrm{P}})^{-1}\phi^{\mathrm{P}}- \frac{1}{\xi}\vdiv \big{[}\kappa(\nabla p^{\mathrm{P}} - \rho \boldsymbol{g})\big{]} &= s^{\mathrm{P}} & \text{in $\Omega^{\mathrm{P}}$,}     \label{eqP:massPE}\\
\sqrt{\mu^{\mathrm{E}}}\curl \bomega^{\mathrm{E}}+\nabla p^{\mathrm{E}}&=\ff^{\mathrm{E}} & \text{in $\Omega^{\mathrm{E}}$},\\
\bomega-\sqrt{\mu^{\mathrm{E}}}\curl \bu^{\mathrm{E}} &=0&\text{in $\Omega^{\mathrm{E}}$},\\
\mathrm{div}\bu^{\mathrm{E}}+(2\mu^{\mathrm{E}}+\lambda^{\mathrm{E}})^{-1}p^{\mathrm{E}} &=0 &\text{in $\Omega^{\mathrm{E}}$},\\
\bu^{\mathrm{E}}&=0 &\text{on $\Gamma$},\\
\label{eq:transmission}
\bu^{\mathrm{P}}=\bu^{\mathrm{E}},\quad \sqrt{\mu^{\mathrm{P}}}\bomega^{\mathrm{P}}\times\bfitn+\phi^{\mathrm{P}}\bfitn=\sqrt{\mu^{\mathrm{E}}}\bomega^{\mathrm{E}}\times\bfitn+p^{\mathrm{E}}\bfitn,\quad \frac{\kappa}{\xi} (\nabla p^{\mathrm{P}} - \rho\gg)\cdot\nn  & = 0& \text{on $\Sigma$}.
\end{align}
\end{subequations}

The weak formulation of the rotation-based Biot's poroelasticity in $\Omega^{\mathrm{P}}$ is as follows:
\begin{align*}
-\sqrt{\mu^{\mathrm{P}}}\int_{\Omega^{\mathrm{P}}}\curl\bv^{\mathrm{P}}\cdot \bomega^{\mathrm{P}}+\int_{\Omega^{\mathrm{P}}}\phi^{\mathrm{P}}\mathrm{div}\bv^{\mathrm{P}} -\langle\sqrt{\mu^{\mathrm{E}}}\bomega^{\mathrm{E}}\times\bfitn+p^{\mathrm{E}}\bfitn, \bv^{\mathrm{P}}\rangle_{\Sigma}& = -\int_{\Omega^{\mathrm{P}}}\ff^{\mathrm{P}}\cdot \bv^{\mathrm{P}},  
	\\
\int_{\Omega^{\mathrm{P}}}\bomega^{\mathrm{P}}\cdot \btheta^{\mathrm{P}}-\sqrt{\mu^{\mathrm{E}}}\int_{\Omega^{\mathrm{P}}}\btheta^{\mathrm{P}} \cdot\curl \bu^{\mathrm{P}} &= 0,  	\\
	(2\mu^{\mathrm{P}}+\lambda^{\mathrm{P}})^{-1} \int_{\Omega^{\mathrm{P}}}\phi^{\mathrm{P}}\psi^{\mathrm{P}}+\int_{\Omega^{\mathrm{P}}}\psi^{\mathrm{P}}\vdiv\bu^{\mathrm{P}}  -\alpha(2\mu^{\mathrm{P}}+\lambda^{\mathrm{P}})^{-1}\int_{\Omega^{\mathrm{P}}}p^{\mathrm{P}}\psi^{\mathrm{P}} &= 0, 	\\
	-\big[c_0+\frac{\alpha^2}{(2\mu^{\mathrm{P}}+\lambda^{\mathrm{P}})}\big]\int_{\Omega^{\mathrm{P}}}p^{\mathrm{P}}q^{\mathrm{P}} + \frac{\alpha}{2\mu^{\mathrm{P}}+\lambda^{\mathrm{P}}}\int_{\Omega^{\mathrm{P}}}\phi^{\mathrm{P}} q^{\mathrm{P}}-\int_{\Omega^{\mathrm{P}}} \frac{\kappa}{\xi}\nabla p^{\mathrm{P}}\cdot \nabla q^{\mathrm{P}} 
&= -\frac{\rho}{\xi}\int_{\Omega^{\mathrm{P}}}\kappa\gg\cdot \nabla q^{\mathrm{P}} -\int_{\Omega^{\mathrm{P}}} s^{\mathrm{P}}q^{\mathrm{P}} , 
\end{align*}
for each $(\bv^{\mathrm{P}},\btheta^{\mathrm{P}} ,\psi^{\mathrm{P}} ,q^{\mathrm{P}} )\in \mathbf{H}^1(\Omega^{\mathrm{P}})\times \mathbf{L}^2(\Omega^{\mathrm{P}})\times\mathrm{L}^2(\Omega^{\mathrm{P}})\times \mathrm{H}^1(\Omega^{\mathrm{P}})$. Similarly, for the equations of linear 
elasticity in $\Omega^{\mathrm{E}}$ we get  
\begin{align*}
-\sqrt{\mu^{\mathrm{E}}}\int_{\Omega^{\mathrm{E}}} \curl\bv^{\mathrm{E}}\cdot\bomega^{\mathrm{E}} +\int_{\Omega^{\mathrm{E}}} p^{\mathrm{E}} \vdiv\bv^{\mathrm{E}}+\langle\sqrt{\mu^{\mathrm{E}}}\bomega^{\mathrm{E}}\times\bfitn+p^{\mathrm{E}}\bfitn, \bv^{\mathrm{E}}\rangle_{\Sigma}&=-\int_{\Omega^{\mathrm{E}}}\ff^{\mathrm{E}}\cdot \bv^{\mathrm{E}},\\
\int_{\Omega^{\mathrm{E}}}\bomega^{\mathrm{E}} \cdot\btheta^{\mathrm{E}} -\sqrt{\mu^{\mathrm{E}}}\int_{\Omega^{\mathrm{E}}}\btheta^{\mathrm{E}}\cdot\curl \bu^{\mathrm{E}} &=0,\\
\int_{\Omega}\vdiv\bu^{\mathrm{E}} q^{\mathrm{E}}+(2\mu^{\mathrm{E}}+\lambda^{\mathrm{E}})^{-1}\int_{\Omega^{\mathrm{E}}}p^{\mathrm{E}} q^{\mathrm{E}} &=0,
\end{align*}
for each $(\bv^{\mathrm{E}},\btheta^{\mathrm{E}} ,q^{\mathrm{E}} )\in \mathbf{H}^1_{\Gamma}(\Omega^{\mathrm{E}})\times \mathbf{L}^2(\Omega^{\mathrm{E}})\times\mathrm{L}^2(\Omega^{\mathrm{E}})$, where
$\mathbf{H}^1_{\Gamma}(\Omega^{\mathrm{E}})=\{\bv\in \mathbf{H}^1(\Omega^{\mathrm{E}}):\bv^{\mathrm{E}}=\cero\; \mbox{on}\;\Gamma\}.$ 
We define $\overrightarrow{\bomega}:= \{\bomega^{\mathrm{P}},\phi^{\mathrm{P}},\bomega^{\mathrm{E}},p^{\mathrm{E}}\}$ and write the weak formulation: find 
$(\overrightarrow{\bomega},\bu,p^{\mathrm{P}})\in \bH\times \bV\times \rQ^{\mathrm{P}}$ such that
\begin{align*}
a(\overrightarrow{\bomega},\overrightarrow{\btheta})+b_1(\overrightarrow{\btheta},\bu)-b_2(\overrightarrow{\btheta},p^{\mathrm{P}})&=0 &\forall \overrightarrow{\btheta}\in\bH,\\
b_1(\overrightarrow{\bomega},\bv)&=F(\bv) & \forall \bv\in\bV,\\
b_3(\overrightarrow{\bomega},q^{\mathrm{P}})-c(p^{\mathrm{P}},q^{\mathrm{P}})&=F(q^{\mathrm{P}}) & \forall q\in \rQ^{\mathrm{P}},
\end{align*} 
where $\overrightarrow{\btheta}:= (\btheta^{\mathrm{P}},\psi^{\mathrm{P}},\btheta^{\mathrm{E}},q^{\mathrm{E}})$. We define spaces as 
\begin{align*}
	\mathbf{H}:=\bL^2(\Omega^{\mathrm{P}})\times \rL^2(\Omega^{\mathrm{P}})\times \bL^2(\Omega^{\mathrm{E}})\times \rL^2(\Omega^{\mathrm{E}}),\quad
	\bV:=&\bH_0^1(\Omega), \quad 
	\rQ^{\mathrm{P}}:=\rH^1(\Omega^{\mathrm{P}}),
\end{align*}
and the bilinear forms $a:\mathbf{H}\times\mathbf{H}\to \RR$, 
$b_1:\mathbf{H}\times\bV\to \RR$, $b_2:\mathbf{H}\times\rQ^{\mathrm{P}}\to \RR$, $b_3:\mathbf{H}\times\rQ^{\mathrm{P}}\to\RR$, $c:\rQ^{\mathrm{P}}\times\rQ^{\mathrm{P}}\to \RR$, 
and linear functionals $F:\bV\to\RR$, $G:\rQ^{\mathrm{P}}\to\RR$ 
are specified in the following way
\begin{align*}
	& a(\vomega,\vtheta) := \int_{\Omega^{\mathrm{P}}}\bomega^{\mathrm{P}}\cdot \btheta^{\mathrm{P}} + \frac{1}{2\mu^{\mathrm{P}}+\lambda^{\mathrm{P}}}\int_{\Omega^{\mathrm{P}}}\phi^{\mathrm{P}}\psi^{\mathrm{P}}+\int_{\Omega^{\mathrm{E}}}\bomega^{\mathrm{E}}\cdot \btheta^{\mathrm{E}} + \frac{1}{2\mu^{\mathrm{E}}+\lambda^{\mathrm{E}}}\int_{\Omega^{\mathrm{E}}}p^{\mathrm{E}}q^{\mathrm{E}},\\
	 &b_1(\vtheta,\bv):=-\sqrt{\mu^{\mathrm{P}}}\int_{\Omega^{\mathrm{P}}}\btheta^{\mathrm{P}}\cdot\curl\bv +\int_{\Omega^{\mathrm{P}}}\psi^{\mathrm{P}}\vdiv\bv-\sqrt{\mu^{\mathrm{E}}}\int_{\Omega^{\mathrm{E}}}\btheta^{\mathrm{E}}\cdot\curl\bv +\int_{\Omega^{\mathrm{E}}}p^{\mathrm{E}}\vdiv\bv,\\
	& b_2(\vtheta,p^{\mathrm{P}} ):=\frac{\alpha}{(\lambda^{\mathrm{P}}+2\mu^{\mathrm{P}})}\int_{\Omega^{\mathrm{P}}}p^{\mathrm{P}}\psi^{\mathrm{P}},\qquad b_3(\vomega,q^{\mathrm{P}} ):=\frac{\alpha}{(2\mu^{\mathrm{P}}+\lambda^{\mathrm{P}})}\int_{\Omega^{\mathrm{P}}}q^{\mathrm{P}}\phi^{\mathrm{P}},\\
	& c(p^{\mathrm{P}} ,q^{\mathrm{P}} ):=\left[c_0+\frac{\alpha^2}{(2\mu^{\mathrm{P}}+\lambda^{\mathrm{P}})}\right]\int_{\Omega^{\mathrm{P}}}p^{\mathrm{P}}q^{\mathrm{P}}+\frac{1}{\xi}\int_{\Omega^{\mathrm{P}}} \kappa\nabla p^{\mathrm{P}}\cdot \nabla q^{\mathrm{P}},\\
	&F(\bv) := -\int_{\Omega^{\mathrm{P}}} \ff^{\mathrm{P}}\cdot\bv-\int_{\Omega^{\mathrm{E}}} \ff^{\mathrm{E}}\cdot\bv,\qquad G(q^{\mathrm{P}}) := -\frac{\rho}{\xi} \int_{\Omega^{\mathrm{P}}}\kappa \gg\cdot \nabla q^{\mathrm{P}} 
	-  \int_{\Omega^{\mathrm{P}}} s^{\mathrm{P}}q^{\mathrm{P}}.
\end{align*}
For the forthcoming analysis, we will consider the following $(\mu^{\mathrm{E}},\mu^{\mathrm{P}})-$dependent
norm (see, for instance, \cite[Remark~2.7]{girault_book86} for the case of a single-physics domain) for the displacements:
\begin{align*}
    \|\bv\|_{\bV}^2:=\mu^{\mathrm{P}}\|\curl\bv\|_{0,\Omega^{\mathrm{P}}}^2+\mu^{\mathrm{P}}\|\vdiv\bv\|_{0,\Omega^{\mathrm{P}}}^2+\mu^{\mathrm{E}}\|\curl\bv\|_{0,\Omega^{\mathrm{E}}}^2+\mu^{\mathrm{E}}\|\vdiv\bv\|_{0,\Omega^{\mathrm{E}}}^2,
\end{align*}
and $\mathbf{H}$ will be endowed with the norm
\begin{equation*}
\|\vtheta\|^2_{\mathbf{H}}:=\|\btheta^{\mathrm{P}}\|^2_{0,\Omega^{\mathrm{P}}}+\frac{1}{\mu^{\mathrm{P}}}\|\psi^{\mathrm{P}}_0\|^2_{0,\Omega^{\mathrm{P}}}+\frac{1}{2\mu^{\mathrm{P}}+\lambda^{\mathrm{P}}}\|\psi^{\mathrm{P}}\|^2_{0,\Omega^{\mathrm{P}}}+\|\btheta^{\mathrm{E}}\|^2_{0,\Omega^{\mathrm{E}}}+\frac{1}{\mu^{\mathrm{E}}}\|q^{\mathrm{E}}_0\|^2_{0,\Omega^{\mathrm{E}}}+\frac{1}{2\mu^{\mathrm{E}}+\lambda^{\mathrm{E}}}\|q^{\mathrm{E}}\|^2_{0,\Omega^{\mathrm{E}}}.
\end{equation*}
Now, we write down the compact form of the weak formulation as follows:
\begin{equation}\label{weakEP}
B_{\mathrm{I}}((\overrightarrow{\bomega},\bu,p^{\mathrm{P}}),(\overrightarrow{\btheta},\bv,q^{\mathrm{P}}))=F(\bv)+G(q),
\end{equation}
where the multilinear form now has a subscript I (for \emph{interface} problem), and it is defined as  
\[B_{\mathrm{I}}((\overrightarrow{\bomega},\bu,p^{\mathrm{P}}),(\overrightarrow{\btheta},\bv,q^{\mathrm{P}})):=a(\overrightarrow{\bomega},\overrightarrow{\btheta})+b_1(\overrightarrow{\btheta},\bu)-b_2(\overrightarrow{\btheta},p^{\mathrm{P}})
+b_1(\overrightarrow{\bomega},\bv)+b_3(\overrightarrow{\bomega},q^{\mathrm{P}})-c(p^{\mathrm{P}},q^{\mathrm{P}}).\]
We now turn our attention to the stability estimates. The following theorem will also be very useful in the forthcoming analysis.  
\begin{theorem}
For every $(\overrightarrow{\bomega},\bu,p^{\mathrm{P}})\in \bH\times \bV\times Q^{\mathrm{P}}$, there exits $(\overrightarrow{\btheta},\bv,q^{\mathrm{P}})\in \bH\times \bV\times Q^{\mathrm{P}}$ with $\VERT(\overrightarrow{\btheta},\bv,q^{\mathrm{P}})\VERT\le C_1\VERT(\overrightarrow{\bomega},\bu,p^{\mathrm{P}})\VERT$ such that
\[
B_{\mathrm{I}}((\overrightarrow{\bomega},\bu,p^{\mathrm{P}}),(\overrightarrow{\btheta},\bv,q^{\mathrm{P}}))\ge C_2 \VERT(\overrightarrow{\bomega},\bu,p^{\mathrm{P}})\VERT^2,
\]
where $\VERT(\overrightarrow{\btheta},\bv,q^{\mathrm{P}})\VERT^2:= \|\bv\|^2_{\bV}+\|\overrightarrow{\btheta}\|^2_{\bH}+\left(c_0+\frac{\alpha^2}{(2\mu^{\mathrm{P}}+\lambda^{\mathrm{P}})}\right)\|q^{\mathrm{P}}\|^2_{0,\Omega^{\mathrm{P}}}+\|\kappa/\xi (\nabla q^{\mathrm{P}})\|^2_{0,\Omega^{\mathrm{P}}}$.
\end{theorem}
\begin{proof}
Invoking the relevant inf-sup condition, for each $p^{\mathrm{E}}\in\rL^2(\Omega^{\mathrm{E}})$ and $\phi^{\mathrm{P}}\in\rL^2(\Omega^{\mathrm{P}})$,
we can find $\bv_0^{\mathrm{E}}\in \bH^1_0(\Omega^{\mathrm{E}})$ and $\bv_0^{\mathrm{P}}\in\bH^1_0(\Omega^{\mathrm{P}})$ such that
\begin{align*}
&(\vdiv\bv_0^{\mathrm{E}}, p^{\mathrm{E}})_{0,\Omega^{\mathrm{E}}}\ge C_{\Omega^{\mathrm{E}}}/ \mu^{\mathrm{E}}\|p^{\mathrm{E}}_0\|_{0,\Omega^{\mathrm{E}}}^2, \quad \sqrt{\mu^{\mathrm{E}}}\|\bnabla \bv_0^{\mathrm{E}}\|_{0,\Omega^{\mathrm{E}}}\le 1/\sqrt{\mu^{\mathrm{E}}}\|p^{\mathrm{E}}_0\|_{0,\Omega^{\mathrm{E}}},\\
&(\vdiv\bv_0^{\mathrm{P}}, \phi^{\mathrm{P}})_{0,\Omega^{\mathrm{P}}}\ge C_{\Omega^{\mathrm{P}}}/ \mu^{\mathrm{P}}\|\phi^{\mathrm{P}}_0\|_{0,\Omega^{\mathrm{P}}}^2, \quad \sqrt{\mu^{\mathrm{P}}}\|\bnabla \bv_0^{\mathrm{P}}\|_{0,\Omega^{\mathrm{P}}}\le 1/\sqrt{\mu^{\mathrm{P}}}\|\phi^{\mathrm{P}}_0\|_{0,\Omega^{\mathrm{P}}}.
\end{align*}
Hence, for $\bv_0\in\bH^1_0(\Omega)$ such that $\bv_0|_{\Omega^{\mathrm{E}}}=\bv_0^{\mathrm{E}}$ and $\bv_0|_{\Omega^{\mathrm{P}}}=\bv_0^{\mathrm{P}}$, we have
\begin{align*}
&B_{\mathrm{I}}((\overrightarrow{\bomega},\bu,p^{\mathrm{P}}),(\cero,\bv_0,0))=-\sqrt{\mu^{\mathrm{P}}}\int_{\Omega^{\mathrm{P}}}\bomega^{\mathrm{P}}\cdot\curl\bv_0^{\mathrm{P}} +\int_{\Omega^{\mathrm{P}}}\phi^{\mathrm{P}}\vdiv\bv_0^{\mathrm{P}}-\sqrt{\mu^{\mathrm{E}}}\int_{\Omega^{\mathrm{E}}}\bomega^{\mathrm{E}}\cdot\curl\bv_0^{\mathrm{E}} +\int_{\Omega^{\mathrm{E}}}p^{\mathrm{E}}\vdiv\bv_0^{\mathrm{E}}\\
&\qquad \qquad \ge \left(C_{\Omega^{\mathrm{P}}}-\frac{1}{2\epsilon_1}\right)\frac{1}{\mu^{\mathrm{P}}}\|\phi^{\mathrm{P}}_0\|_{0,\Omega^{\mathrm{P}}}^2+\left(C_{\Omega^{\mathrm{E}}}-\frac{1}{2\epsilon_2}\right)\frac{1}{\mu^{\mathrm{E}}}\|p^{\mathrm{E}}_0\|_{0,\Omega^{\mathrm{E}}}^2-\frac{\epsilon_2}{2}\|\bomega^{\mathrm{E}}\|_{0,\Omega^{\mathrm{E}}}-\frac{\epsilon_1}{2}\|\bomega^{\mathrm{P}}\|_{0,\Omega^{\mathrm{P}}}.
\end{align*}
Selecting $\overrightarrow{\btheta}=\overrightarrow{\bomega}$, $\bv=-\bu$ and $q^{\mathrm{P}}=-p^{\mathrm{P}}$ we have
\begin{align*}
&B_{\mathrm{I}}((\overrightarrow{\bomega},\bu,p^{\mathrm{P}}),(\overrightarrow{\bomega},-\bu,-p^{\mathrm{P}}))=  \|\bomega^{\mathrm{E}}\|^2_{0,\Omega^{\mathrm{E}}}+(2\mu^{\mathrm{E}}+\lambda^{\mathrm{E}})^{-1}\|p^{\mathrm{E}}\|^2_{0,\Omega^{\mathrm{E}}}+\|\bomega^{\mathrm{P}}\|^2_{0,\Omega^{\mathrm{P}}}+(2\mu^{\mathrm{P}}+\lambda^{\mathrm{P}})^{-1}\|\phi^{\mathrm{P}}\|^2_{0,\Omega^{\mathrm{P}}}\nonumber\\
&\qquad \qquad \qquad -2\alpha(2\mu^{\mathrm{P}}+\lambda^{\mathrm{P}})^{-1}(p^{\mathrm{P}},\phi^{\mathrm{P}})_{0,\Omega^{\mathrm{P}}}
+(c_0+\alpha^2(2\mu^{\mathrm{P}}+\lambda^{\mathrm{P}})^{-1})\|p^{\mathrm{P}}\|^2_{0,\Omega^{\mathrm{P}}}+\|(\kappa/\xi)^{1/2}\nabla p\|^2_{0,\Omega^{\mathrm{P}}}.
\end{align*}
Next, we take $\overrightarrow{\btheta}_1:=(-\sqrt{\mu^{\mathrm{P}}}\curl\bu^{\mathrm{P}},\mu^{\mathrm{P}}\vdiv \bu^{\mathrm{P}},-\sqrt{\mu^{\mathrm{E}}}\curl\bu^{\mathrm{E}},\mu^{\mathrm{E}}\vdiv \bu^{\mathrm{E}})\in\bH$, $\bv=\cero$ and $q^{\mathrm{P}}=0$. Then 
\begin{align*}
&B_{\mathrm{I}}((\overrightarrow{\bomega},\bu,p^{\mathrm{P}}),(\overrightarrow{\btheta}_1,\cero,0))\\
&=\mu^{\mathrm{E}}\|\curl\bu^{\mathrm{E}}\|^2_{0,\Omega^{\mathrm{E}}}+\mu^{\mathrm{E}}\|\vdiv\bu^{\mathrm{E}}\|^2_{0,\Omega^{\mathrm{E}}}-\sqrt{\mu^{\mathrm{E}}}(\bomega^{\mathrm{E}},\curl\bu^{\mathrm{E}})_{0,\Omega^{\mathrm{E}}}
+\mu^{\mathrm{E}}/(2\mu^{\mathrm{E}}+\lambda^{\mathrm{E}})(p^{\mathrm{E}}, \vdiv\bu^{\mathrm{E}})_{0,\Omega^{\mathrm{E}}}\\
&\quad+\mu^{\mathrm{P}}\|\curl\bu^{\mathrm{P}}\|^2_{0,\Omega^{\mathrm{P}}}+\mu^{\mathrm{P}}\|\vdiv\bu^{\mathrm{P}}\|^2_{0,\Omega^{\mathrm{P}}}-\sqrt{\mu^{\mathrm{P}}}(\bomega^{\mathrm{P}},\curl\bu^{\mathrm{P}})_{0,\Omega^{\mathrm{P}}}+\mu^{\mathrm{P}}(2\mu^{\mathrm{P}}+\lambda^{\mathrm{P}})^{-1}(\phi^{\mathrm{P}}, \vdiv\bu^{\mathrm{P}})_{0,\Omega^p}\nonumber\\
&\quad-\alpha\mu^{\mathrm{P}}(2\mu^{\mathrm{P}}+\lambda^{\mathrm{P}})^{-1} (p^{\mathrm{P}},\vdiv\bu^{\mathrm{P}})_{0,\Omega^{\mathrm{P}}}\nonumber\\
&\ge \frac{\mu^{\mathrm{P}}}{2}\|\curl\bu^{\mathrm{P}}\|^2_{0,\Omega^{\mathrm{P}}}+\left(1-\frac{\mu^{\mathrm{P}}}{(2\mu^{\mathrm{P}}+\lambda^{\mathrm{P}})}\right)\mu^{\mathrm{P}}\|\vdiv\bu^{\mathrm{P}}\|^2_{0,\Omega^{\mathrm{P}}}
 - \frac{1}{2}\|\bomega^{\mathrm{P}}\|^2_{0,\Omega^{\mathrm{P}}}- \frac{\alpha^2}{2(2\mu^{\mathrm{P}}+\lambda^{\mathrm{P}})}\|p^{\mathrm{P}}\|^2_{0,\Omega^{\mathrm{P}}}\\
& -\frac{1}{2(2\mu^{\mathrm{P}}+\lambda^{\mathrm{P}})}\|\phi^{\mathrm{P}}\|_{0,\Omega^{\mathrm{P}}}^2+ \frac{\mu^{\mathrm{E}}}{2}\|\curl\bu^{\mathrm{E}}\|^2+\frac{\mu^{\mathrm{E}}}{2}\|\vdiv\bu^{\mathrm{E}}\|^2_{0,\Omega^{\mathrm{E}}} - \frac{1}{2}\|\bomega^{\mathrm{E}}\|^2_{0,\Omega^{\mathrm{E}}}- \frac{\mu^{\mathrm{E}}}{2(2\mu^{\mathrm{E}}+\lambda^{\mathrm{E}})^2}\|p^{\mathrm{E}}\|^2_{0,\Omega^{\mathrm{E}}},\nonumber
\end{align*}
Then we can make the choice $\bv=-\bu+\delta_1 \bv_0$, $\overrightarrow{\btheta}=\overrightarrow{\bomega}+\delta_2\overrightarrow{\btheta}_1$ and $q^{\mathrm{P}}=-p^{\mathrm{P}}$, leading to 
\begin{align*}
&B_{\mathrm{I}}((\overrightarrow{\bomega},\bu,p^{\mathrm{P}}),(\overrightarrow{\bomega}+\delta_2\overrightarrow{\btheta}_1,-\bu+\delta_1 \bv_0,-p^{\mathrm{P}}))\\
&=B_{\mathrm{I}}((\overrightarrow{\bomega},\bu,p^{\mathrm{P}}),(\overrightarrow{\bomega},-\bu,-p^{\mathrm{P}}))+\delta_1B_{\mathrm{I}}((\overrightarrow{\bomega},\bu,p^{\mathrm{P}}),(\cero,\bv_0,0))+\delta_2B_{\mathrm{I}}((\overrightarrow{\bomega},\bu,p^{\mathrm{P}}),(\overrightarrow{\btheta}_1,\cero,0))\\
&\ge\left(1-\frac{\delta_1\epsilon_2}{2}-\frac{\delta_2}{2}\right)\|\bomega^{\mathrm{E}}\|^2_{0,\Omega^{\mathrm{E}}}+\delta_2\frac{\mu^{\mathrm{E}}}{2}\|\curl\bu^{\mathrm{E}}\|^2_{0,\Omega^{\mathrm{E}}}+\delta_2\frac{\mu^{\mathrm{E}}}{2}\|\vdiv\bu^{\mathrm{E}}\|^2_{0,\Omega^{\mathrm{E}}}+\delta_1\left(C_{\Omega^{\mathrm{E}}}-\frac{1}{2\epsilon_2}\right)\frac{1}{\mu^{\mathrm{E}}}\|p_0^{\mathrm{E}}\|^2_{0,\Omega^{\mathrm{E}}}\\
&\quad+\frac{1}{2\mu^{\mathrm{E}}+\lambda^{\mathrm{E}}}\left(1-\frac{\delta_2\mu^{\mathrm{E}}}{2\mu^{\mathrm{E}}+\lambda^{\mathrm{E}}}\right)\|p^{\mathrm{E}}\|^2_{0,\Omega^{\mathrm{E}}}+\left(1-\frac{\delta_1\epsilon_1}{2}-\frac{\delta_2}{2}\right)\|\bomega^{\mathrm{P}}\|^2_{0,\Omega^{\mathrm{P}}}+\frac{\mu^{\mathrm{P}}\delta_2}{2}\|\curl\bu^{\mathrm{P}}\|^2_{0,\Omega^{\mathrm{P}}}\\
&\quad+\delta_1\left(1-\frac{\mu^{\mathrm{P}}}{(2\mu^{\mathrm{P}}+\lambda^{\mathrm{P}})}\right)\|\vdiv\bu^{\mathrm{P}}\|^2_{0,\Omega^{\mathrm{P}}}+\left(C_{\Omega^{\mathrm{P}}}-\frac{1}{2\epsilon_2}\right)\frac{\delta_1}{\mu^{\mathrm{P}}}\|\phi_0^{\mathrm{P}}\|^2_{0,\Omega^{\mathrm{P}}}+\frac{1}{2\mu^{\mathrm{P}}+\lambda^{\mathrm{P}}}\left(\frac{1}{2}-\frac{\delta_2}{2}\right)\|\phi^{\mathrm{P}}\|^2_{0,\Omega^{\mathrm{P}}}\\
&\quad+\|(\kappa/\xi)^{1/2} \nabla p^{\mathrm{P}}\|^2_{0,\Omega^{\mathrm{P}}}+\left(c_0+\frac{\alpha^2}{(2\mu^{\mathrm{P}}+\lambda^{\mathrm{P}})}\left(\frac{1}{2}-\frac{\delta_2}{2}\right)\right)\|p^{\mathrm{P}}\|^2_{0,\Omega^{\mathrm{P}}}.
\end{align*}
Assuming the values $\epsilon_1=\epsilon_2 =\min\{1/C_{\Omega^{\mathrm{E}}},1/C_{\Omega^{\mathrm{P}}}\}$, $\delta_1 = 1/2\epsilon_1$ and $\delta_2=1/2$, we then have
{\small $$B_{\mathrm{I}}((\overrightarrow{\bomega},\bu,p^{\mathrm{P}}),(\overrightarrow{\btheta},\bv,q^{\mathrm{P}}))\ge \min\left\{\min\left\{\frac{C_{\Omega^{\mathrm{E}}}^2}{4},\frac{C_{\Omega^{\mathrm{P}}}^2}{4}\right\},\frac{1}{4}\right\} \VERT(\overrightarrow{\bomega},\bu,p^{\mathrm{P}})\VERT^2.$$}
And finally, the proof concludes after realising that 
\[
\VERT(\overrightarrow{\btheta},\bv,q^{\mathrm{P}})\VERT^2= \VERT(\overrightarrow{\bomega}+\delta_2\overrightarrow{\btheta}_1,-\bu+\delta_1 \bv_0,-p^{\mathrm{P}})\VERT^2\le 2 \VERT(\overrightarrow{\bomega},\bu,p^{\mathrm{P}})\VERT^2.
\]
\end{proof}

\subsection{Discrete spaces and Galerkin formulation} 
Let $\{\cT_{h}\}_{h>0}$ be a shape-regular
family of partitions of the closed domain 
$\bar\Omega$, conformed by tetrahedra (or triangles 
in 2D) $K$ of diameter $h_K$, with mesh size
$h:=\max\{h_K:\; K\in\cT_{h}\}$. In addition, we assume that 
the mesh is conforming with the interface. This is achieved, for example, by generating conforming simplicial meshes for  $\Omega^{\mathrm{P}}$ and for $\Omega^{\mathrm{E}}$ and 
requiring that they match on $\Sigma$ so that the union of the sub-domain meshes is a triangulation of $\Omega^{\mathrm{P}} \cup \Sigma \cup \Omega^{\mathrm{E}}$. 
We specify the finite-dimensional subspaces   for 
displacement, fluid pressure,  rotations, and total 
 pressure; as follows 
\begin{align*}	
&\bV_h:=\{\bv_h \in \mathbf{C}(\overline{\Omega})\cap\bV: \bv_h|_{K} \in \bbP_{k+1}(K)^d,\ \forall K\in \cT_h\},\quad 	
\rQ_h^{\mathrm{P}}:=\{q_h^{\mathrm{P}}\in\mathrm{C}(\overline{\Omega^{\mathrm{P}}}): q_h^{\mathrm{P}}|_K\in\bbP_{k+1}(K),\ \forall K\in\cT_{h}\},\\
&\bW_h^{\mathrm{P}}:=\{\btheta_h^{\mathrm{P}}\in\mathbf{L}^2(\Omega^{\mathrm{P}}): \btheta_h^{\mathrm{P}}|_K\in\bbP_k(K)^d,\ \forall K\in\cT_{h}\}, \quad
\bW_h^{\mathrm{E}}:=\{\btheta_h^{\mathrm{E}}\in\mathbf{L}^2(\Omega^{\mathrm{E}}): \btheta_h^{\mathrm{E}}|_K\in\bbP_k(K)^d,\ \forall K\in\cT_{h}\},\\ 
 & \rZ_h^{\mathrm{P}}:=\{\psi_h^{\mathrm{P}}\in\mathrm{L}^2(\Omega^{\mathrm{P}}): \psi_h^{\mathrm{P}}|_K\in\bbP_k(K),\ \forall K\in\cT_{h}\},\quad 
  \rZ_h^{\mathrm{E}}:=\{q_h^{\mathrm{E}}\in\mathrm{L}^2(\Omega^{\mathrm{E}}): q_h^{\mathrm{E}}|_K\in\bbP_k(K),\ \forall K\in\cT_{h}\}.
\end{align*}
Define $\overrightarrow{\bomega}_h:= \{\bomega^{\mathrm{P}}_h,\phi^{\mathrm{P}}_h,\bomega^{\mathrm{E}}_h,p^{\mathrm{E}}_h\}\in \bW_h^{\mathrm{P}}\times \rZ_h^{\mathrm{P}}\times \bW_h^{\mathrm{E}}\times \rZ_h^{\mathrm{E}} := \bH_h$. The discrete weak formulation of the rotation based elasticity is read as: find $(\overrightarrow{\bomega}_h,\bu_h,p^{\mathrm{P}}_h)\in \bH_h\times\bV_h\times\rQ_h^{\mathrm{P}}$ such that
 \begin{align}\label{weakdisEP}
 B_{\mathrm{I}}((\overrightarrow{\bomega}_h,\bu_h,p^{\mathrm{P}}_h),(\overrightarrow{\btheta},\bv,q^{\mathrm{P}}))=F(\bv)+G(q),
 \end{align}
 for all $(\overrightarrow{\btheta},\bv,q^{\mathrm{P}})\in  \bH_h\times\bV_h\times\rQ_h^{\mathrm{P}}$. 
For each $k\ge0$, the modified (stablized) discrete weak formulation of the rotation based elasticity is read as: find $(\overrightarrow{\bomega}_h,\bu_h,p^{\mathrm{P}}_h)\in  \bH_h\times\bV_h\times\rQ_h^{\mathrm{P}}$ such that
 \begin{align}\label{weakdisEPstab}
 B_{\mathrm{I}}((\overrightarrow{\bomega}_h,\bu_h,p^{\mathrm{P}}_h),(\overrightarrow{\btheta},\bv,q^{\mathrm{P}}))+\!\!\sum_{e\in\mathcal{E}(\mathcal{T}_h)\cap\Omega^{\mathrm{E}}}\frac{h_e}{\mu^{\mathrm{E}}}\int_{e}[\![p^{\mathrm{E}}_h]\!][\![q^{\mathrm{E}}]\!]+\sum_{e\in\mathcal{E}(\mathcal{T}_h)\cap\Omega^{\mathrm{P}}}\frac{h_e}{\mu^{\mathrm{P}}}\int_{e}[\![\phi_h^{\mathrm{P}}]\!][\![\psi^{\mathrm{P}}]\!]=F(\bv)+G(q),
 \end{align}
 for all $(\overrightarrow{\btheta},\bv,q^{\mathrm{P}})\in  \bH_h\times\bV_h\times\rQ_h^{\mathrm{P}}$.


\subsection{\textit{A posteriori} error analysis}
%
Let  $\estE^2$,  $ \estP^2$ and  $\estint^2$ be the elasticity estimator (\emph{cf.} \eqref{estosc}), the poroelasticity estimator (\emph{cf.} \eqref{estosc1})
and the interface estimator (see below), respectively. Then we define 
\[
\Xi^2 :=\sum_{K\in\mathcal{T}_h\cap\Omega^{\mathrm{E}}} \estE^2+\sum_{K\in\mathcal{T}_h\cap\Omega^{\mathrm{P}}} \estP^2+\sum_{e\in \mathcal{E}(\mathcal{T}_h)\cap\Sigma} 
\estint^2, 
\]
where 
\[
\estint^2:= h_e (\mu^{\mathrm{E}}+\mu^{\mathrm{P}})^{-1}\|\mathbf{R}_{\Sigma}\|_{0,e}^2+h_e \xi \kappa^{-1}\|\widehat{R}_{\Sigma}\|_{0,e}^2,
\]
and 
\[
\mathbf{R}_{\Sigma}:= \{\sqrt{\mu^{\mathrm{P}}} \bomega^{\mathrm{P}}_h\times \nn+\phi^{\mathrm{P}}_h \nn -\sqrt{\mu^{\mathrm{E}}} \bomega^{\mathrm{E}}_h\times \nn- p^{\mathrm{E}}_h\nn\}, \quad 
\widehat{R}_{\Sigma}:=\{\kappa\xi^{-1}(\nabla p^{\mathrm{P}}_h-\rho \gg)\cdot \nn\}.
\]
Next we define the global data oscillations term $\Upsilon$ as
\begin{align*}
\Upsilon^2:=\sum_{K\in\mathcal{T}_h\cap\Omega^{\mathrm{E}}}\UpsilonE^2+\sum_{K\in\mathcal{T}_h\cap\Omega^{\mathrm{P}}}\UpsilonP^2,
\end{align*}
where $\UpsilonE$ and $\UpsilonP$ are the local data oscillations for elasticity and poroelasticity, respectively.
\subsubsection{Reliability estimate}
In this section, we prove the reliability bound for the interface estimator. 
\begin{theorem}[Reliability for the transmission problem]
Let $(\overrightarrow{\bomega},\bu,p^{\mathrm{P}})$ and $(\overrightarrow{\bomega}_h,\bu_h,p^{\mathrm{P}}_h)$ be the solutions of the weak formulations (\ref{weakEP}) and (\ref{weakdisEP}) (or (\ref{weakdisEPstab})), respectively. Then the following reliability bound holds
\[
\VERT(\overrightarrow{\bomega}-\overrightarrow{\bomega}_h,\bu-\bu_h,p^{\mathrm{P}}-p^{\mathrm{P}}_h)\VERT\le C_{\mathrm{rel}} (\Xi+\Upsilon),
\]
where $C_{\mathrm{rel}}>0$ is a positive constant independent of mesh size and parameters.
\end{theorem}
\begin{proof}
Since $(\overrightarrow{\bomega}-\overrightarrow{\bomega}_h,\bu-\bu_h,p^{\mathrm{P}}-p_h^{\mathrm{P}})\in \bH\times \bV\times\rQ^{\mathrm{P}}$, then from stability theorem, we have  
\[
C_2\VERT(\overrightarrow{\bomega}-\overrightarrow{\bomega}_h,\bu-\bu_h,p^{\mathrm{P}}-p_h^{\mathrm{P}})\VERT^2\le B_{\mathrm{I}}((\overrightarrow{\bomega}-\overrightarrow{\bomega}_h,\bu-\bu_h,p^{\mathrm{P}}-p_h^{\mathrm{P}}),(\overrightarrow{\btheta},\bv,q^{\mathrm{P}})),
\] 
with $\VERT(\overrightarrow{\btheta},\bv,q^{\mathrm{P}})\VERT\le C_1 \VERT(\overrightarrow{\bomega}-\overrightarrow{\bomega}_h,\bu-\bu_h,p^{\mathrm{P}}-p_h^{\mathrm{P}})\VERT$. 
And from the definition of the continuous and discrete weak forms, it follows that:
\begin{align*}
&B_{\mathrm{I}}((\overrightarrow{\bomega}-\overrightarrow{\bomega}_h,\bu-\bu_h,p^{\mathrm{P}}-p_h^{\mathrm{P}}),(\overrightarrow{\btheta},\bv,q^{\mathrm{P}}))\\
&=-(\ff^{\mathrm{P}}-\ff_h^{\mathrm{P}},\bv-\bv_h)_{0,\Omega^{\mathrm{P}}}-(\ff^{\mathrm{E}}-\ff_h^{\mathrm{E}},\bv-\bv_h)_{0,\Omega^{\mathrm{E}}}-\frac{\rho}{\xi}(\kappa\gg, \nabla (q-q_h))_{0,\Omega^{\mathrm{P}}} +(s^{\mathrm{P}}-s^{\mathrm{P}}_h,q-q_h)_{0,\Omega^{\mathrm{P}}} \\
&\quad-(\ff_h^{\mathrm{P}},\bv-\bv_h)_{0,\Omega^{\mathrm{P}}}-(\ff_h^{\mathrm{E}},\bv-\bv_h)_{0,\Omega^{\mathrm{E}}} +(s^{\mathrm{P}}_h,q-q_h)_{0,\Omega^{\mathrm{P}}} 
-B_{\mathrm{I}}((\overrightarrow{\bomega}_h,\bu_h,p_h^{\mathrm{P}}),(\overrightarrow{\btheta},\bv-\bv_h,q^{\mathrm{P}}-q^{\mathrm{P}}_h)).
\end{align*}
Applying integration by parts, Cauchy-Schwarz inequality and approximation results,  yields 
\[
B_{\mathrm{I}}((\overrightarrow{\bomega}-\overrightarrow{\bomega}_h,\bu-\bu_h,p^{\mathrm{P}}-p_h^{\mathrm{P}}),(\overrightarrow{\btheta},\bv,q^{\mathrm{P}}))\le C (\Xi+\Upsilon) \VERT(\overrightarrow{\btheta},\bv,q^{\mathrm{P}})\VERT.\]
\end{proof}
\subsubsection{Efficiency bound}
\begin{lemma}\label{eleEP1}
The following estimates are satisfied 
\[
\Theta  \lesssim \VERT(\overrightarrow{\bomega}-\overrightarrow{\bomega}_h,\bu-\bu_h,p^{\mathrm{P}}-p^{\mathrm{P}}_h)\VERT+\Upsilon,\qquad 
\Psi \lesssim \VERT(\overrightarrow{\bomega}-\overrightarrow{\bomega}_h,\bu-\bu_h,p^{\mathrm{P}}-p^{\mathrm{P}}_h)\VERT+\Upsilon.
\]
\end{lemma}
\begin{proof}
The first bound follows from Theorem \ref{poroeff}, while the second one follows from Theorem \ref{elasteff}.
\end{proof}
\begin{lemma}\label{eleEP3}
There holds:
{\small\begin{align*}
&(\sum_{e\in\Sigma}h_e(\mu^{\mathrm{E}}+\mu^{\mathrm{P}})^{-1} \|\mathbf{R}_{\Sigma}\|_{0,e}^2)^{1/2}\lesssim \sum_{e\in\Sigma}\Big(\sum_{K\in P_e\cap\Omega^{\mathrm{E}}}((\mu^{\mathrm{E}})^{-1/2}h_K\|\ff^{\mathrm{E}}-\ff_h^{\mathrm{E}}\|_{0,K}+(\mu^{\mathrm{E}})^{-1/2}\|p^{\mathrm{E}}-p_h^{\mathrm{E}}\|_{0,K}+\|\bomega^{\mathrm{E}}-\bomega_h^{\mathrm{E}}\|_{0,K})\\
&\quad+\sum_{K\in P_e\cap\Omega^{\mathrm{P}}}((\mu^{\mathrm{P}})^{-1/2}h_K\|\ff^{\mathrm{P}}-\ff_h^{\mathrm{P}}\|_{0,K}+(\mu^{\mathrm{P}})^{-1/2}\|\phi^{\mathrm{P}}-\phi_h^{\mathrm{P}}\|_{0,K}+\|\bomega^{\mathrm{P}}-\bomega_h^{\mathrm{P}}\|_{0,K})\Big).
\end{align*}}
\end{lemma}
\begin{proof}
For each $e\in \mathcal{E}(\mathcal{T}_h)\cap\Sigma$, we locally define 
$ \bzeta_e=(\mu^{\mathrm{E}}+\mu^{\mathrm{P}})^{-1}h_e\mathbf{R}_{\Sigma}b_e.$ 
Using (\ref{edgee1}) implies 
\[
h_e(\mu^{\mathrm{E}}+\mu^{\mathrm{P}})^{-1}\|\mathbf{R}_{\Sigma}\|_{0,e}^2\lesssim \int_e \mathbf{R}_{\Sigma}\cdot ((\mu^{\mathrm{E}}+\mu^{\mathrm{P}})^{-1}h_e\mathbf{R}_{\Sigma}b_e) =\int_{e}\mathbf{R}_{\Sigma}\cdot \bzeta_e.
\]
Integration by parts gives
{\small
 \begin{align*}
 \int_e &\sqrt{\mu^{\mathrm{E}}}(\bomega_h^{\mathrm{E}}-\bomega^{\mathrm{E}})\times\nn +(p_h^{\mathrm{E}}-p^{\mathrm{E}})\nn  \cdot \bzeta_e -\sqrt{\mu^{\mathrm{P}}}(\bomega_h^{\mathrm{P}}-\bomega^{\mathrm{P}})\times\nn +(\phi_h^{\mathrm{P}}-\phi^{\mathrm{P}})\nn \cdot \bzeta_e\\
 &=\sum_{K\in P_e\cap\Omega^{\mathrm{E}}}\int_K(\sqrt{\mu^{\mathrm{E}}}\curl(\bomega_h^{\mathrm{E}}-\bomega^{\mathrm{E}})+\nabla (p_h^{\mathrm{E}}-p^{\mathrm{E}})\cdot \bzeta_e+\sum_{K\in P_e\cap\Omega^{\mathrm{E}}}\int_K(\sqrt{\mu^{\mathrm{E}}}(\bomega_h^{\mathrm{E}}-\bomega^{\mathrm{E}})\cdot  \curl \bzeta_e+ (p_h^{\mathrm{E}} -p^{\mathrm{E}})\nabla\cdot  \bzeta_e) \\
&\quad + \sum_{K\in P_e\cap\Omega^{\mathrm{P}}}\int_K(\sqrt{\mu^{\mathrm{P}}}\curl(\bomega_h^{\mathrm{P}}-\bomega^{\mathrm{P}})+\nabla (\phi_h^{\mathrm{P}}-\phi^{\mathrm{P}})\cdot \bzeta_e +\sum_{K\in P_e\cap\Omega^{\mathrm{P}}}\int_K(\sqrt{\mu^{\mathrm{P}}}(\bomega_h^{\mathrm{P}}-\bomega^{\mathrm{P}})\cdot  \curl \bzeta_e+ (\phi_h^{\mathrm{P}} -\phi^{\mathrm{P}})\nabla\cdot  \bzeta_e).
 \end{align*}}
Recall that $\ff^{\mathrm{P}}-\sqrt{\mu^{\mathrm{P}}}\curl\bomega^{\mathrm{P}} -\nabla p^{\mathrm{P}}=\cero|_K$ and $\ff^{\mathrm{E}}-\sqrt{\mu^{\mathrm{E}}}\curl\bomega^{\mathrm{E}} -\nabla p^{\mathrm{E}}=\cero|_K$. Then, we have
\begin{align*}
\frac{h_e}{\mu^{\mathrm{E}}+\mu^{\mathrm{P}}}\|\mathbf{R}_{\Sigma}\|_{0,e}^2 
&\lesssim \sum_{K\in P_e\cap\Omega^{\mathrm{E}}}\int_{K} \left((\ff_h^{\mathrm{E}}-\ff^{\mathrm{E}})\cdot \bzeta_e +\sqrt{\mu^{\mathrm{E}}}\int_K (\bomega_h^{\mathrm{E}}-\bomega) \cdot \curl \bzeta_e +\int_K (p_h^{\mathrm{E}}-p^{\mathrm{E}})\nabla\cdot \bzeta \right)\\
&\quad+ \sum_{K\in P_e\cap\Omega^{\mathrm{E}}}\int_{K} \mathbf{R}_1^{\mathrm{E}} \cdot \bzeta_e+ \sum_{K\in P_e\cap\Omega^{\mathrm{P}}}\int_{K} \mathbf{R}_1^{\mathrm{P}} \cdot \bzeta_e \\
&\quad+ \sum_{K\in P_e\cap\Omega^{\mathrm{P}}}\int_{K} \left((\ff_h^{\mathrm{P}}-\ff^{\mathrm{P}})\cdot \bzeta_e +\sqrt{\mu^{\mathrm{P}}}\int_K (\bomega_h^{\mathrm{P}}-\bomega) \cdot \curl \bzeta_e +\int_K (\phi_h^{\mathrm{P}}-\phi^{\mathrm{P}})\nabla\cdot \bzeta \right).
\end{align*}
Next we can apply Cauchy-Schwarz inequality, leading to 
\begin{align*}
\frac{h_e}{\mu^{\mathrm{E}}+\mu^{\mathrm{P}}}\|\mathbf{R}_e\|_{0,e}^2 & \lesssim 
\sum_{K\in P_e\cap\Omega^{\mathrm{E}}}((\mu^{\mathrm{E}})^{-1/2}h_K\|\ff^{\mathrm{E}}-\ff_h^{\mathrm{E}}\|_{0,K}+(\mu^{\mathrm{E}})^{-1/2}\|p^{\mathrm{E}}-p_h^{\mathrm{E}}\|_{0,K}+\|\bomega^{\mathrm{E}}-\bomega_h^{\mathrm{E}}\|_{0,K})\times\\
&\qquad\qquad\quad((\mu^{\mathrm{E}})^{1/2}\|\bnabla \bzeta\|_{0,K}+(\mu^{\mathrm{E}})^{1/2}h_K^{-1}\|\bzeta\|_{0,K})\\
&\quad+\sum_{K\in P_e\cap\Omega^{\mathrm{P}}}((\mu^{\mathrm{P}})^{-1/2}h_K\|\ff^{\mathrm{P}}-\ff_h^{\mathrm{P}}\|_{0,K}+(\mu^{\mathrm{P}})^{-1/2}\|\phi^{\mathrm{P}}-\phi_h^{\mathrm{P}}\|_{0,K}+\|\bomega^{\mathrm{P}}-\bomega_h^{\mathrm{P}}\|_{0,K})\times\\
&\qquad\qquad\quad((\mu^{\mathrm{P}})^{1/2}\|\bnabla \bzeta\|_{0,K}+(\mu^{\mathrm{P}})^{1/2}h_K^{-1}\|\bzeta\|_{0,K}).
\end{align*}
And the sought estimate is then a consequence of the bounds 
\begin{align*}
(\mu^{\mathrm{E}})^{1/2}\|\bnabla \bzeta\|_{0,K}+(\mu^{\mathrm{E}})^{1/2}h_K^{-1}\|\bzeta\|_{0,K}
&\lesssim (\mu^{\mathrm{E}})^{1/2} h_K^{-1}\|\bzeta\|_{0,K}\lesssim h_e^{1/2}(\mu^{\mathrm{E}}+\mu^{\mathrm{P}})^{-1/2}\|\mathbf{R}_e\|_{0,e},\\
(\mu^{\mathrm{P}})^{1/2}\|\bnabla \bzeta\|_{0,K}+(\mu^{\mathrm{P}})^{1/2}h_K^{-1}\|\bzeta\|_{0,K}
&\lesssim (\mu^{\mathrm{P}})^{1/2} h_K^{-1}\|\bzeta\|_{0,K}\lesssim h_e^{1/2}(\mu^{\mathrm{E}}+\mu^{\mathrm{P}})^{-1/2}\|\mathbf{R}_e\|_{0,e}.
\end{align*}
\end{proof}
\begin{lemma}\label{eleEP4}
The following bound holds 
\begin{align*}
\Big(\sum_{e\in\mathcal{E}(\mathcal{T}_h)\cap\Sigma}\estint^2\Big)^{1/2}\lesssim \VERT(\overrightarrow{\bomega}-\overrightarrow{\bomega}_h,\bu-\bu_h,p^{\mathrm{P}}-p^{\mathrm{P}}_h)\VERT+\Upsilon.
\end{align*}
\end{lemma}
\begin{proof}
It follows straightforwardly from Lemma \ref{lemP6}. 
\end{proof}
\begin{theorem}[Efficiency estimate]
Let $(\overrightarrow{\bomega},\bu,p^{\mathrm{P}})$ and $(\overrightarrow{\bomega}_h,\bu_h,p^{\mathrm{P}}_h)$ be the solutions to (\ref{weakEP}) and (\ref{weakdisEP}) (or (\ref{weakdisEPstab})), respectively. Then, the following reliability bound holds
\[
\Xi\le C_{\mathrm{eff}} (\VERT(\overrightarrow{\bomega}-\overrightarrow{\bomega}_h,\bu-\bu_h,p^{\mathrm{P}}-p^{\mathrm{P}}_h)\VERT+\Upsilon),
\]
where $C_{\mathrm{eff}}>0$ is a   constant independent of $h$ and of the sensible model parameters.
\end{theorem}
\begin{proof}
The bound results from combining Lemmas \ref{eleEP1}--\ref{eleEP4}.
\end{proof}
\section{Computational examples}\label{sec:results}

The accuracy of the three finite element discretisations and the robustness of the corresponding  \textit{a posteriori} error estimators will be demonstrated in this section. As usual, such robustness is quantified in terms of the effectivity index of a given computable indicator $\Phi \in \{\Theta,\Psi,\Xi\}$, i.e., the ratio between the total actual error and the estimated error   
\[ \texttt{eff}(\Phi) =  (\texttt{e}^2_{\bu} + \texttt{e}^2_p + \dots)^{1/2}/ \Phi_K,\]
and $\texttt{eff}$ is expected to remain constant independently of the number of degrees of freedom associated with each mesh refinement. The direct solver UMFPACK is used for all linear systems, and the algorithms are implemented in the FEniCS library \cite{alnaes_ans15}, using multiphenics \cite{multiphenics} for the handling of subdomains and incorporation of restricted finite element spaces.  

\begin{table}[t]
\setlength{\tabcolsep}{3.5pt}
 \begin{center}
{\footnotesize  \begin{tabular}{|rccccccc|}
\hline 
DoFs & $h$ & $\texttt{e}_{\bomega}$ & $\texttt{r}_{\bomega}$ & $\texttt{e}_{\bu}$ & $\texttt{r}_{\bu}$ & \texttt{e} & $\texttt{eff}(\Theta)$ $\vphantom{\int^X_X}$\\
\hline 
\multicolumn{8}{|c|}{$E = 1$, $\nu = 0.25$, $ k = 0 $ $\vphantom{\int^X_X}$}\\
\hline 
   114 & 0.3536 & 2.14e+0 & 0.00 & 2.64e+0 & 0.00 & 3.40e+0 & 0.249\\
   418 & 0.1768 & 1.11e+0 & 0.95 & 1.40e+0 & 0.92 & 1.79e+0 & 0.248\\
  1602 & 0.0884 & 5.61e-01 & 0.99 & 7.07e-01 & 0.98 & 9.02e-01 & 0.246\\
  6274 & 0.0442 & 2.81e-01 & 1.00 & 3.54e-01 & 1.00 & 4.52e-01 & 0.245\\
 24834 & 0.0221 & 1.40e-01 & 1.00 & 1.77e-01 & 1.00 & 2.26e-01 & 0.244\\
 98818 & 0.0110 & 7.02e-02 & 1.00 & 8.85e-02 & 1.00 & 1.13e-01 & 0.244\\
\hline 
\multicolumn{8}{|c|}{$E = 1$, $\nu = 0.25$, $ k = 1 $ $\vphantom{\int^X_X}$}\\
\hline
   354 & 0.3536 & 5.33e-01 & 1.63 & 7.65e-01 & 1.51 & 9.32e-01 & 0.146\\
  1346 & 0.1768 & 1.43e-01 & 1.90 & 2.09e-01 & 1.87 & 2.53e-01 & 0.151\\
  5250 & 0.0884 & 3.67e-02 & 1.96 & 5.17e-02 & 2.02 & 6.34e-02 & 0.148\\
 20738 & 0.0442 & 9.24e-03 & 1.99 & 1.28e-02 & 2.02 & 1.58e-02 & 0.146\\
 82434 & 0.0221 & 2.32e-03 & 2.00 & 3.18e-03 & 2.01 & 3.93e-03 & 0.146\\
328706 & 0.0110 & 5.79e-04 & 2.00 & 7.93e-04 & 2.00 & 9.82e-04 & 0.146\\
 \hline 
\multicolumn{8}{|c|}{$E = 10^5$, $\nu = 0.499$, $ k = 0 $ $\vphantom{\int^X_X}$}\\
\hline
   114 & 0.3536 & 6.43e+2 & 0.00 & 7.24e+2 & 0.00 & 9.68e+2 & 0.222\\
   418 & 0.1768 & 3.19e+2 & 1.01 & 4.02e+2 & 0.85 & 5.14e+2 & 0.247\\
  1602 & 0.0884 & 1.61e+2 & 0.99 & 2.04e+2 & 0.98 & 2.60e+2 & 0.245\\
  6274 & 0.0442 & 8.08e+1 & 1.00 & 1.02e+2 & 1.00 & 1.30e+2 & 0.244\\
 24834 & 0.0221 & 4.04e+1 & 1.00 & 5.11e+1 & 1.00 & 6.51e+1 & 0.244\\
 98818 & 0.0110 & 2.02e+1 & 1.00 & 2.55e+1 & 1.00 & 3.26e+1 & 0.244\\
  \hline 
\multicolumn{8}{|c|}{$E = 10^5$, $\nu = 0.499$, $ k = 1 $ $\vphantom{\int^X_X}$}\\
\hline
   354 & 0.3536 & 2.37e+2 & 0.00 & 4.18e+2 & 0.00 & 4.81e+2 & 0.172\\
  1346 & 0.1768 & 4.14e+1 & 2.52 & 5.94e+1 & 2.82 & 7.24e+1 & 0.150\\
  5250 & 0.0884 & 1.06e+1 & 1.97 & 1.49e+1 & 1.99 & 1.83e+1 & 0.148\\
 20738 & 0.0442 & 2.67e+0 & 1.99 & 3.68e+0 & 2.02 & 4.55e+0 & 0.146\\
 82434 & 0.0221 & 6.68e-01 & 2.00 & 9.17e-01 & 2.01 & 1.13e+0 & 0.146\\
328706 & 0.0110 & 1.67e-01 & 2.00 & 2.29e-01 & 2.00 & 2.83e-01 & 0.145\\
 \hline
  \end{tabular}}
  
  \smallskip
  \caption{Example 1A: Errors, convergence rates, and effectivity indexes under uniform mesh refinement. Smooth manufactured solutions for the rotation-based elasticity problem.}
\label{table:ex01A}
 \end{center}
\end{table}

We start with a simple case of manufactured polynomial solutions on $\Omega = (0,1)^2$ where both displacement and fluid pressure vanish on $\partial\Omega$ 
\[
p(x,y) = xy(1-x)(a-y), \quad 
\bu(x,y) = \begin{pmatrix} \pi \sin^2(\pi x)\sin(\pi y)\cos(\pi y)+p(x,y)/2\lambda \\-\pi \sin(\pi x)\cos(\pi y)\sin^2(\pi y)+p(x,y)/2\lambda\end{pmatrix}.
\]
To ensure the zero boundary condition for displacement, we choose $a=1$. In the interface problem, we choose $a=0.5$ for fluid pressure.
Unless specified otherwise, all parameters (except the Poisson ratio) are taken equal to 1. For the transmission problem the interface is the horizontal segment located on $y = \frac{1}{2}$, and the porous domain is below the interface. 
A sequence of successively refined uniform meshes is constructed and exact and estimated errors between these closed-form solutions and the finite element approximations (in this case focusing on the first and second-order schemes, with $k=0$ and $k=1$) are computed. The results are collected in Tables~\ref{table:ex01A}, \ref{table:ex01B}, \ref{table:ex01C}, where the convergence rates are computed as 
\begin{equation}
\label{eq:uniform-rate}
\texttt{r}_{(\cdot)}  =\log(e_{(\cdot)}/\tilde{e}_{(\cdot)})[\log(h/\tilde{h})]^{-1},
\end{equation}
where $e,\tilde{e}$ denote errors generated on two
consecutive  meshes of size $h$ and~$\tilde{h}$. 

\begin{table}[t!]
\setlength{\tabcolsep}{3.5pt}
 \begin{center}
{\footnotesize  \begin{tabular}{|rccccccccc|}
\hline 
DoFs & $h$ & $\texttt{e}_{\vomega}$ & $\texttt{r}_{\vomega}$ & $\texttt{e}_{\bu}$ & $\texttt{r}_{\bu}$  & $\texttt{e}_{p}$ & $\texttt{r}_{p}$ & \texttt{e} & $\texttt{eff}(\Psi)$ $\vphantom{\int^X_X}$\\
\hline 
\multicolumn{10}{|c|}{$E = 1$, $\nu = 0.25$, $\kappa = 1$, $ k = 0 $ $\vphantom{\int^X_X}$}\\
\hline 
   139 & 0.3536 & 2.14e+0 & -- & 2.64e+0 & -- & 5.51e-02 & -- & 3.40e+0 & 0.249\\
   499 & 0.1768 & 1.11e+0 & 0.95 & 1.40e+0 & 0.92 & 2.91e-02 & 0.92 & 1.79e+0 & 0.248\\
  1891 & 0.0884 & 5.61e-01 & 0.99 & 7.07e-01 & 0.98 & 1.50e-02 & 0.96 & 9.02e-01 & 0.246\\
  7363 & 0.0442 & 2.81e-01 & 1.00 & 3.54e-01 & 1.00 & 7.57e-03 & 0.98 & 4.52e-01 & 0.245\\
 29059 & 0.0221 & 1.40e-01 & 1.00 & 1.77e-01 & 1.00 & 3.80e-03 & 0.99 & 2.26e-01 & 0.244\\
115459 & 0.0110 & 7.02e-02 & 1.00 & 8.85e-02 & 1.00 & 1.90e-03 & 1.00 & 1.13e-01 & 0.244\\
\hline 
\multicolumn{10}{|c|}{$E = 1$, $\nu = 0.25$,  $\kappa = 1$, $ k = 1 $ $\vphantom{\int^X_X}$}\\
\hline
   435 & 0.3536 & 5.33e-01 & -- & 7.65e-01 & -- & 7.19e-03 & -- & 9.32e-01 & 0.146\\
  1635 & 0.1768 & 1.43e-01 & 1.90 & 2.09e-01 & 1.87 & 1.96e-03 & 1.88 & 2.53e-01 & 0.151\\
  6339 & 0.0884 & 3.67e-02 & 1.96 & 5.17e-02 & 2.02 & 5.11e-04 & 1.94 & 6.34e-02 & 0.148\\
 24963 & 0.0442 & 9.24e-03 & 1.99 & 1.28e-02 & 2.02 & 1.30e-04 & 1.97 & 1.58e-02 & 0.146\\
 99075 & 0.0221 & 2.32e-03 & 2.00 & 3.18e-03 & 2.01 & 3.29e-05 & 1.99 & 3.93e-03 & 0.146\\
394755 & 0.0110 & 5.79e-04 & 2.00 & 7.93e-04 & 2.00 & 8.26e-06 & 1.99 & 9.82e-04 & 0.146\\
 \hline 
\multicolumn{10}{|c|}{$E = 10^5$, $\nu = 0.499$,  $\kappa = 1$, $ k = 0 $ $\vphantom{\int^X_X}$}\\
\hline
   139 & 0.3536 & 6.43e+2 & -- & 7.24e+2 & -- & 5.10e-02 & -- & 9.68e+2 & 0.222\\
   499 & 0.1768 & 3.19e+2 & 1.01 & 4.02e+2 & 0.85 & 2.85e-02 & 0.84 & 5.14e+2 & 0.247\\
  1891 & 0.0884 & 1.61e+2 & 0.99 & 2.04e+2 & 0.98 & 1.49e-02 & 0.94 & 2.60e+2 & 0.245\\
  7363 & 0.0442 & 8.08e+1 & 1.00 & 1.02e+2 & 1.00 & 7.55e-03 & 0.98 & 1.30e+2 & 0.244\\
 29059 & 0.0221 & 4.04e+1 & 1.00 & 5.11e+1 & 1.00 & 3.80e-03 & 0.99 & 6.51e+1 & 0.244\\
115459 & 0.0110 & 2.02e+1 & 1.00 & 2.55e+1 & 1.00 & 1.90e-03 & 1.00 & 3.26e+1 & 0.244\\
  \hline 
\multicolumn{10}{|c|}{$E = 10^5$, $\nu = 0.499$,  $\kappa = 1$, $ k = 1 $ $\vphantom{\int^X_X}$}\\
\hline
   435 & 0.3536 & 2.37e+2 & -- & 4.18e+2 & --& 7.16e-03 & -- & 4.81e+2 & 0.172\\
  1635 & 0.1768 & 4.14e+1 & 2.52 & 5.94e+1 & 2.82 & 1.96e-03 & 1.87 & 7.24e+1 & 0.150\\
  6339 & 0.0884 & 1.06e+1 & 1.97 & 1.49e+1 & 1.99 & 5.11e-04 & 1.94 & 1.83e+1 & 0.148\\
 24963 & 0.0442 & 2.67e+0 & 1.99 & 3.68e+0 & 2.02 & 1.30e-04 & 1.97 & 4.55e+0 & 0.146\\
 99075 & 0.0221 & 6.68e-01 & 2.00 & 9.17e-01 & 2.01 & 3.29e-05 & 1.99 & 1.13e+0 & 0.146\\
394755 & 0.0110 & 1.67e-01 & 2.00 & 2.29e-01 & 2.00 & 8.26e-06 & 1.99 & 2.83e-01 & 0.145\\
 \hline
 \multicolumn{10}{|c|}{$E = 10^5$, $\nu = 0.499$,  $\kappa = 10^{-12}$, $ k = 0 $ $\vphantom{\int^X_X}$}\\
\hline
   139 & 0.3536 & 6.43e+2 & -- & 7.24e+2 & -- & 1.81e-03 & -- & 9.68e+2 & 0.222\\
   499 & 0.1768 & 3.19e+2 & 1.01 & 4.02e+2 & 0.85 & 4.62e-04 & 1.97 & 5.14e+2 & 0.247\\
  1891 & 0.0884 & 1.61e+2 & 0.99 & 2.04e+2 & 0.98 & 1.18e-04 & 1.96 & 2.60e+2 & 0.245\\
  7363 & 0.0442 & 8.08e+1 & 1.00 & 1.02e+2 & 1.00 & 3.02e-05 & 1.97 & 1.30e+2 & 0.244\\
 29059 & 0.0221 & 4.04e+1 & 1.00 & 5.11e+1 & 1.00 & 7.72e-06 & 1.97 & 6.51e+1 & 0.244\\
115459 & 0.0110 & 2.02e+1 & 1.00 & 2.55e+1 & 1.00 & 2.00e-06 & 1.95 & 3.26e+1 & 0.244\\
  \hline
 \multicolumn{10}{|c|}{$E = 10^5$, $\nu = 0.499$,  $\kappa = 10^{-12}$, $ k = 1 $ $\vphantom{\int^X_X}$}\\
\hline
   435 & 0.3536 & 2.37e+2 & -- & 4.18e+2 & -- & 9.97e-04 & -- & 4.81e+2 & 0.172\\
  1635 & 0.1768 & 4.14e+1 & 2.52 & 5.94e+1 & 2.82 & 3.60e-05 & 4.79 & 7.24e+1 & 0.150\\
  6339 & 0.0884 & 1.06e+1 & 1.97 & 1.49e+1 & 1.99 & 5.40e-06 & 2.74 & 1.83e+1 & 0.148\\
 24963 & 0.0442 & 2.67e+0 & 1.99 & 3.68e+0 & 2.02 & 1.15e-06 & 2.23 & 4.55e+0 & 0.146\\
 99075 & 0.0221 & 6.68e-01 & 2.00 & 9.17e-01 & 2.01 & 2.73e-07 & 2.08 & 1.13e+0 & 0.146\\
394755 & 0.0110 & 1.67e-01 & 2.00 & 2.29e-01 & 2.00 & 6.69e-08 & 2.03 & 2.83e-01 & 0.145\\
\hline
  \end{tabular}}
  
  \smallskip
  \caption{Example 1B: Errors (combining rotation and total pressure into $\vomega$), convergence rates, and effectivity indexes under uniform mesh refinement. Smooth manufactured solutions for the rotation-based Biot problem.}
\label{table:ex01B}
 \end{center}
\end{table}

The expected $O(h^{k+1})$ convergence is observed for all fields in their respective norms, accordingly to the theory from \cite{anaya_cmame19,anaya_sisc20}, 
and the effectivity index is close to constant for all mesh refinements. This same behaviour is seen even when the elastic or the poroelastic material is nearly incompressible (setting $E= 10^5, \nu = 0.499$) and when the poroelastic material it is nearly impermeable (setting $\kappa = 10^{-12}$), and we also note that the effectivity index is slightly modified, but it is still constant and not affected by the different parameter scaling, again confirming the robustness of the estimators. The variation in efficiency is natural as our analysis only focuses on $h$-adaptivity based \textit{a posteriori} error estimation (and an extension to $hp$-adaptivity based \textit{a posteriori} error estimators might help to overcome such a variation).

\begin{table}[t!]
\setlength{\tabcolsep}{3.5pt}
 \begin{center}
 {\footnotesize \begin{tabular}{|rcccccccccccccc|}
\hline 
DoFs  & $\texttt{e}_{\bomega^{\mathrm{P}}}$ & $\texttt{r}_{\bomega^{\mathrm{P}}}$& $\texttt{e}_{\phi^{\mathrm{P}}}$ & $\texttt{r}_{\phi^{\mathrm{P}}}$  & $\texttt{e}_{p^{\mathrm{P}}}$ & $\texttt{r}_{p^{\mathrm{P}}}$ & $\texttt{e}_{\bu}$ & $\texttt{r}_{\bu}$ & $\texttt{e}_{\bomega^{\mathrm{E}}}$ & $\texttt{r}_{\bomega^{\mathrm{E}}}$ & $\texttt{e}_{p^{\mathrm{E}}}$ & $\texttt{r}_{p^{\mathrm{E}}}$ & \texttt{e} & $\texttt{eff}(\Xi)$ $\vphantom{\int^X_X}$\\
\hline 
\multicolumn{15}{|c|}{$E^{\mathrm{E}} = E^{\mathrm{P}} = 1$, $\nu^{\mathrm{E}} = 0.25$, $\nu^{\mathrm{P}} = 0.25$, $\kappa = 1$, $k=0$  $\vphantom{\int^X_X}$}\\
\hline 
   139 & 1.5258 & --   & 0.1982  & --   & 3.28e-02 & 0.00 & 2.6433 & --   & 1.48e+0 & --   & 1.86e-01 & --   & 3.4036 & 0.281\\
   499 & 0.7902 & 0.95 & 0.0833  & 1.25 & 1.33e-02 & 1.30 & 1.3992 & 0.92 & 7.73e-01 & 0.94 & 8.41e-02 & 1.14 & 1.7870 & 0.294\\
  1891 & 0.3984 & 0.99 & 0.0365  & 1.19 & 6.37e-03 & 1.07 & 0.7071 & 0.98 & 3.91e-01 & 0.98 & 3.82e-02 & 1.14 & 0.9023 & 0.298\\
  7363 & 0.1995 & 1.00 & 0.0170  & 1.10 & 3.20e-03 & 0.99 & 0.3541 & 1.00 & 1.96e-01 & 1.00 & 1.77e-02 & 1.11 & 0.4519 & 0.298\\
 29059 & 0.0998 & 1.00 & 0.0082  & 1.05 & 1.60e-03 & 1.00 & 0.1771 & 1.00 & 9.80e-02 & 1.00 & 8.48e-03 & 1.07 & 0.2260 & 0.298\\
115459 & 0.0499 & 1.00 & 0.0040  & 1.02 & 8.01e-04 & 1.00 & 0.0885 & 1.00 & 4.90e-02 & 1.00 & 4.15e-03 & 1.03 & 0.1130 & 0.298\\
 \hline 
\multicolumn{15}{|c|}{$E^{\mathrm{E}} = E^{\mathrm{P}} = 1$, $\nu^{\mathrm{E}} = 0.25$, $\nu^{\mathrm{P}} = 0.25$, $\kappa = 1$, $k = 1$  $\vphantom{\int^X_X}$}\\
\hline 
   435 & 0.3750 & -- & 0.0334  & -- & 3.90e-03 & -- & 0.7644 & -- & 3.75e-01 & -- & 3.62e-02 & -- & 0.9317 & 0.151\\
  1635 & 0.1009 & 1.89 & 0.0073  & 2.20 & 1.08e-03 & 1.86 & 0.2092 & 1.87 & 1.01e-01 & 1.90 & 7.92e-03 & 2.19 & 0.2534 & 0.154\\
  6339 & 0.0259 & 1.96 & 0.0018  & 2.05 & 2.83e-04 & 1.93 & 0.0517 & 2.02 & 2.58e-02 & 1.96 & 1.90e-03 & 2.06 & 0.0634 & 0.150\\
 24963 & 0.0065 & 1.99 & 0.0004  & 2.02 & 7.25e-05 & 1.96 & 0.0128 & 2.02 & 6.51e-03 & 1.99 & 4.68e-04 & 2.02 & 0.0158 & 0.148\\
 99075 & 0.0016 & 2.00 & 0.0001  & 2.01 & 1.83e-05 & 1.98 & 0.0032 & 2.01 & 1.63e-03 & 2.00 & 1.16e-04 & 2.01 & 0.0039 & 0.148\\
394755 & 0.0004 & 2.00 & 2.48e-05  & 2.01 & 4.62e-06 & 1.99 & 0.0008 & 2.00 & 4.08e-04 & 2.00 & 2.90e-05 & 2.00 & 0.0010 & 0.147\\
\hline 
\multicolumn{15}{|c|}{$E^{\mathrm{E}} = 10^5$,  $E^{\mathrm{P}} = 10^5$, $\nu^{\mathrm{E}} = 0.499$, $\nu^{\mathrm{P}} = 0.499$, $\kappa = 1$, $k=0$ $\vphantom{\int^X_X}$}\\
\hline 
   139 & 450.02 & -- & 66.730  & -- & 2.14e-02 & -- & 723.84 & -- & 4.50e+2 & -- & 6.67e+1 & -- & 968.45 & 0.241\\
   499 & 225.32 & 1.00 & 13.844  & 2.27 & 1.19e-02 & 0.85 & 402.24 & 0.85 & 2.25e+2 & 1.00 & 1.38e+1 & 2.27 & 513.55 & 0.293\\
  1891 & 113.83 & 0.99 & 6.1821  & 1.16 & 6.24e-03 & 0.93 & 203.74 & 0.98 & 1.14e+2 & 0.99 & 6.18e+0 & 1.16 & 259.81 & 0.296\\
  7363 & 57.065 & 1.00 & 2.8910  & 1.10 & 3.18e-03 & 0.97 & 102.12 & 1.00 & 5.71e+1 & 1.00 & 2.89e+0 & 1.10 & 130.22 & 0.297\\
 29059 & 28.551 & 1.00 & 1.4069  & 1.04 & 1.60e-03 & 0.99 & 51.085 & 1.00 & 2.86e+1 & 1.00 & 1.41e+0 & 1.04 & 65.145 & 0.297\\
115459 & 14.277 & 1.00 & 0.6965  & 1.01 & 8.01e-04 & 1.00 & 25.543 & 1.00 & 1.43e+1 & 1.00 & 6.97e-01 & 1.01 & 32.575 & 0.297\\
\hline 
\multicolumn{15}{|c|}{$E^{\mathrm{E}} = 10^5$,  $E^{\mathrm{P}} = 10^5$, $\nu^{\mathrm{E}} = 0.499$, $\nu^{\mathrm{P}} = 0.499$, $\kappa = 1$, $k=1$ $\vphantom{\int^X_X}$}\\
\hline 
   435 & 143.37 & -- & 87.195  & -- & 3.85e-03 & -- & 418.47 & -- & 1.43e+2 & -- & 8.72e+1 & -- & 481.08 & 0.175\\
  1635 & 29.208 & 2.30 & 2.0812  & 5.39 & 1.07e-03 & 1.84 & 59.380 & 2.82 & 2.92e+1 & 2.30 & 2.08e+0 & 5.39 & 72.394 & 0.153\\
  6339 & 7.4723 & 1.97 & 0.3699  & 2.49 & 2.83e-04 & 1.93 & 14.911 & 1.99 & 7.47e+0 & 1.97 & 3.70e-01 & 2.49 & 18.283 & 0.150\\
 24963 & 1.8825 & 1.99 & 0.0907  & 2.03 & 7.25e-05 & 1.96 & 3.6848 & 2.02 & 1.88e+0 & 1.99 & 9.07e-02 & 2.03 & 4.5477 & 0.148\\
 99075 & 0.4716 & 2.00 & 0.0224  & 2.02 & 1.83e-05 & 1.98 & 0.9172 & 2.01 & 4.72e-01 & 2.00 & 2.24e-02 & 2.02 & 1.1345 & 0.148\\
394755 & 0.1180 & 2.00 & 0.0056  & 2.01 & 4.62e-06 & 1.99 & 0.2290 & 2.00 & 1.18e-01 & 2.00 & 5.57e-03 & 2.01 & 0.2834 & 0.147\\
\hline 
\multicolumn{15}{|c|}{$E^{\mathrm{E}} = 10^5$,  $E^{\mathrm{P}} = 10^5$, $\nu^{\mathrm{E}} = 0.499$, $\nu^{\mathrm{P}} = 0.499$, $\kappa = 10^{-12}$, $k=0$ $\vphantom{\int^X_X}$}\\
\hline
   139 & 450.02 & -- & 66.730  & -- & 8.47e-04 & -- & 723.84 & -- & 4.50e+2 & -- & 6.67e+1 & -- & 968.45 & 0.241\\
   499 & 225.32 & 1.00 & 13.844  & 2.27 & 2.15e-04 & 1.98 & 402.24 & 0.85 & 2.25e+2 & 1.00 & 1.38e+1 & 2.27 & 513.55 & 0.293\\
  1891 & 113.83 & 0.99 & 6.1821  & 1.16 & 5.59e-05 & 1.94 & 203.74 & 0.98 & 1.14e+2 & 0.99 & 6.18e+0 & 1.16 & 259.81 & 0.296\\
  7363 & 57.065 & 1.00 & 2.8910  & 1.10 & 1.45e-05 & 1.95 & 102.12 & 1.00 & 5.71e+1 & 1.00 & 2.89e+0 & 1.10 & 130.22 & 0.297\\
 29059 & 28.551 & 1.00 & 1.4069  & 1.04 & 3.79e-06 & 1.93 & 51.085 & 1.00 & 2.86e+1 & 1.00 & 1.41e+0 & 1.04 & 65.145 & 0.297\\
115459 & 14.277 & 1.00 & 0.6965  & 1.01 & 1.02e-06 & 1.90 & 25.543 & 1.00 & 1.43e+1 & 1.00 & 6.97e-01 & 1.01 & 32.575 & 0.297\\
 \hline 
\multicolumn{15}{|c|}{$E^{\mathrm{E}} = 10^5$,  $E^{\mathrm{P}} = 10^5$, $\nu^{\mathrm{E}} = 0.499$, $\nu^{\mathrm{P}} = 0.499$, $\kappa = 10^{-12}$, $k=1$ $\vphantom{\int^X_X}$}\\
\hline
   435 & 143.37 & -- & 87.195  & -- & 6.56e-04 & -- & 418.47 & -- & 1.43e+2 & -- & 8.72e+1 & -- & 481.08 & 0.175\\
  1635 & 29.208 & 2.30 & 2.0812  & 5.39 & 2.18e-05 & 4.91 & 59.380 & 2.82 & 2.92e+1 & 2.30 & 2.08e+0 & 5.39 & 72.394 & 0.153\\
  6339 & 7.4723 & 1.97 & 0.3699  & 2.49 & 3.41e-06 & 2.68 & 14.911 & 1.99 & 7.47e+0 & 1.97 & 3.70e-01 & 2.49 & 18.283 & 0.150\\
 24963 & 1.8825 & 1.99 & 0.0907  & 2.03 & 7.81e-07 & 2.12 & 3.6848 & 2.02 & 1.88e+0 & 1.99 & 9.07e-02 & 2.03 & 4.5477 & 0.148\\
 99075 & 0.4716 & 2.00 & 0.0224  & 2.02 & 1.90e-07 & 2.04 & 0.9172 & 2.01 & 4.72e-01 & 2.00 & 2.24e-02 & 2.02 & 1.1345 & 0.148\\
394755 & 0.1180 & 2.00 & 0.0056  & 2.01 & 4.71e-08 & 2.01 & 0.2290 & 2.00 & 1.18e-01 & 2.00 & 5.57e-03 & 2.01 & 0.2834 & 0.147\\
\hline 
 \end{tabular}}
  
  \smallskip
  \caption{Example 1C: Errors, convergence rates, and effectivity indexes under uniform mesh refinement. Smooth manufactured solutions for the rotation-based interfacial elasticity/poroelasticity problem with $k =0,1$.}
\label{table:ex01C}
 \end{center}
\end{table}

For the second and third examples, we employ adaptive mesh refinement consisting in the usual steps of solving, then computing the local and global estimators, marking, refining, and smoothing. The marking of elements for refinement follows the classical D\"orfler approach \cite{dorfler_sinum96}: a given $K\in \cT_h$ is \emph{marked} (added to the marking set $\cM_h\subset\cT_h$)  whenever the local error indicator $\Phi_K$ satisfies 
$ \sum_{K \in \cM_h} \Phi^2_K \geq \zeta \sum_{K\in\cT_h} \Phi_K^2$,
where $0<\zeta<1$ is a user-defined parameter (meaning that one refines elements that contribute to a proportion $\zeta$ of the total squared error). Elements in $\cM_h$ are then refined (their diameter is halved) and an additional smoothing step is applied before starting a new iteration of the algorithm. When computing convergence rates under uniform refinement, we use the following modification to \eqref{eq:uniform-rate}:  
\[\texttt{r}_{(\cdot)} = -2\log(e_{(\cdot)}/\tilde{e}_{(\cdot)})[\log({\tt DoF}/\widetilde{{\tt DoF}})]^{-1}.\]

The second example investigates again the accuracy of the three numerical methods but this time we use the L-shaped domain $\Omega = (-1,1)^2\setminus (0,1)^2$ and the transmission problem has the interface defined as the segment going from the reentrant corner $(0,0)$ to the bottom-left corner of the domain $(-1,-1)$, and the porous domain is the one above the interface. In addition to the singularity of the geometry, we use manufactured solutions with sharp gradients near the reentrant corner 
\[
p^{\mathrm{P}}(x,y) = \exp(-25(x^2+y^2)), \qquad 
\bu(x,y) = \bigl( \exp(-50(x^2+y^2)),\exp(-50(x^2+y^2)) \bigr)^{\tt t}.
\]
It is expected that the convergence of the methods is hindered due to the lack of regularity of the exact solutions whenever one follows a uniform mesh refinement. Such a slower error decay is clearly observed in the top half of Table~\ref{table:ex2}, while adaptive mesh refinement (with a D\"orfler constant of $\zeta = 0.001$) yields asymptotic optimal convergence evidenced on the bottom half of the table, where also one reaches much smaller errors using a fraction of the degrees of freedom needed in the uniform case. Here we focus on the methods with $k=1$, and samples of adaptively refine meshes and approximate solutions are portrayed in Figure~\ref{fig:ex02}. In this case we have used the following contrast of parameters between the subdomains 
$c_0 = 0$, $\alpha = 1$, $E^\mathrm{E} = 10$,  
 $E^\mathrm{P} = 1$,  
 $\nu^\mathrm{E} = 0.25$, $\nu^\mathrm{P} = 0.45$,  
$\xi = 1$, $\kappa = 10^{-3}$.

\begin{figure}[t!]
\begin{center}
\includegraphics[width=0.23\textwidth]{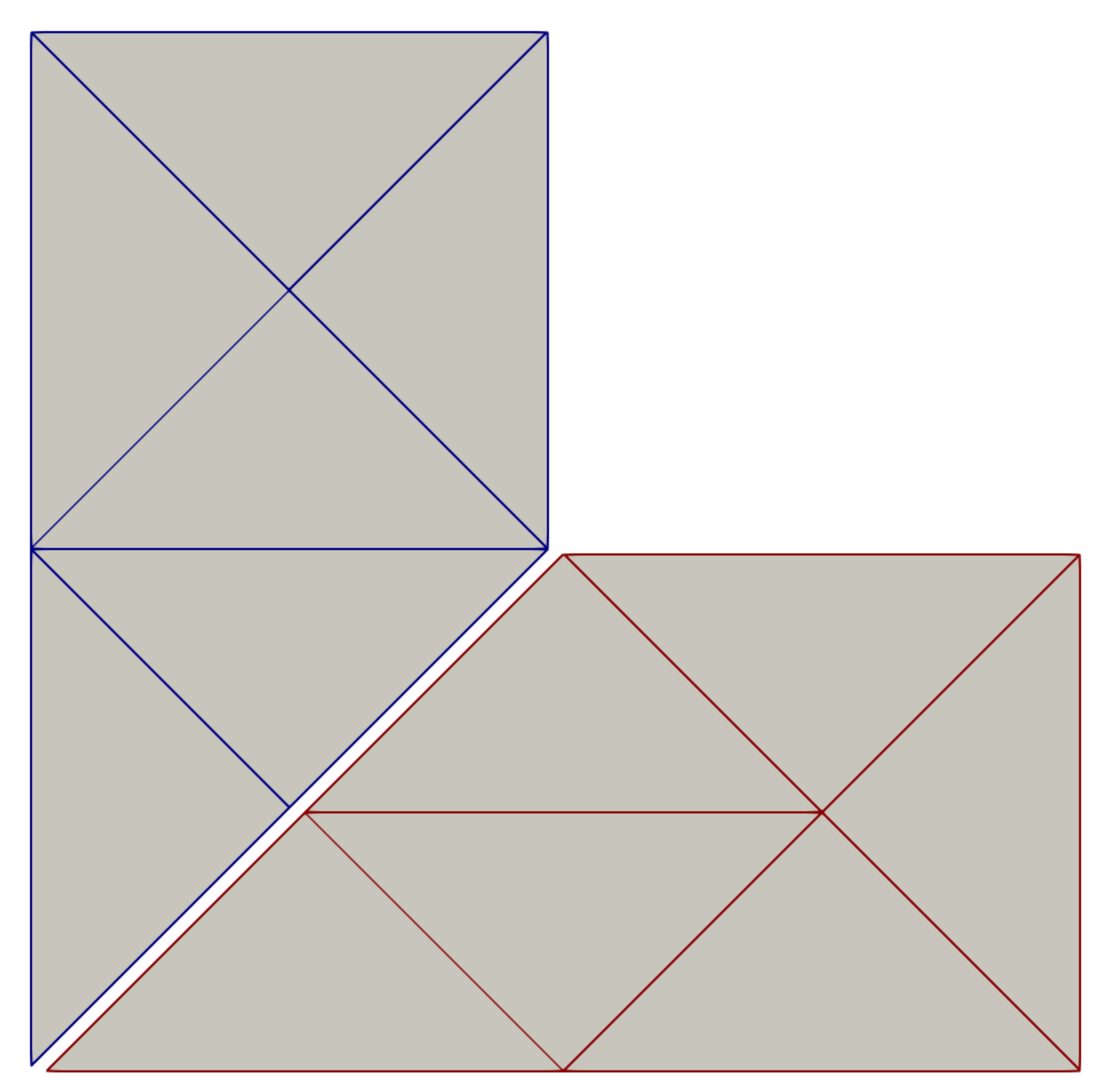}
\includegraphics[width=0.23\textwidth]{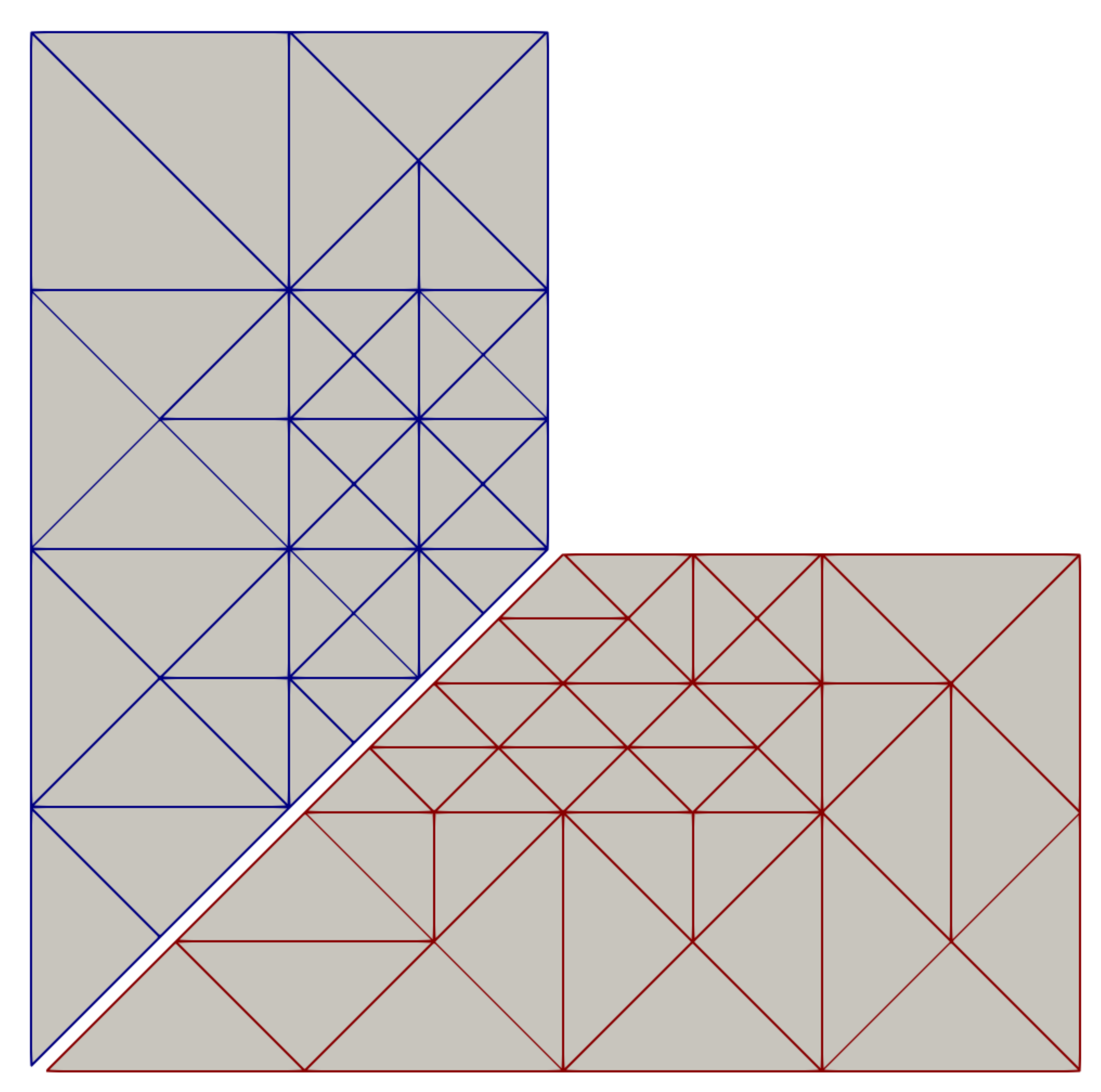}
\includegraphics[width=0.23\textwidth]{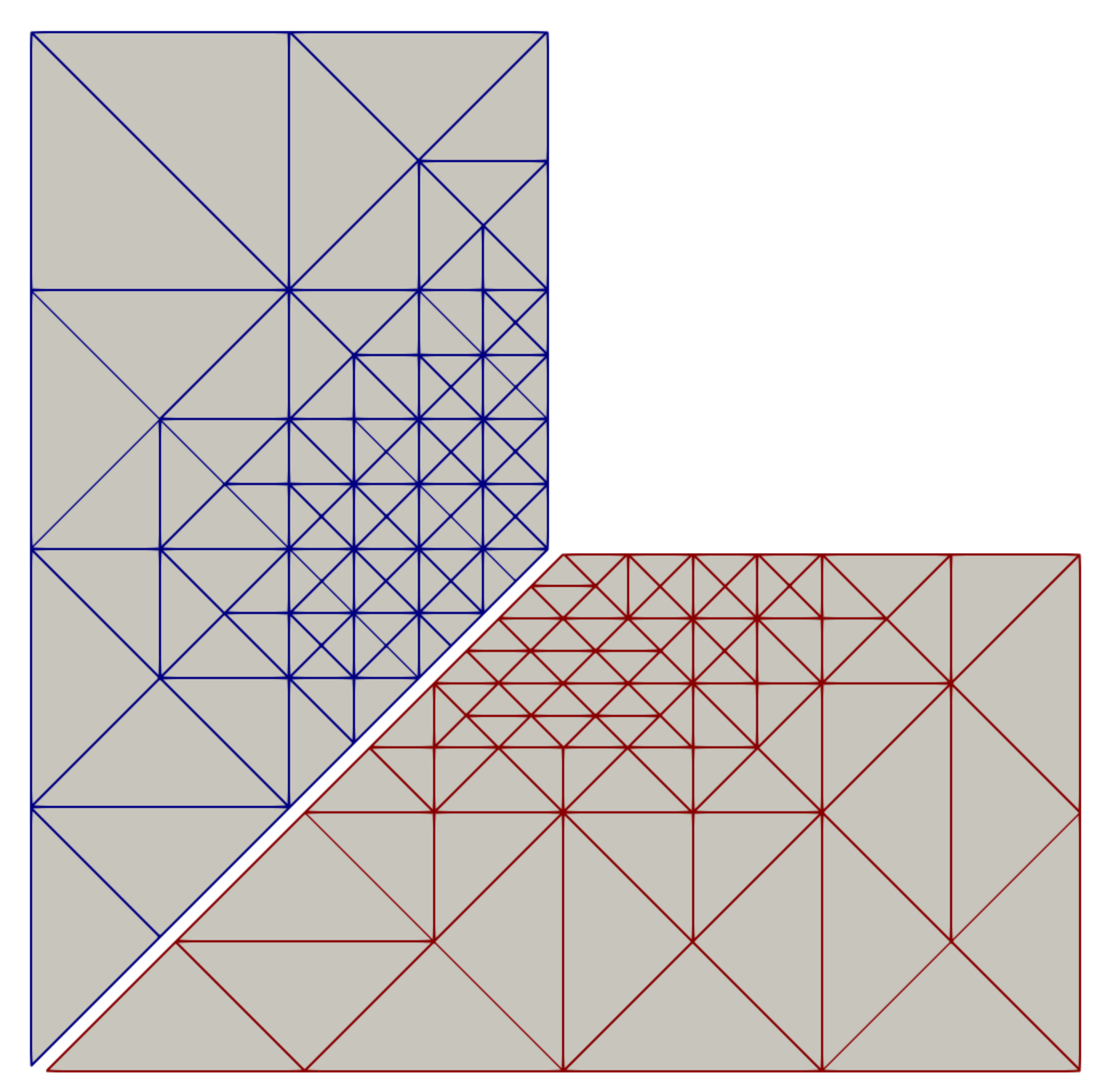}
\includegraphics[width=0.23\textwidth]{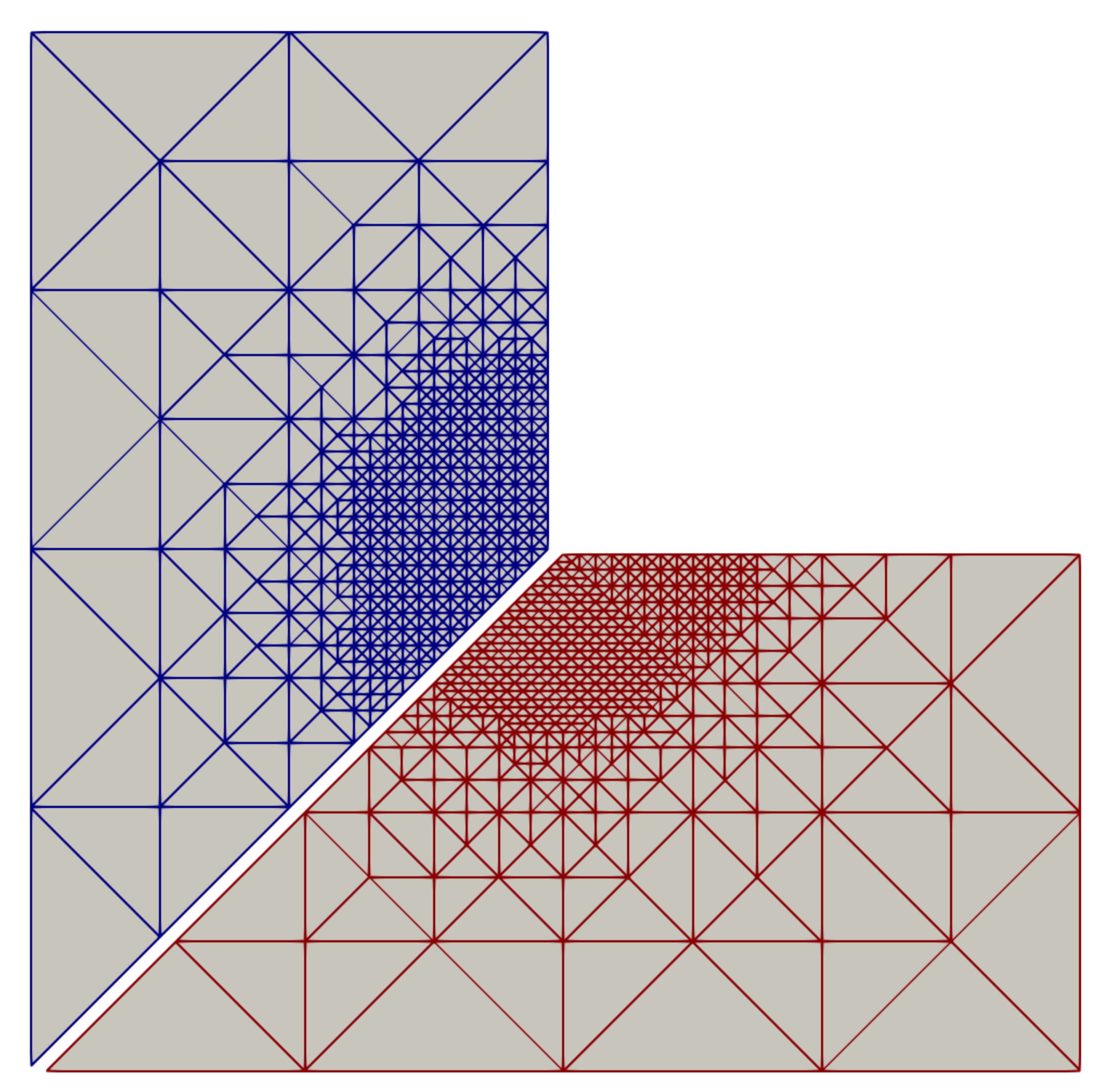}\\[1ex]
\includegraphics[width=0.3\textwidth]{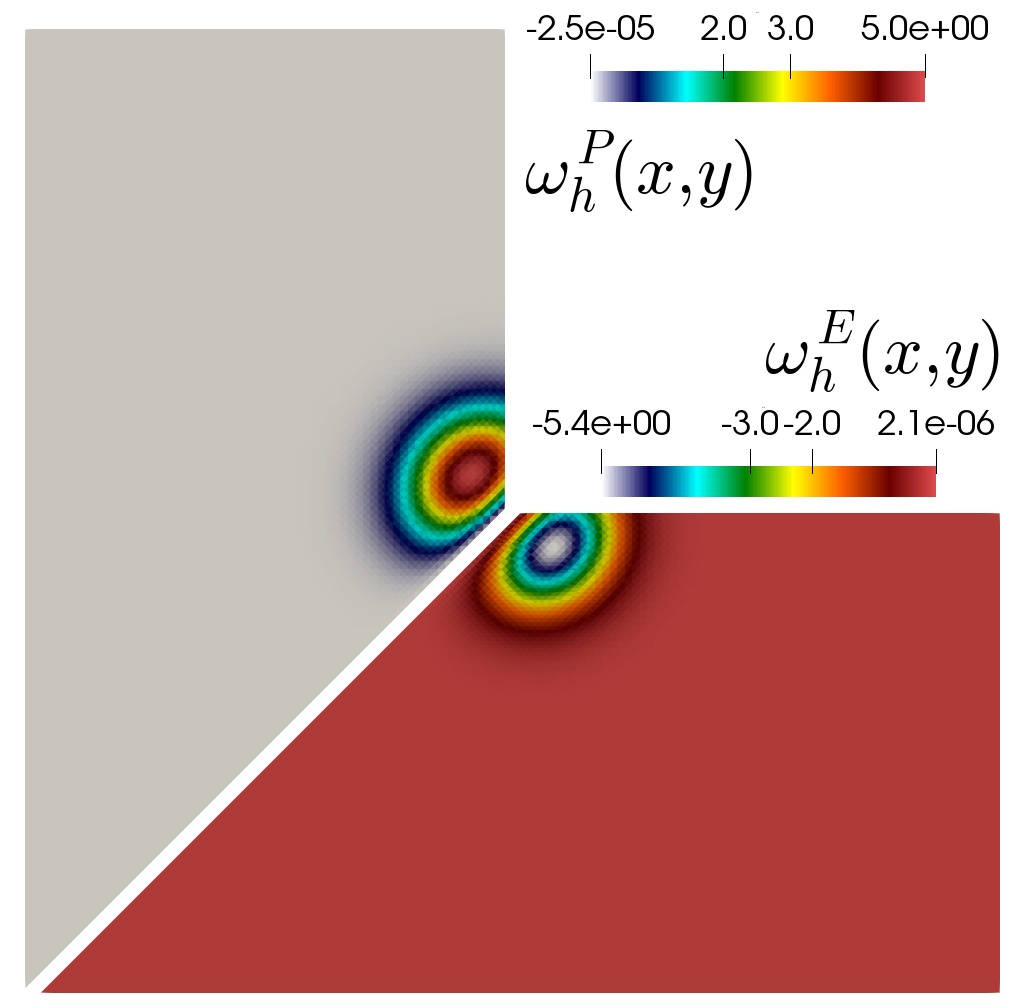}
\includegraphics[width=0.3\textwidth]{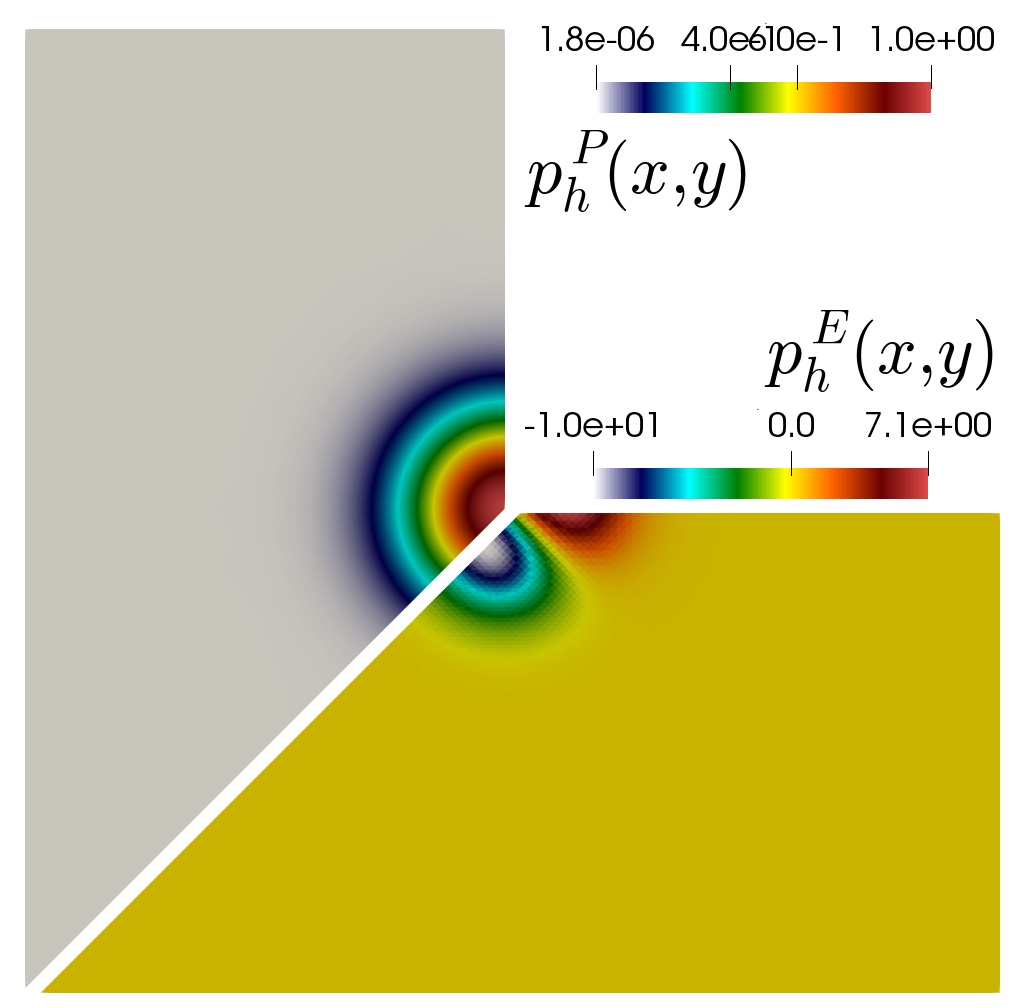}
\includegraphics[width=0.3\textwidth]{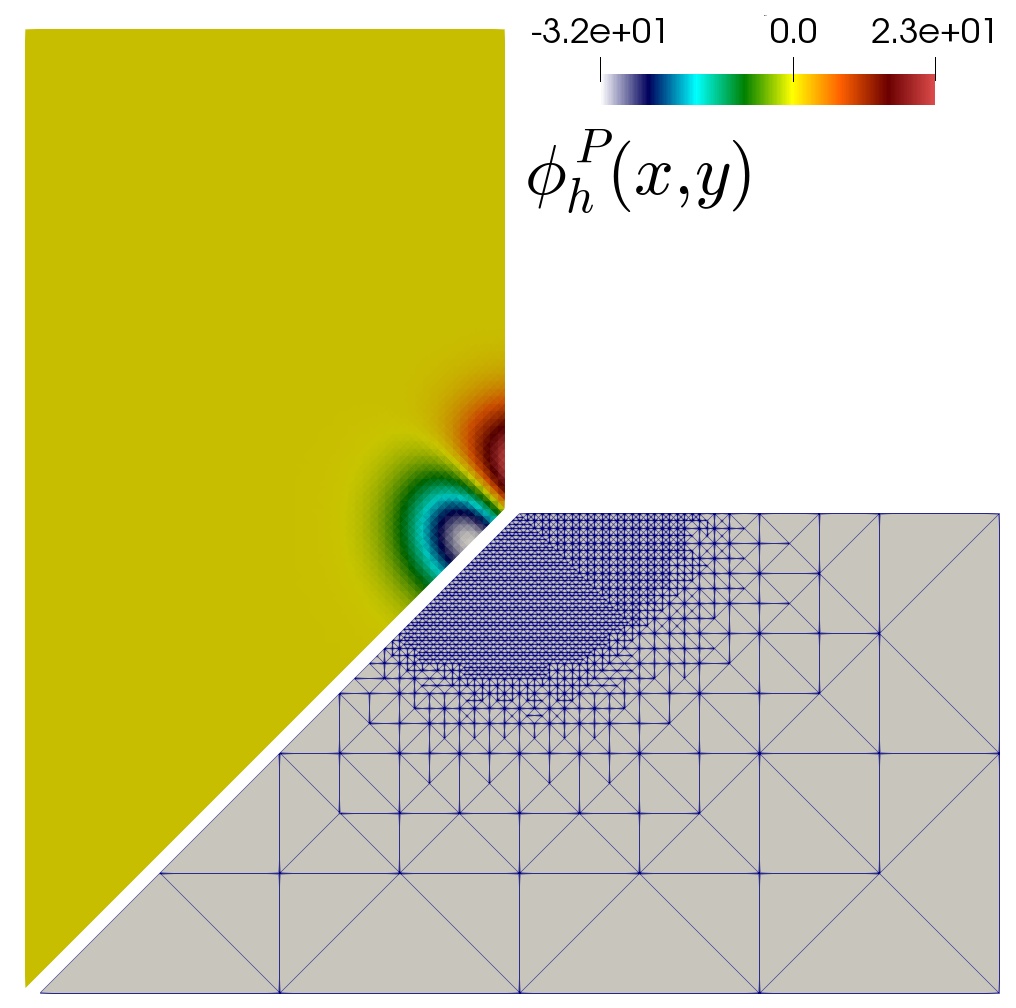}
 \end{center}
\caption{Example 2. Adaptively refined meshes (top) and  approximate solutions (bottom) for the Biot/elasticity problem.}
\label{fig:ex02}
 \end{figure}

 \begin{table}[t!]
\setlength{\tabcolsep}{3.5pt}
 \begin{center}
 {\footnotesize \begin{tabular}{|rcccccccccccccc|}
\hline 
DoFs  & $\texttt{e}_{\bomega^{\mathrm{P}}}$ & $\texttt{r}_{\bomega^{\mathrm{P}}}$& $\texttt{e}_{\phi^{\mathrm{P}}}$ & $\texttt{r}_{\phi^{\mathrm{P}}}$  & $\texttt{e}_{p^{\mathrm{P}}}$ & $\texttt{r}_{p^{\mathrm{P}}}$ & $\texttt{e}_{\bu}$ & $\texttt{r}_{\bu}$ & $\texttt{e}_{\bomega^{\mathrm{E}}}$ & $\texttt{r}_{\bomega^{\mathrm{E}}}$ & $\texttt{e}_{p^{\mathrm{E}}}$ & $\texttt{r}_{p^{\mathrm{E}}}$ & \texttt{e} & $\texttt{eff}(\Xi)$ $\vphantom{\int^X_X}$\\
\hline 
\multicolumn{15}{|c|}{With uniform mesh refinement}\\
\hline 
   157 & 8.56e-01 & -- & 7.86e+0 & -- & 1.36e+0 & -- & 1.61e+0 & -- & 9.86e-01 & -- & 2.35e+0 & -- & 8.246 & 0.071\\
   575 & 3.80e-01 & 1.17 & 3.05e+0 & 1.37 & 3.64e-01 & 1.90 & 7.07e-01 & 1.19 & 4.15e-01 & 1.25 & 1.16e+0 & 1.02 & 4.199 & 0.151\\
  2203 & 2.22e-01 & 0.77 & 1.74e+0 & 0.81 & 1.77e-01 & 1.05 & 3.92e-01 & 0.85 & 2.22e-01 & 0.90 & 6.19e-01 & 0.91 & 1.823 & 0.087\\
  8627 & 7.72e-02 & 1.52 & 4.39e-01 & 1.99 & 4.10e-02 & 2.11 & 1.14e-01 & 1.78 & 5.98e-02 & 1.89 & 1.54e-01 & 2.01 & 0.765 & 0.091\\
 34147 & 2.66e-02 & 1.54 & 1.09e-01 & 2.01 & 1.09e-02 & 1.92 & 3.45e-02 & 1.72 & 1.58e-02 & 1.92 & 4.10e-02 & 1.91 & 0.319 & 0.034\\
135875 & 1.97e-02 & 0.44 & 3.79e-02 & 1.53 & 3.64e-03 & 1.58 & 2.16e-02 & 0.68 & 7.28e-03 & 1.12 & 1.37e-02 & 1.58 & 0.248 & 0.152\\
\hline 
\multicolumn{15}{|c|}{With adaptive mesh refinement}\\
\hline 
   157 & 8.57e-01 & -- & 7.86e+0 & -- & 1.36e+0 & -- & 1.62e+0 & -- & 9.87e-01 & 0.00 & 2.36e+0 & -- & 8.239 & 0.081\\
   551 & 3.79e-01 & 1.30 & 3.05e+0 & 1.51 & 3.65e-01 & 2.10 & 7.07e-01 & 1.32 & 4.17e-01 & 1.37 & 1.16e+0 & 1.13 & 3.201 & 0.081\\
  1020 & 2.21e-01 & 1.75 & 1.74e+0 & 1.82 & 1.77e-01 & 2.36 & 3.91e-01 & 1.92 & 2.22e-01 & 2.04 & 6.18e-01 & 2.05 & 1.822 & 0.087\\
  2307 & 7.36e-02 & 2.69 & 4.38e-01 & 3.38 & 4.20e-02 & 3.52 & 1.11e-01 & 3.09 & 5.90e-02 & 3.25 & 1.53e-01 & 3.43 & 0.463 & 0.091\\
  5779 & 1.81e-02 & 3.06 & 1.06e-01 & 3.10 & 1.21e-02 & 2.71 & 2.74e-02 & 3.04 & 1.45e-02 & 3.06 & 3.96e-02 & 2.94 & 0.112 & 0.089\\
 20209 & 4.59e-03 & 2.19 & 2.67e-02 & 2.20 & 3.21e-03 & 2.12 & 6.94e-03 & 2.19 & 3.63e-03 & 2.21 & 1.02e-02 & 2.18 & 0.028 & 0.089\\
 70299 & 1.15e-03 & 2.22 & 6.69e-03 & 2.22 & 1.23e-03 & 1.54 & 1.74e-03 & 2.22 & 9.07e-04 & 2.22 & 2.56e-03 & 2.21 & 0.007 & 0.090\\
\hline 
 \end{tabular}}
  
  \smallskip
  \caption{Example 2: Errors, convergence rates, and effectivity indexes under uniform vs adaptive mesh refinement for the rotation-based interfacial elasticity/poroelasticity problem on the L-shaped domain, with $k=1$.}
\label{table:ex2}
 \end{center}
\end{table} 

The last test illustrates the use of mesh adaptivity guided by the \textit{a posteriori} error estimator $\Xi$ on an interface elasticity/poroelasticity problem applied to oil reservoir poromechanics, similarly to the test in \cite[Sect. 8.2]{girault_cmame20} (see also \cite{girault_ogst19,anaya_sisc20}). In CO$_2$ sequestration 
in deep subsurface reservoirs one is interested in the distribution of pressure and displacement across the interface between the non-pay rock and the aquifer zones in the case where the poroelastic domain is an array of thin-walled structures fully surrounded by an elastic region. The multi-domain is the unit cube $\Omega = (0,1)^3$\,m$^3$ and the aquifer array has a width of 0.015\,m. A well is represented by 
a localised source $s^\mathrm{P}(x,y,z)=s_0\exp(-1000[(x-0.5)^2+(y-0.5)^2+(z-0.5)^2])$. This is an injection zone of relatively small radius reaching the centre of the pay zone at $(0.5,0.5,0.5)$. On the surface of the non-pay 
rock  we impose the sliding condition $\bu\cdot\nn = 0$. The interfacial conditions are as in \eqref{eq:transmission}. 
The simulation uses  the following values for the model parameters 
$s_0=0.5$, $c_0 = 10^{-3}$,  $\alpha = 0.75$, $E^\mathrm{E} = 5\cdot 10^3$, 
 $E^\mathrm{P} = 10^{3}$,  
 $\nu^\mathrm{E} = 0.3$, $\nu^\mathrm{P} = 0.45$,  
$\xi = 10^{-3}$, $\kappa = 10^{-7}$, $\gg = (0,0,-9.81)^{\tt t}$.

Initial coarse meshes are constructed for both subdomains, then we solve the coupled transmission problem, and then apply six steps of the iterative mesh refinement strategy based on the estimator $\Xi$. To observe how the adaptivity takes place on both elastic and poroelastic domains, we plot in Figure~\ref{fig:ex03}, samples of the approximate solutions on the first three steps of adaptive mesh refinement. The concentration of amount of fluid near the centre of the pay zone is seen in the first row, and we can also see the concentration of refinement near the interface in all panels.  

\begin{figure}[t!]
\begin{center}
\includegraphics[width=0.3\textwidth]{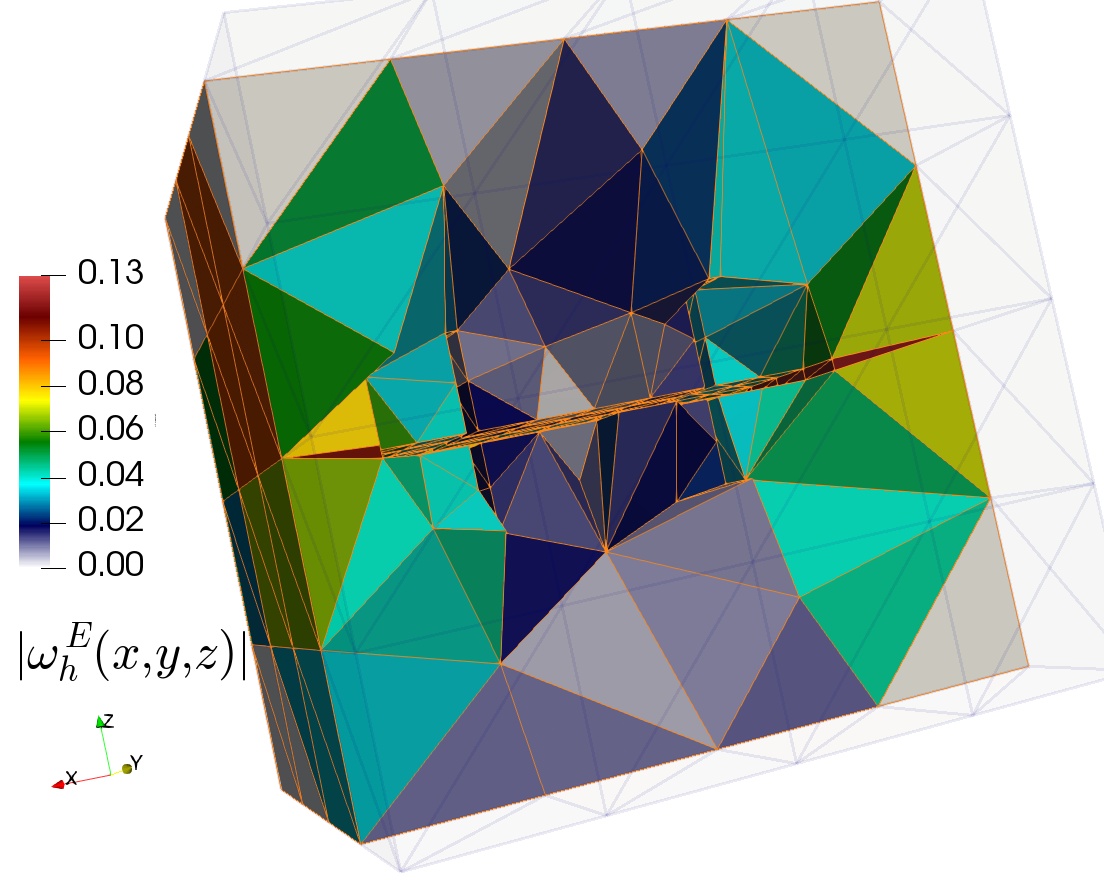}
\includegraphics[width=0.3\textwidth]{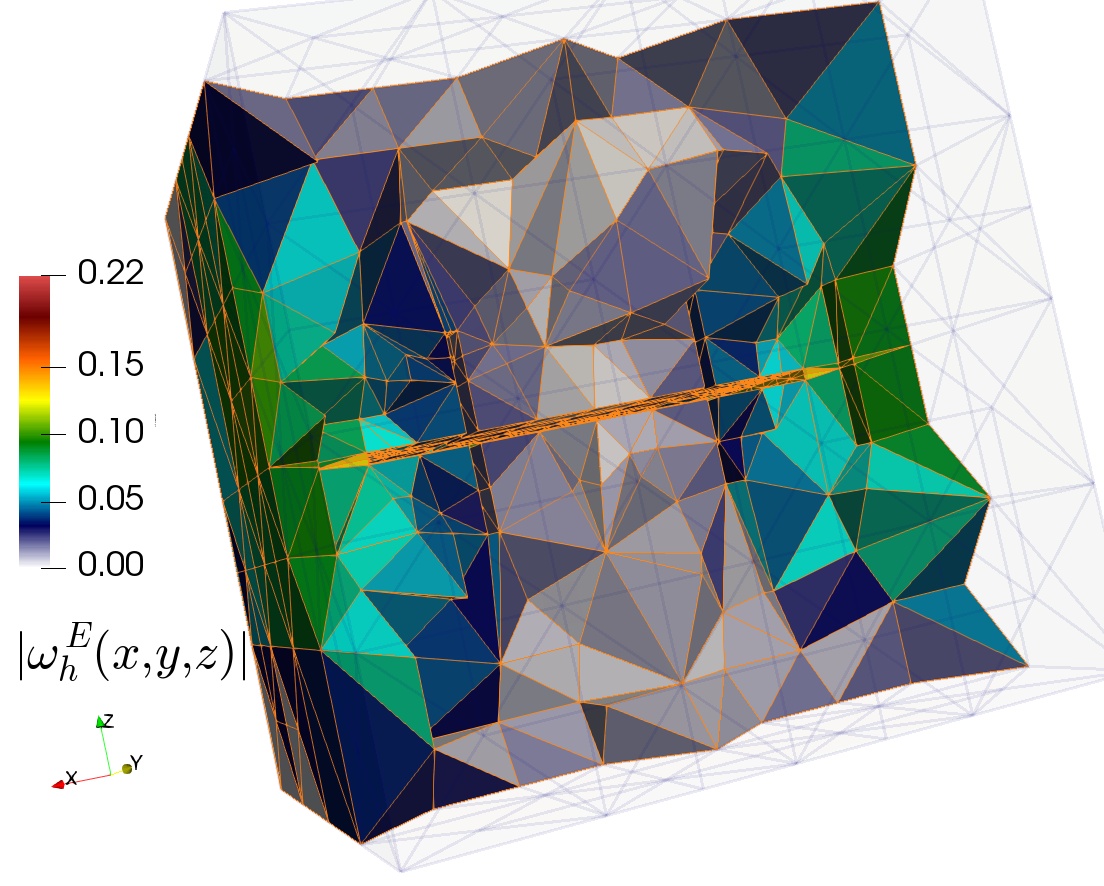}
\includegraphics[width=0.3\textwidth]{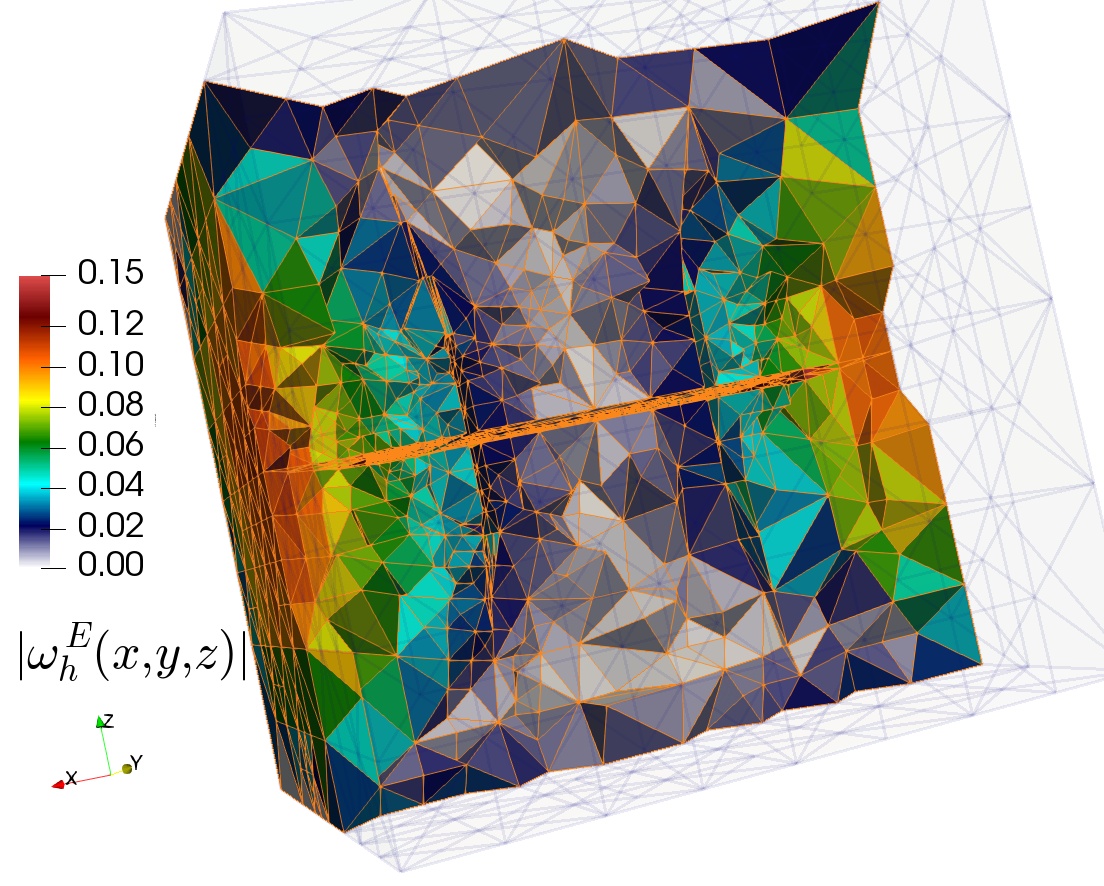}\\
\includegraphics[width=0.3\textwidth]{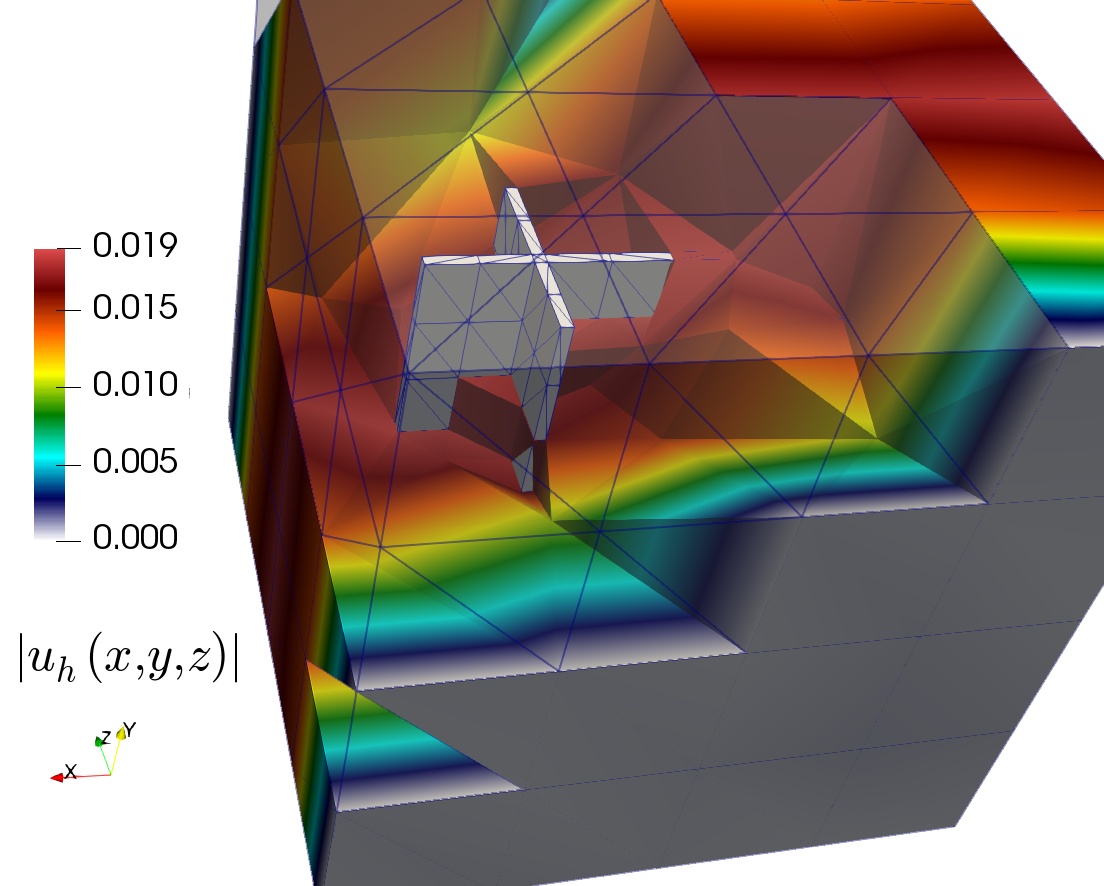}
\includegraphics[width=0.3\textwidth]{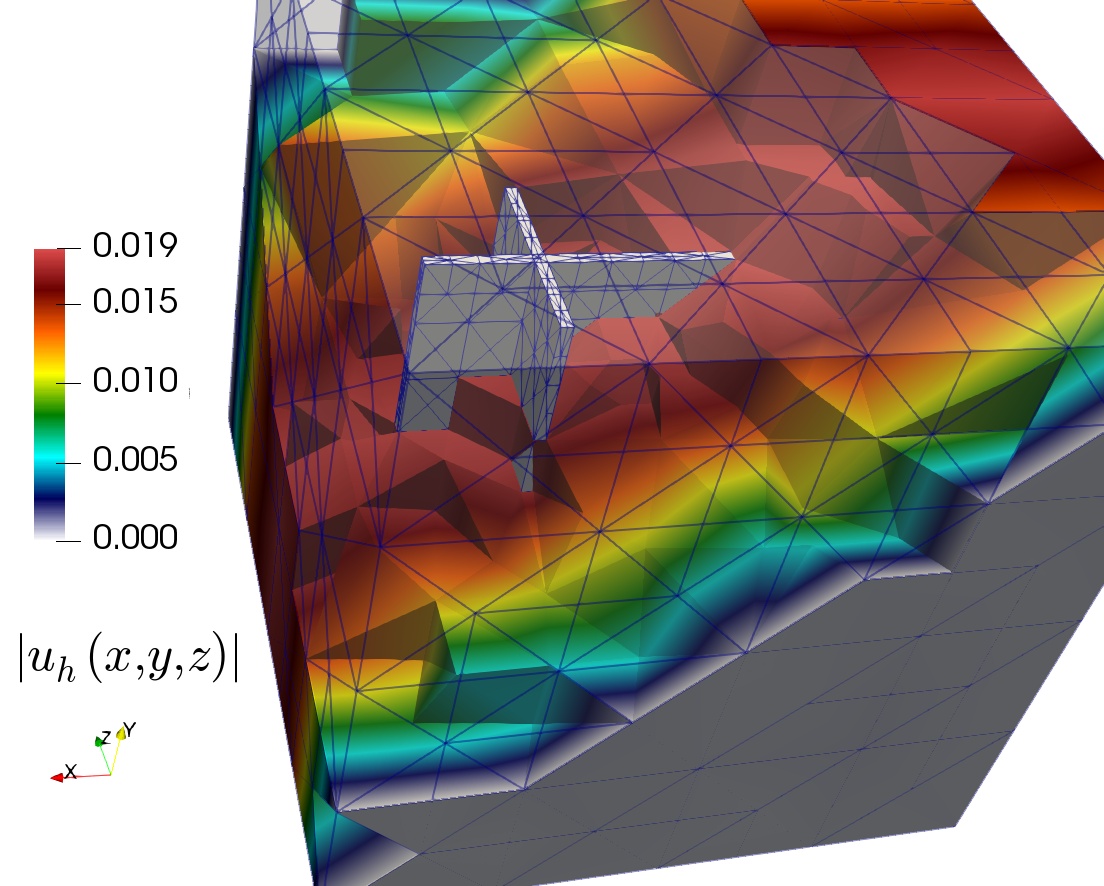}
\includegraphics[width=0.3\textwidth]{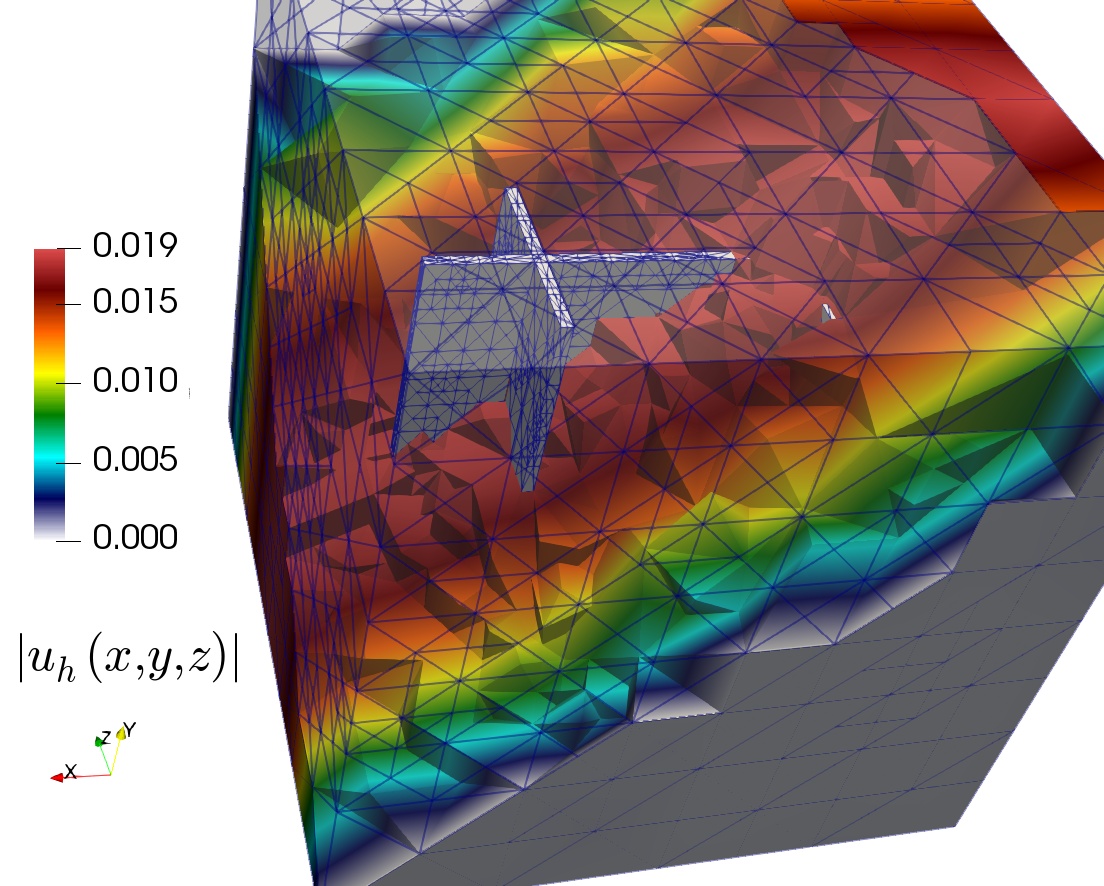}\\
\includegraphics[width=0.3\textwidth]{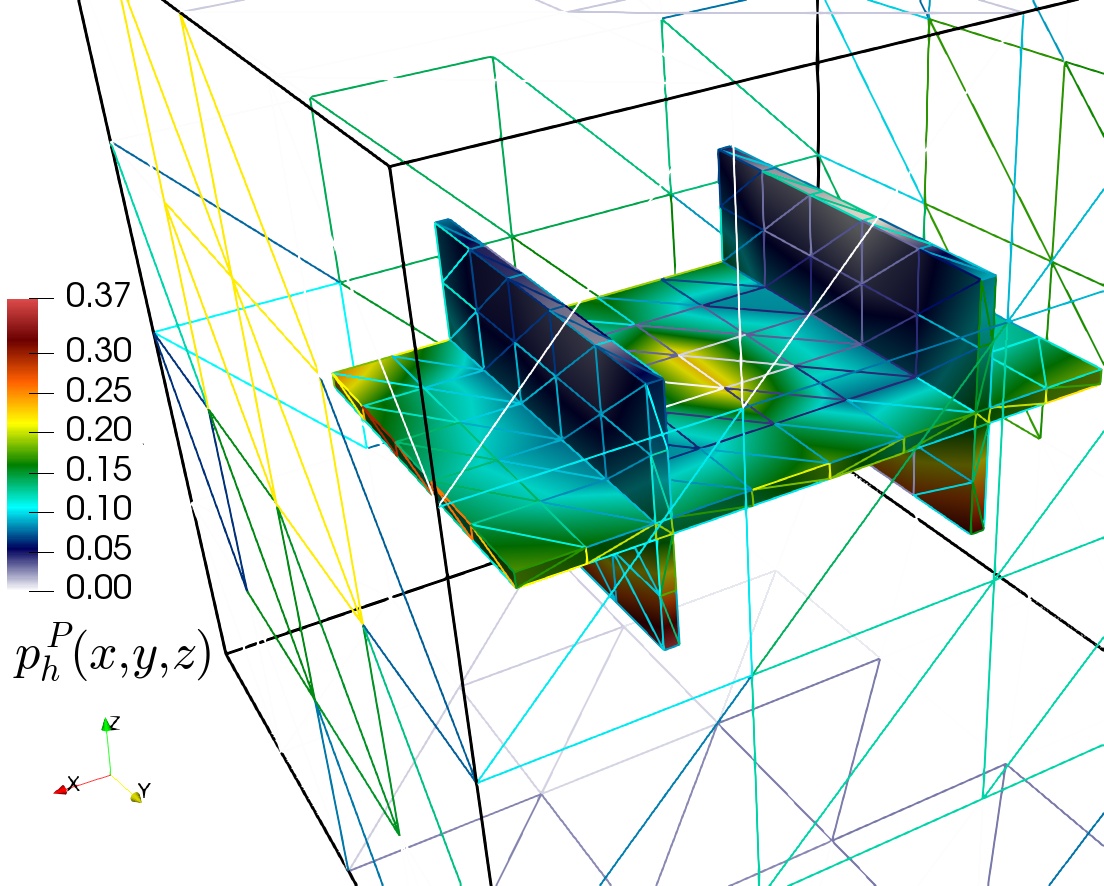}
\includegraphics[width=0.3\textwidth]{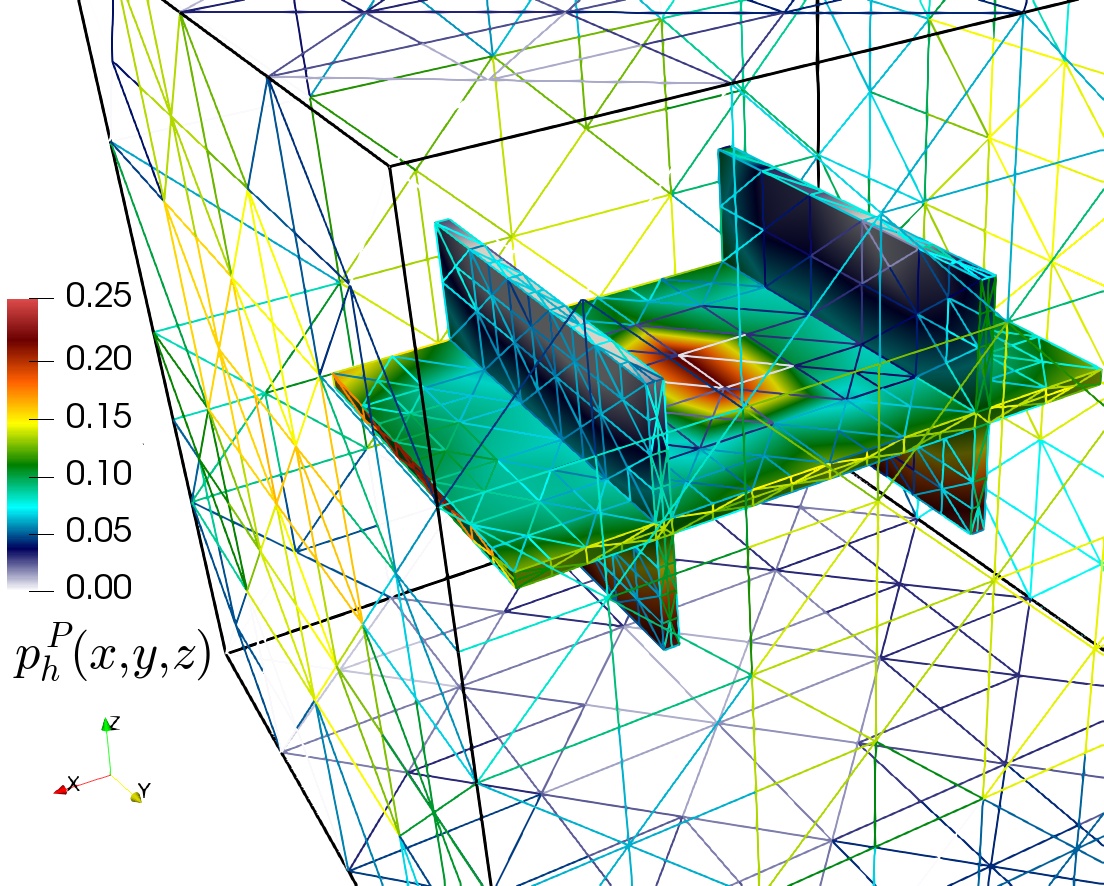}
\includegraphics[width=0.3\textwidth]{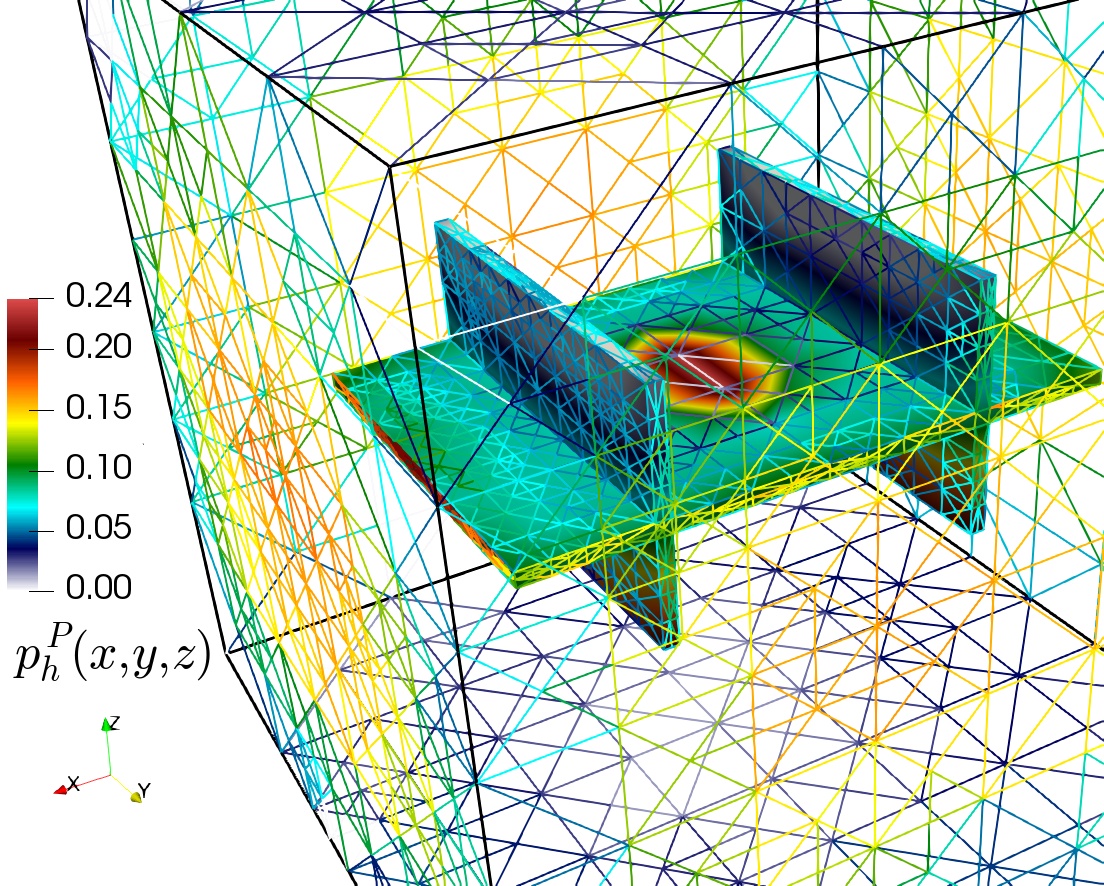}
\end{center}
\caption{Example 3. Approximate elastic rotation on a horizontally clipped elastic geometry (top), displacement on a diagonally clipped domain (centre), and fluid pressure on a zoomed poroelastic domain (bottom); for three steps of adaptive refinement for the Biot/elasticity application in fractured reservoirs.}
\label{fig:ex03}
 \end{figure}

\bibliographystyle{siam} 
\bibliography{BibFile-elastPoroelast}

\begin{appendix}

\section{Well-posedness analysis for rotation-based poroelasticity}\label{sec:poroelast-wellp}
The bilinear forms and the linear functionals
appearing in the variational problem (\emph{cf.} Section \ref{sec:poroelast}) of interest are all bounded
by constants independent of $\mu^{\mathrm{P}}$ and $\lambda^{\mathrm{P}}$ \cite{anaya_sisc20,anaya_cmame19}. In addition, we have the following result. 

%
\begin{lemma}\label{stability}
Let $(\vomega,\bu, p) \in \mathbf{H} \times\bV\times \rQ$, where $\vomega=(\bomega,\phi)$,
be a solution of the system \eqref{weak-uP}-\eqref{weak-phi}, then there exists a constant $C > 0$, such that
\begin{equation}\label{stability-bound-aux}
	\VERT(\bu,\bomega,\phi,p)\VERT  \leq C  \big\{(\mu^{\mathrm{P}})^{-1/2}\|\ff^{\mathrm{P}}\|_{0, \Omega}+\|(\kappa/\xi)^{1/2}\rho\gg\|_{0, \Omega} +\rho_1^{1/2}\|s^{\mathrm{P}}\|_{0, \Omega} \big\},
\end{equation}
where $\rho_1=\min\left(\left(c_0+\frac{\alpha^2}{(2\mu^{\mathrm{P}}+\lambda^{\mathrm{P}})}\right)^{-1},(\kappa/\xi)^{-1}\right)$.
\end{lemma}
\begin{proof}
Using Theorem \ref{stab-poro} implies
\[
C_2 \VERT(\bu,\bomega,\phi,p)\VERT^2\le B_{\mathrm{P}}((\bu,\bomega,\phi,p),(\bv,\btheta,\psi,q))
= F(\bv)+G(q),
\]
with $\VERT(\bv,\btheta,\psi,q)\VERT\le C_1\VERT(\bu,\bomega,\phi,p)\VERT$, where $C_1$ and $C_2$ are constants given in Theorem \ref{stab-poro}. And \eqref{stability-bound-aux} results from applying Cauchy-Schwarz inequality. 
\end{proof}


\subsection{Solvability of the continuous problem}\label{subsec:solvability}
Let us rewrite \eqref{weak-uP}-\eqref{weak-phi} as: find $\vec{\bu} := (\vomega,\bu,p) \in \bX := \mathbf{H}\times\bV \times \rQ$ such that
$(\mathcal{S} + \mathcal{T}) \vec{\bu} = \mathcal{F}$,
where the linear operators $\mathcal{S}:\bX \rightarrow \bX^\star$, $\mathcal{T}:\bX \rightarrow \bX^\star$, and $\mathcal{F} \in \bX^\star$ are defined as
\begin{align*}
	\langle\mathcal{S}(\vec{\bu}), \vec{\bv}\rangle :&= a(\vomega, \vtheta) + b_1(\vtheta,\bu) -b_1(\vomega,\bv)+c(p,q),
	\\
	\langle\mathcal{T}(\vec{\bu}), \vec{\bv}\rangle :&= -b_2(\vtheta,p) - b_2(\vomega,q),\qquad 
	\langle\mathcal{F}, \vec{\bv}\rangle := -F(\bv) - G(q),
\end{align*}
for all $\vec{\bu} := (\vomega,\bu,p),\, \vec{\bv} := (\vtheta,\bv,q) \in \bX$, where  $\langle\cdot,\cdot\rangle$ is the duality pairing between $\bX$ and its dual $\bX^\star$.

\begin{lemma}\label{well-posedness1}
The operator $\mathcal{S}:\bX \rightarrow \bX^\star$ is invertible.
\end{lemma}
\begin{proof}
First, for a given functional $\mathcal{F} := (\mathcal{F}_{\mathbf{H}},\mathcal{F}_{\bV},\mathcal{F}_{\rQ})$, observe that establishing the invertibility of $\mathcal{S}$ is equivalent to proving the unique solvability of the operator problem
\begin{align}\label{mainproblemS}
	\mathcal{S}(\vec{\bu}) = \mathcal{F}.
\end{align}	
Furthermore, proving unique solvability of  \eqref{mainproblemS} is in turn equivalent to proving the unique solvability of the two following uncoupled problems: find $(\vomega,\bu) \in \mathbf{H}\times\bV$ such that
\begin{align}\label{problemS1a}
	\begin{split}
	a(\vomega,\vtheta)+b_1(\vtheta,\bu) &\;=\; F_{\mathbf{H}}(\vtheta)\qquad\forall\,\vtheta\in\mathbf{H},
	\\
	b_1(\vomega,\bv)&\;=\; F_{\bV}(\bv)\qquad\forall\,\bv\in\bV,
	\end{split}
\end{align}
and: find $p \in \rQ$, such that
\begin{align}\label{subproblemS2}
	c(p,q) = F_{\rQ}(q) \qquad\forall q\in\rQ,
\end{align}
where $F_{\mathbf{H}}$, $F_{\bV}$, and $F_{\rQ}$ are the functionals induced by $\mathcal{F}_{\mathbf{H}}$, $\mathcal{F}_{\bV}$, and $\mathcal{F}_{\rQ}$, respectively. 
The unique solvability of \eqref{subproblemS2}
follows by virtue of the Lax-Milgram lemma, and the well-posedness
of (\ref{problemS1a}) follows from a straightforward application of the
Babu\v ska-Brezzi theory.
\end{proof}

\begin{lemma}\label{well-posedness2}
The operator $\mathcal{T}:\bX \rightarrow \bX^\star$ is compact.
\end{lemma}
\begin{proof}
	We begin by defining the operator $\mathbb{B}:\mathrm{L}^2(\Omega)\rightarrow\rQ$ as
	\begin{align*}
	\langle\mathbb{B}(\psi),q\rangle_{0,\Omega} &:= \alpha(2\mu^{\mathrm{P}}+\lambda^{\mathrm{P}})^{-1}\int_{\Omega}q\psi \qquad \forall\, q\in\rQ, \forall\,\psi\in\mathrm{L}^2(\Omega).
	\end{align*}
	This operator is the composition of a compact injection and a continuous 
	map and it is therefore compact. 
And	denoting by $\mathbb{B}^\star$ the adjoint of $\mathbb{B}$, we infer that the following map is also compact
	\[\mathcal{T}(\vec{\bu}) = ((\cero,-\mathbb{B}(\phi),\cero,0),\cero,-\mathbb{B}^\star(p)).\]
\end{proof}

\begin{lemma}\label{well-posedness3}
The operator $(\mathcal{S}+\mathcal{T}):\bX \rightarrow \bX^\star$ is injective.
\end{lemma}
\begin{proof}
It is sufficient to show that the only solution to the homogeneous problem
\begin{align*}
	a(\vomega,\vtheta)+b_1(\vtheta,\bu)-b_2(\vtheta,p)&= \; 0\qquad \forall\,\vtheta\in\mathbf{H},\\
	b_1(\vomega,\bv) &= \;0\qquad \forall\, \bv \in\bV,\\
	b_2(\vomega,q)  -   c(p ,q)  &=\;0\qquad \forall \,q\in\rQ,
\end{align*}
is the null-vector in  $\bX$. Thus, from Lemma \ref{stability},
and the fact that $F=G=0$, we have $\bu=\cero$, $\vomega=\cero$, $p=0$. 
\end{proof}

By virtue of Lemmas \ref{stability}, \ref{well-posedness1}, \ref{well-posedness2}, and \ref{well-posedness3}, and the abstract Fredholm alternative theorem, one straightforwardly derives the main result of this section, stated in the upcoming theorem.

\begin{theorem}\label{wellposedness-continuous}
There exists a unique solution $(\vomega,\bu,p )\in\mathbf{H}\times\bV\times\rQ$, where $\vomega=(\bomega,\phi)$,
to \eqref{weak-uP}-\eqref{weak-phi}. Furthermore, there exists a positive constant $C>0$, such that
\begin{align*}
\left\VERT(\bu,\bomega,\phi,p)\right\VERT  \leq C  \big\{(\mu^{\mathrm{P}})^{-1/2}\|\ff^{\mathrm{P}}\|_{0, \Omega}+\|(\kappa/\xi)^{1/2}\rho\gg\|_{0, \Omega} +\rho_1^{1/2}\|s^{\mathrm{P}}\|_{0, \Omega} \big\},
\end{align*}
where $\rho_1=\min\left(\left(c_0+\frac{\alpha^2}{(2\mu^{\mathrm{P}}+\lambda^{\mathrm{P}})}\right)^{-1},(\kappa/\xi)^{-1}\right)$.
\end{theorem}

\section{\textit{A priori} error analysis for rotation-based poroelasticity}\label{sec:poroelast-FE}
%
Denoting $\bW_h\times \rZ_h:=\mathbf{H}_h$, 
a Galerkin 
scheme for \eqref{weak-uP}-\eqref{weak-phi} is: find $(\vomega_h,\bu_h,p_h):=((\bomega_h,\phi_h),\bu_h,p_h) \in \mathbf{H}_h\times\bV_h\times\rQ_h$ 
such that
\begin{align}
	a(\vomega_h,\vtheta_h)+b_1(\vtheta_h,\bu_h)-b_2(\vtheta_h,p_h)&= \; 0& \forall\,\vtheta_h:=(\btheta_h,\psi_h)\in\mathbf{H}_h, \label{weak-uP-G}\\
	b_1(\vomega_h,\bv_h) &= \;F(\bv_h) &\forall\, \bv_h \in\bV_h, \label{weak-pP-G}\\
	b_3(\vomega_h,q_h)  -   c(p_h ,q_h)  &=\;G(q_h ) &\forall \,q_h\in\rQ_h. \label{weak-phi-G}
\end{align}

\subsection{Stability of the discrete problem}

All bilinear forms and functionals introduced in
Section~\ref{sec:poroelast} preserve stability 
on the discrete spaces. Also,  
 $a(\cdot,\cdot)$, and $c(\cdot,\cdot)$ 
maintain  coercivity  on  $\mathbf{H}_h$
and $\rQ_h$, respectively. 
Such stability properties permit us to establish the well-posedness of  \eqref{weak-uP-G}-\eqref{weak-phi-G}.
\begin{theorem}\label{wellposedness-h}
There exists a unique solution $(\vomega_h,\bu_h,p_h) \in \mathbf{H}_h\times\bV_h\times\rQ_h$, where $\vomega_h=(\bomega_h,\phi_h)$, to \eqref{weak-uP-G}-\eqref{weak-phi-G}. Furthermore, there exists a positive constant $C_{\mathrm{Stab}}>0$, independent of $h$, $\mu^{\mathrm{P}},\lambda^{\mathrm{P}}$, such that
\begin{align*}
\left\VERT(\bu_h,\bomega_h,\phi_h,p_h)\right\VERT  \leq C  \big\{(\mu^{\mathrm{P}})^{-1/2}\|\ff^{\mathrm{P}}\|_{0, \Omega}+\|(\kappa/\xi)^{1/2}\rho\gg\|_{0, \Omega} +\rho_1^{1/2}\|s^{\mathrm{P}}\|_{0, \Omega} \big\},
\end{align*}
where $\rho_1=\min\bigl((c_0+\frac{\alpha^2}{(2\mu^{\mathrm{P}}+\lambda^{\mathrm{P}})})^{-1},(\kappa/\xi)^{-1}\bigr)$.
\end{theorem}
\begin{proof}
It follows as in the proof of Lemmas \ref{stability} and \ref{well-posedness3}. 
\end{proof}

\subsection{\textit{A priori} error bounds}

Approximation  properties  of  the  spaces  in  \eqref{fe-spaces-porous} (see, e.g.,  \cite{brezzi_book91})
produce the following theoretical rate of convergence:
 \[
 \left\VERT(\bu-\tilde{\bu},\bomega-\tilde{\bomega},\phi-\tilde{\phi},p-\tilde{p})\right\VERT\le C\,h^{\min\{s,k+1\}} (\|\bomega\|_{s, \Omega}+\sqrt{\mu^{\mathrm{P}}}\|\bu\|_{s+1, \Omega}+\rho_{\phi}\|\phi\|_{s, \Omega}+\rho_p\|p\|_{s+1,\Omega}),
 \]
 where $\rho_{\phi}=\sqrt{1/\mu^{\mathrm{P}}}+\sqrt{1/(2\mu^{\mathrm{P}}+\lambda^{\mathrm{P}})}$, $\rho_{p}=\max\bigl((c_0+\frac{\alpha^2}{(2\mu^{\mathrm{P}}+\lambda^{\mathrm{P}})})^{1/2},(\kappa/\xi)^{1/2}\bigr)$,  and $C>0$.
\begin{theorem}
In  addition  to  the  hypotheses of Theorems
$\mathrm{\ref{wellposedness-continuous}}$ and $\mathrm{\ref{wellposedness-h}}$, assume  that  there exists $s >0$ such  that
$\bomega \in \mathbf{H}^s(\Omega)$, $\bu\in \mathbf{H}^{1+s}(\Omega)$, $\phi \in \mathrm{H}^s(\Omega)$, $p\in \mathrm{H}^{1+s}(\Omega)$. Then,  there  exists  $C_{\mathrm{conv}} > 0$, independent of $h$ and $\lambda^{\mathrm{P}}$, such that with the discrete spaces $\mathrm{(\ref{fe-spaces-porous})}$, there holds
\[
	\left\VERT(\bu-\bu_h,\bomega-\bomega_h,\phi-\phi_h,p-p_h)\right\VERT  \leq C_{\mathrm{conv}}\,h^{\min\{s,k+1\}} (\|\bomega\|_{s, \Omega}+\sqrt{\mu^{\mathrm{P}}}\|\bu\|_{s+1, \Omega}+\rho_{\phi}\|\phi\|_{s, \Omega}+\rho_p\|p\|_{s+1,\Omega}).
\]
\end{theorem}
\begin{proof}
Using triangle inequality we can split the error into two parts
\begin{align*}
\left\VERT(\bu-\bu_h,\bomega-\bomega_h,\phi-\phi_h,p-p_h)\right\VERT \le \left\VERT(\bu-\tilde{\bu},\bomega-\tilde{\bomega},\phi-\tilde{\phi},p-\tilde{p})\right\VERT +\left\VERT(\tilde{\bu}-\bu_h,\tilde{\bomega}-\bomega_h,\tilde{\phi}-\phi_h,\tilde{p}-p_h)\right\VERT. 
\end{align*}
Then we can estimate the first term thanks to approximation results. To estimate the second term, we use the stability result given in Theorem \ref{stab-poro}, then
\begin{align*}
C_1\left\VERT(\tilde{\bu}-\bu_h,\tilde{\bomega}-\bomega_h,\tilde{\phi}-\phi_h,\tilde{p}-p_h)\right\VERT^2&\le B_{\mathrm{P}}(\tilde{\bu}-\bu_h,\tilde{\bomega}-\bomega_h,\tilde{\phi}-\phi_h,\tilde{p}-p_h),(\bv,\btheta,\psi,q))\\
&\le B_{\mathrm{P}}(\tilde{\bu}-\bu,\tilde{\bomega}-\bomega,\tilde{\phi}-\phi,\tilde{p}-p),(\bv,\btheta,\psi,q)),
\end{align*}
with $\left\VERT(\bv,\btheta,{\psi},q)\right\VERT\le C_2 \left\VERT(\tilde{\bu}-\bu_h,\tilde{\bomega}-\bomega_h,\tilde{\phi}-\phi_h,\tilde{p}-p_h)\right\VERT$. We can then invoke the 
continuity results to get 
\begin{align*}
C_1\left\VERT(\tilde{\bu}-\bu_h,\tilde{\bomega}-\bomega_h,\tilde{\phi}-\phi_h,\tilde{p}-p_h)\right\VERT&\le C_2 \left\VERT(\tilde{\bu}-\bu,\tilde{\bomega}-\bomega,\tilde{\phi}-\phi,\tilde{p}-p)\right\VERT.
\end{align*}
\end{proof}
\end{appendix}

\end{document}